\documentclass{amsart}
\usepackage[table]{xcolor}
\usepackage{
  amsmath,
  amssymb,
  amsthm,
  hyperref,
  paralist,
  tikz,
  thmtools,
  multicol,
  enumitem,
  longtable,
  mathrsfs,
  microtype
}

\usepackage[margin=3.25cm]{geometry}

\usepackage[charter]{mathdesign}
\usepackage[mathcal]{eucal}

\usetikzlibrary{arrows}
\usetikzlibrary{patterns}

\newtheorem{theorem}{Theorem}[section]
\newtheorem{lemma}[theorem]{Lemma}

\theoremstyle{definition}
\newtheorem{definition}[theorem]{Definition}
\newtheorem{example}[theorem]{Example}

\theoremstyle{remark}
\newtheorem{remark}[theorem]{Remark}

\numberwithin{equation}{section}

\ifcsname theorem\endcsname{}\else\declaretheorem[parent=section]{theorem}\fi
\ifcsname corollary\endcsname{}\else\declaretheorem[sibling=theorem]{corollary}\fi
\ifcsname lemma\endcsname{}\else\declaretheorem[sibling=theorem]{lemma}\fi
\ifcsname proposition\endcsname{}\else\declaretheorem[sibling=theorem]{proposition}\fi
\ifcsname conjecture\endcsname{}\else\fi
\ifcsname problem\endcsname{}\else\fi
\ifcsname question\endcsname{}\else\fi
\ifcsname definition\endcsname{}\else\declaretheorem[sibling=theorem, style=definition]{definition}\fi
\ifcsname exercise\endcsname{}\else\fi
\ifcsname example\endcsname{}\else\fi
\ifcsname remark\endcsname{}\else\declaretheorem[sibling=theorem, style=remark]{remark}\fi

\providecommand {\N}{{\bf N}}
\providecommand {\Z}{{\bf Z}}

\providecommand {\C}{{\bf C}}

\renewcommand {\P}{{\bf P}}

\providecommand {\from}{{\colon}}

\providecommand{\spec}{\operatorname{Spec}}
\providecommand{\proj}{\operatorname{Proj}}

\providecommand{\Aut}{\operatorname{Aut}}
\providecommand{\codim}{\operatorname{codim}}

\tikzset{every label/.style={font=\scriptsize}}
\tikzset{every loop/.style={}}
\tikzstyle{curve} = [circle, draw, thick, inner sep=1.5pt, font=\scriptsize]
\tikzstyle{int} = [fill=white, inner sep=1pt, font=\scriptsize]

\DeclareMathOperator{\sym}{{\mathfrak S}}
\DeclareMathOperator{\F}{{\mathbf F}}
\renewcommand{\N}{{\mathbf N}}
\DeclareMathOperator{\br}{br}
\DeclareMathOperator{\Bl}{Bl}
\DeclareMathOperator{\Def}{Def}
\DeclareMathOperator{\BrCov}{BrCov}
\DeclareMathOperator{\Sect}{Sect}
\DeclareMathOperator{\Maps}{K}
\DeclareMathOperator{\Norm}{Norm}
\DeclareMathOperator{\id}{id}
\DeclareMathOperator{\Schemes}{{\underline{Schemes}}}
\DeclareMathOperator{\Stab}{Stab}
\newcommand{\orb}[1]{\mathcal{#1}}
\newcommand{\stack}[1]{\mathscr{#1}}
\newcommand{\coarse}[1]{{|#1|}}

\renewcommand{\O}{{\mathcal O}}

\begin{document}

\title[Stacky curves and plane quintics]{Covers of stacky curves and limits of plane quintics}
\author{Anand Deopurkar}
\email{anand.deopurkar@anu.edu.au}
\address{Mathematical Sciences Institute, The Australian National University, Canberra, ACT 2601, Australia}

\subjclass{14H10, 14H45, 14H50}

\maketitle

\begin{abstract}
  We construct a well-behaved compactification of the space of finite covers of a stacky curve using admissible cover degenerations.
  Using our construction, we compactify the space of tetragonal curves on Hirzebruch surfaces.
  As an application, we explicitly describe the boundary divisors of the closure in $\overline{\orb M}_6$ of the locus of smooth plane quintic curves.
\end{abstract}

\section{Introduction}

One of the great successes of moduli theory is the construction of an explicit and well-behaved compactification of the moduli space of curves.
Constructing similarly well-behaved compactifications of spaces of higher dimensional objects remains a challenge.
Towards this goal, we can first consider the case of objects fibered over curves.
In a landmark paper \cite{abr.vis:02}, Abramovich and Vistoli constructed a compact moduli space of such objects by viewing a fibration $S \to P$ as a map from $P$ to the moduli stack that classifies the fibers.
They constructed a compact moduli space of maps from curves to a proper Deligne--Mumford stack, analogous to the space of maps from curves to a projective scheme, due to Kontsevich.

The idea of compactifying the moduli of fibered objects as maps suffers from one defect.
The resulting compactifications are often ``too big.''
They contain an open subset that parametrizes nice objects of geometric interest, but the closure of this set is typically not the whole compactification.
There are often irreducible components whose general members are degenerate, and usually it is difficult to isolate the ``good'' components.

\subsection{Results}
The first result of this paper gives a good compactification of the space of maps to a smooth, one-dimensional stack, inspired by a similar compactification in the non-stacky setting due to Harris--Mumford \cite{har.mum:82}  and Abramovich--Corti--Vistoli \cite{abr.cor.vis:03}.
\begin{theorem}[\autoref{prop:technical_main}, roughly stated]\label{thm:technical_main}
  Let $\stack{X}$ be a proper, one dimensional, smooth Deligne--Mumford stack of finite type over $\C$.
  The space of finite coverings of $\stack X$ admits a compactification that parametrizes stacky admissible covers.
  This compactification is smooth, has a normal crossings boundary, and admits a morphism to the Abramovich--Vistoli space of maps to $\stack X$.
\end{theorem}

The heart of the paper is an application to the following classical question: which stable curves of genus $g = {{d-1} \choose 2}$ are limits of smooth plane curves of degree $d$?
The question is vacuous for $d = 1$ and $d = 2$.
For $d = 3$, we know that all smooth genus 1 curves arise as plane cubics.
For $d = 4$, we recall that all non-hyperelliptic smooth genus 3 curves arise as plane quartics by their canonical embedding; the closure of such curves is all of $\overline{\orb M}_3$.
The first non-trivial case is the case of quintics.
\begin{theorem}\label{thm:quintic_boundary}
  Let $Q \subset {\orb M}_6$ be the locus of plane quintic curves and $\overline Q$ its closure in $\overline {\orb M}_6$.
  The boundary $\overline Q \cap \left(\overline {\orb M}_6 \setminus {\orb M}_6\right)$ of $Q$ is the union of $13$ irreducible divisorial components, which we describe explicitly.
\end{theorem}

The following is a description of the generic points of the boundary divisors of $Q$, grouped by their dual graphs.
The label on each vertex denotes the normalization of the irreducible component corresponding to the vertex.
For example, (2) represents an irreducible nodal curve whose normalization is hyperelliptic of genus 5.

\begin{itemize}
\item With the dual graph
  \tikz[baseline={(0,-0.1)}]{
    \draw 
    (0,0) node (X) [curve, label={180:$X$}] {} 
    (X) edge [loop right] (X);
  }
  \begin{enumerate}
  \item
    \label{main-1}
    A nodal plane quintic.
  \item 
    \label{main-2}
    $X$ hyperelliptic of genus $5$.
  \end{enumerate}
  
\item With the dual graph
  \tikz[baseline={(0,-0.1)}]{
    \draw 
    (0,0) node (X) [curve, label={180:$X$}] {} 
    (1,0) node (Y) [curve, label={0:$Y$}] {}
    (X) edge node[int]{$p$} (Y);
  }
  \begin{enumerate}[resume]
  \item
    \label{main-3}
    $(X,p)$ the normalization of a cuspidal plane quintic, and $Y$ of genus 1.
  \item 
    \label{main-4}
    $X$ of genus 2, $Y$ Maroni special of genus $4$, $p \in X$ a Weierstrass point, and $p \in Y$ a ramification point of the unique degree 3 map $Y \to \P^1$.
  \item
    \label{main-5}
    $X$ a plane quartic, $Y$ hyperelliptic of genus 3, $p \in X$ a point on a bitangent, and $p \in Y$ a Weierstrass point.
  \item 
    \label{main-6}
    $X$ a plane quartic, $Y$ hyperelliptic of genus 3, and $p \in X$ a hyperflex ($K_X = 4p$).
  \item 
    \label{main-7}
    $X$ hyperelliptic of genus $4$, $Y$ of genus 2, and $p \in X$ a Weierstrass point.
  \item 
    \label{main-8}
    $X$ of genus 1, and $Y$ hyperelliptic of genus 5.
  \end{enumerate}
  
\item With the dual graph
  \tikz[baseline={(0,-0.1)}]{
    \draw 
    (0,0) node (X) [curve, label={180:$X$}] {}
    (1,0) node (Y) [curve, label={0:$Y$}] {}
    (X) edge [bend right=45]node[int] {$q$} (Y)
    (X) edge [bend left=45] node[int] {$p$} (Y);
  } 
  
  \begin{enumerate}[resume]
  \item 
    \label{main-9}
    $X$ Maroni special of genus $4$, $Y$ of genus $1$, and $p, q \in X$ on a fiber of the unique degree 3 map $X \to \P^1$.
  \item 
    \label{main-10}
    $X$ hyperelliptic of genus 3, $Y$ of genus $2$, and $p \in Y$ a Weierstrass point.
  \item 
    \label{main-11}
    $X$ of genus $2$, $Y$ a plane quartic, $p, q \in X$ hyperelliptic conjugate, and the line through $p$, $q$ tangent to $Y$ at a third point.
  \item 
    \label{main-12}
    $X$ hyperelliptic of genus $3$, $Y$ of genus $2$, and $p, q \in X$ hyperelliptic conjugate.
  \end{enumerate}

\item With the dual graph
  \tikz[baseline={(0,-0.1)}]{
    \draw 
    (0,0) node (X) [curve, label={180:$X$}] {}
    (1,0) node (Y) [curve, label={0:$Y$}] {}
    (X) edge [bend right=45] (Y)
    (X) edge (Y)
    (X) edge [bend left=45] (Y);
  }
  \begin{enumerate}[resume]
  \item 
    \label{main-13}
    $X$ hyperelliptic of genus $3$, and $Y$ of genus $1$.
  \end{enumerate}
\end{itemize}
The closure of $Q$ in $\orb M_6$ is known to be the union of $Q$ with the locus of hyperelliptic curves \cite{gri:85}.
This result also follows from our techniques.

\subsection{Connection between plane quintics and stacky covers}
Let \[\stack X = [\overline M_{0,4}/\sym_4].\]
Consider a cover (representable, finite, flat morphism) $\phi \from \mathcal P \to \stack X$, where $\mathcal P$ is a smooth orbifold curve.
Denote the coarse space of $\mathcal P$ by $P$.
Assuming $\phi$ has a suitable ramification profile over the boundary point of $\stack X$, the moduli description of $\stack X$ implies that the map $\phi$ is equivalent to a pair $(S \to P, C)$, where $S \to P$ is a $\P^1$-bundle and $C \subset S$ is a curve of relative degree 4.
If the genus of $P$ is $0$, then we can view the space of covers $\phi$ as the space of 4-gonal curves on Hirzebruch surfaces.
The boundary points of the compactification given by \autoref{thm:technical_main} correspond to nodal curves on certain (reducible) degenerations of Hirzebruch surfaces.
Since the curves appearing at the boundary are at worst nodal, we have a forgetful map to $\overline {\orb M}_g$ (where the $g$ is related to the degree of $\phi$).

The moduli space of covers $\phi$ turns out to have two connected components, distinguished by the parity of the $\P^1$-bundle $S \to P$.
Taking the odd component for $g = 6$ yields a compactification of 4-gonal genus 6 curves on $\F_1$.
Under the blow-down $\F_1 \to \P^2$, we see that such curves are precisely the plane quintics.
The image in $\overline {\orb M}_6$ of the odd component of our moduli space gives the closure of the set of plane quintics.

\subsection{Relationship with previous work}
The problem of finding limits of plane curves has led to a search for good compactifications of pairs $(S, C)$ where $S$ is a surface and $C$ is a curve on $S$.
Hassett has described such a compactification for $S = \P^2$ and $C$ of degree 4 \cite{has:99}.
Hacking has described such a compactification in a weighted setting for $S = \P^2$ and $C$ of arbitrary degree  \cite{hac:04}.
In Hacking's compactification, the curve $C$ acquires worse than nodal singularities, due to the choice of weighting.
The idea behind Hassett and Hacking's construction is to view $(S, C)$ as a (weighted) log pair and construct a compactification following Koll\'ar--Shepherd-Barron \cite{kol.she:88} and Alexeev \cite{ale:96}.
It will be interesting to compare the compactifications of $(S, C)$ obtained using our approach with the one obtained using the Koll\'ar--Shepherd-Barron--Alexeev approach.

By definition, the curves in $\overline Q$ are stable reductions of singular quintic plane curves.
Our method, however, does not identify the singular quintics whose stable reductions yield the limiting curves.
Nevertheless, we can use Hassett's methods \cite{has:00} to find such singular quintics \emph{a posteriori} (see \autoref{table:local_stable_reduction}).
Note that \autoref{table:local_stable_reduction} confirms that the curves listed in \autoref{thm:quintic_boundary} arise as limits of plane quintics, but it does not suffice to conclude that the list is exhaustive.
Also, the correspondence in \autoref{table:local_stable_reduction} is not necessarily complete---there may be other singular quintics that yield the same stable limits.
Observe that all the boundary divisors are obtained from the stable reduction of $A$ and $D$ singularities.
Is there an \emph{a priori} reason?

\begin{table}
  \centering
  \rowcolors{2}{gray!15}{white}
  \begin{tabular}{p{.28\textwidth} p{.65\textwidth}}
    \hline
    \rowcolor{gray!25}
    Divisor in \autoref{thm:quintic_boundary} & Singular plane curve\\
     \hline
     \eqref{main-1}, \eqref{main-9}, \eqref{main-11}, \eqref{main-12} & An irreducible plane curve with an $A_{2n+1}$ singularity, $n = 0, 1, 2, 3$\\
     \eqref{main-3}, \eqref{main-4}, \eqref{main-5}, \eqref{main-7} & An irreducible plane curve with an $A_{2n}$ singularity,  $n = 1, 2, 3, 4$\\
     \eqref{main-6} & The union of a smooth plane quartic and a hyperflex line ($A_7$ singularity)\\
     \eqref{main-8} & The union of a smooth conic and a 6-fold tangent smooth cubic ($A_{11}$ singularity)\\
     \eqref{main-10} & An irreducible plane curve with a $D_5$ singularity\\
     \eqref{main-13} & An irreducible plane curve with a $D_4$ singularity\\
     \eqref{main-2} & The union of a nodal cubic and a smooth conic with 5-fold tangency along a nodal branch ($D_{12}$ singularity).
  \end{tabular}
  \caption{Singular plane curves that yield the divisors in \autoref{thm:quintic_boundary}}
  \label{table:local_stable_reduction}
\end{table}

As mentioned before, for $\stack X = [\overline M_{0,4}/\sym_4]$ the space of maps $\phi \from \orb P \to \stack X$ splits into two connected components, corresponding to the parity of the resulting Hirzebruch surface.
There is an alternate explanation of this parity, given by Vakil \cite{vak:01} (though not in the language of stacky covers).
The action of $\sym_4$ on $\overline M_{0,4}$ has as a kernel the Klein four group
\[K = \{\id, (12)(34), (13)(24), (14)(23)\} \subset \sym_4.\]
The quotient $\sym_4/K$ is isomorphic to $\sym_3$.
Let $\psi \from \orb P \to [\overline M_{0,4}/\sym_3]$ be the composition of $\phi$ with the natural map $[\overline M_{0,4}/\sym_4] \to [\overline M_{0,4}/\sym_3]$.

The moduli interpretation of $\psi$ gives rise to a trigonal curve $D$, and $\phi$ gives rise to a theta characteristic on $D$ (see \autoref{sec:parity}).
The parity is manifested as the parity of this theta characteristic.
The procedure of obtaining the trigonal curve $D$ from the tetragonal curve $C$ is the classical construction of the cubic resolvant, studied geometrically by Recillas \cite{rec:73}.

\subsection{More applications}
In addition to $\stack X = [\overline M_{0,4}/\sym_4]$, there are many other one-dimensional moduli stacks $\stack X$ for which \autoref{thm:technical_main} will yield nice compactifications of interesting objects.
For example, taking $\stack X = \overline {\stack M}_{1,1}$ will give a nice compactification of the space of elliptic fibrations.
Similarly, taking $\stack X$ to be the moduli stack of $\Lambda$-polarized K3 surfaces, where $\Lambda$ is a lattice of rank 19, will give a nice compactification of threefolds fibered in K3 surfaces.
Such fibrations play a key role in mirror symmetry (see, for example, \cite{dor.har.nov.ea:15}).

\subsection{Outline of the paper}
\autoref{sec:technical} contains the construction of the admissible cover compactification of covers of $\stack X$.
\autoref{sec:hirz} discusses the case of $\stack X = [\overline M_{0,4}/\sym_4]$.
\autoref{sec:quintics} specializes to the case of plane quintics.
A large part of the paper is concerned with deciphering the geometry of orbifold curves on orbifold surfaces in classical terms.
\autoref{sec:orbiscrolls} describes the geometry of $\P^1$ bundles over orbifold curves that we need for this purpose.
\autoref{thm:quintic_boundary} follows from combining \autoref{prop:not_contracted_1}, \autoref{thm:1-5}, \autoref{thm:6}, \autoref{thm:7}, and \autoref{thm:8}.

\subsection{Notation and conventions}
We work over the complex numbers $\C$.
A \emph{stack} means a Deligne--Mumford stack.
An \emph{orbifold} means a Deligne--Mumford stack without generic stabilizers.
Orbifolds are usually denoted by curly letters ($\orb X, \orb Y$), and stacks with non-trivial generic stabilizers by curlier letters ($\stack X, \stack Y$).
Coarse spaces are denoted by the absolute value sign ($\coarse{\orb X}, \coarse{\stack Y}$), or by the corresponding roman letter ($X$, $Y$) if no confusion is likely.
A \emph{curve} is a proper, reduced, connected, one-dimensional scheme/orbifold/stack, of finite type over $\C$.
The projectivization of a vector bundle is the space of its one dimensional quotients.
Given an object $X$ over $S$ and a morphism $T \to S$, the notation $X_T$ denotes the base-change $X \times_S T$.

\subsubsection*{Acknowledgments}
Thanks to Anand Patel and Ravi Vakil for the stimulating conversations that spawned this project.
Thanks to Alessandro Chiodo and Marco Pacini for sharing their expertise on orbifold curves and theta characteristics.
Thanks to Johan de Jong and Brendan Hassett for their help and encouragement.
Thanks to Changho Han and Aaron Landesman for correcting some errors in an earlier draft.

\section{Moduli of branched covers of a stacky curve}\label{sec:technical}
In this section, we construct the admissible cover compactification of covers of a stacky curve in three steps.
First, we construct the space of branched covers of a family of orbifold curves with a given branch divisor (\autoref{sec:orbifold_covers}).
Second, we take a fixed orbifold curve and construct the space of branch divisors on it and its degenerations (\autoref{sec:fulton_macpherson}).
Third, we combine these results and accommodate generic stabilizers to arrive at the main construction (\autoref{sec:combine}).

\subsection{Covers of a family of orbifold curves with a given branch divisor}\label{sec:orbifold_covers}
Let $S$ be a scheme of finite type over $\C$ and
$\pi \from {\orb X} \to S$ a (balanced) twisted curve as in \cite{abr.vis:02}.
For the convenience of the reader, we recall the definition.
\begin{definition}[Twisted curve]
  A balanced twisted curve is a Deligne--Mumford stack $\orb X$, isomorphic to its coarse space $X$ except at finitely many points.
  The stack structure at these finitely many points is of the following form:
\begin{asparaenum}[]
\item At a node:
  \[ [\spec \left(\C[x,y]/xy\right) / \mu_r ] \text{ where  $\zeta \in \mu_r$ acts by $\zeta \from (x,y) \mapsto (\zeta x, \zeta^{-1}y)$}\]
\item At a smooth point:
  \[ [\spec \C[x]/ \mu_r] \text{ where $\zeta \in \mu_r$ acts by $x \mapsto \zeta x$}.\]
\end{asparaenum}
\end{definition}
A family of twisted curves over $S$ is defined as expected (see \cite[\S~2.1]{abr.cor.vis:03}).
Since all our twisted curves will be balanced, we drop this adjective from now on.
We call the integer $r$ the \emph{order} of the corresponding orbifold point.

\begin{definition}[$\BrCov_d({\orb X}/S, \Sigma)$]
  Let $\Sigma \subset {\orb X}$ be a divisor that is \'etale over $S$ and lies in the smooth and representable locus of $\pi$.
  Define $\BrCov_d({\orb X}/S, \Sigma)$ as the category fibered in groupoids over $\Schemes_S$ whose objects over $T \to S$ are 
  \[(p \from \orb P \to T, \phi \from \orb P \to {\orb X}_T),\]
  where $p$ is a twisted curve over $T$ and $\phi$ is representable, flat, and finite with branch divisor $\br\phi = \Sigma_T$.
\end{definition}

Consider a point of $\BrCov_d({\orb X}/S, \Sigma)$, say $\phi \from \orb P \to {\orb X}_t$ over a point $t$ of $S$.
Then ${\orb P}$ is also a twisted curve with nodes over the nodes of ${\orb X}_t$.
Since $\phi$ is representable, the orbifold points of ${\orb P}$ are only over the orbifold points of ${\orb X}_t$.

We can associate some numerical invariants to $\phi$ which will remain constant in families.
First, we have global invariants such as the number of connected components of ${\orb P}$, their arithmetic genera, and the degree of $\phi$ on them.
Second, for every smooth orbifold point $x \in {\orb X}_t$ we have the local invariant of the cover $\phi \from \orb P \to {\orb X}_t$ given by its monodromy around $x$.
The monodromy is given by the action of the cyclic group $\Aut_x {\orb X}$ on the $d$ element set $\phi^{-1}(x)$.
This data is equivalent to the data of the ramification indices over $x$ of the map $\coarse{\phi} \from \coarse {\orb P} \to \coarse{\orb X_t}$ between the coarse spaces.
Notice that the monodromy data at $x$ also determines the number and the orders of the orbifold points of ${\orb P}$ over $x$.
The moduli problem $\BrCov_d({\orb X}/S, \Sigma)$ is thus a disjoint union of the moduli problems with fixed numerical invariants.
Note, however, that having fixed the degree $d$, the curve ${\orb X}/S$, and the divisor $\Sigma$, the set of possible numerical invariants is finite.

\begin{proposition}\label{prop:brcov_sigma}
  $\BrCov_d({\orb X}/S, \Sigma) \to S$ is a separated \'etale Deligne--Mumford stack of finite type.
  If the orders of all the orbinodes of all the fibers of ${\orb X} \to S$ are divisible by the orders of all the permutations in the symmetric group $\sym_d$, then $\BrCov_d({\orb X}/S, \Sigma) \to S$ is also proper.
\end{proposition}
\begin{proof}
  Fix non-negative integers $g$ and $n$.
  Consider the category fibered in groupoids over $\Schemes_S$ whose objects over $T \to S$ are $(p \from {\orb P} \to T, \phi \from {\orb P} \to {\orb X})$ where ${\orb P} \to T$ is a twisted curve of genus $g$ with $n$ smooth orbifold points and $\phi$ is a twisted stable map with $\phi_*[{\orb P}_t] = d[{\orb X}_t]$.
  This is simply the stack of twisted stable maps to ${\orb X}$ of Abramovich--Vistoli.
  By \cite[Theorem~1.4.1]{abr.vis:02}, this is a proper Deligne--Mumford stack of finite type over $S$.
  The conditions that $\phi \from {\orb P} \to {\orb X}_T$ be finite and unramified away from $\Sigma$ are open conditions.
  For $\phi \from {\orb P} \to {\orb X}_T$ finite and unramified away from $\Sigma$, the condition that $\br\phi = \Sigma_T$ is open and closed.
  Thus, for fixed $g$ and $n$ the stack $\BrCov_d^{g,n}({\orb X}/S, \Sigma)$ is a separated Deligne--Mumford stack of finite type over $S$.
  But there are only finitely many choices for $g$ and $n$, so $\BrCov_d({\orb X}/S, \Sigma)$ is a separated Deligne--Mumford stack of finite type over $S$.

  To see that $\BrCov_d({\orb X}/S, \Sigma) \to S$ is \'etale, consider a point $t \to S$ and a point of $\BrCov_d({\orb X}/S, \Sigma)$ over it, say $\phi \from {\orb P} \to {\orb X}_t$.
  Since $\phi$ is a finite flat morphism of curves with a reduced branch divisor $\Sigma_t$ lying in the smooth and representable locus of ${\orb X}_t$, the map on deformation spaces
  \[\Def_\phi \to \Def({\orb X}_t, \Sigma_t)\]
  is an isomorphism.
  Indeed, we have an isomorphism of cotangent complexes $\phi_* L_{\phi} \to L_{\Sigma_t \to \orb X_t}$ (this complex is equivalent to a skyscraper sheaf supported on $\Sigma_t$).
  We conclude that $\BrCov_d({\orb X}/S, \Sigma) \to S$ is \'etale.

  Finally, assume that the orders of all the orbinodes of the fibers of ${\orb X} \to S$ are sufficiently divisible as required.
  Let us check the valuative criterion for properness.
  For this, we take $S$ to be a DVR $\Delta$ with special point $0$ and generic point $\eta$.
  Let $\phi \from {\orb P}_\eta \to {\orb X}_\eta$ be a finite cover of degree $d$ with branch divisor $\Sigma_{\eta}$.
  We want to show that $\phi$ extends to a finite cover $\phi \from \orb P \to \orb X$ with branch divisor $\Sigma$, possibly after a finite base change on $\Delta$.
  The proof follows \cite[Section 6]{abr.vis:02}.

  Let $x$ be the generic point of a component of ${\orb X}_0$.
  By Abhyankar's lemma, $\phi$ extends to a finite \'etale cover over $x$, possibly after a finite base change on $\Delta$.
  We then have an extension of $\phi$ on all of ${\orb X}$ except over finitely many points of ${\orb X}_0$.

  Let $x \in {\orb X}_0$ be a smooth point.
  Recall that every finite flat cover of a punctured smooth surface extends to a finite flat cover of the surface.
  Indeed, the data of a finite flat cover consists of the data of a vector bundle along with the data of an algebra structure on the vector bundle.
  A vector bundle on a punctured smooth surface extends to a vector bundle on the surface by  \cite{hor:64}.
  The maps defining the algebra structure extend by Hartog's theorem.
  Therefore, we get an extension of $\phi$ over $x$.

  Let $x \in {\orb X}_0$ be a limit of a node in the generic fiber.
  Then $\orb X$ is locally simply connected at $x$.
  (That is, $V \setminus \{x\}$ is simply connected for a sufficiently small \'etale chart $V \to \orb X$ around $x$.)
  In this case, $\phi$ trivially extends to an \'etale cover locally over $x$.

  Let $x \in {\orb X}_0$ be a node that is not a limit of a node in the generic fiber.
  Then ${\orb X}$ has the form
  \[U = [\spec \C[x,y,t]/(xy-t^n)/\mu_r]\]
  near $x$ where $r$ is sufficiently divisible.
  In this case, $\phi$ extends to an \'etale cover over $U$ by \autoref{prop:deep_orbinode}, where we interpret an \'etale cover of degree $d$ as a map to $\stack Y = B\sym_d$.

  We thus have the required extension $\phi \from \orb P \to \orb X$.
  The equality of divisors $\br \phi = \Sigma$ holds outside a codimension 2 locus on $\orb X$, and hence on all of $\orb X$.
  The proof of the valuative criterion is then complete.
\end{proof}
\begin{lemma}\label{prop:deep_orbinode}
  Let $\stack Y$ be a Deligne--Mumford stack and let $U$ be the orbinode
  \[ U = [\spec \C[x,y,t]/(xy-t^n) / \mu_r].\]
  Suppose $\phi \from U \dashrightarrow \stack Y$ is defined away from $0$ and the map $\coarse \phi$ on the coarse spaces extends to all of $\coarse U$.
  Suppose for every $\sigma \in \Aut_{\coarse\phi(0)} \stack Y$ we have $\sigma^r = 1$.
  Then $\phi$ extends to a morphism $\phi \from U \to \stack Y$.
\end{lemma}
\begin{proof}
  Let $G = \Aut_{\coarse \mu(0)} \stack Y$.
  We have an \'etale local presentation of $Y$ of the form $[Y/G]$ where $Y$ is a scheme.
  Since it suffices to prove the statement locally around $0$ in the \'etale topology, we may take $\stack Y = [Y/G]$.
  Consider the following tower:
  \[
  \begin{tikzpicture}
    \node (uni) {$\spec \C[u,v,t]/(uv-t) = \widetilde U$};
    \node[below of=uni] (inter) {$\spec \C[x,y,t]/(xy-t^n) = \overline U$};
    \node[below of=inter] (U) {$[\spec \C[x,y,t]/(xy-t^n) / \mu_r] = U$};
    \draw (inter)+(4,0) node (Y) {$Y$};
    \node (YG) [below of=Y] {$[Y/G]$};
    \draw[->] (uni) edge (inter) (inter) edge (U);
    \draw[->] (Y) edge (YG) (inter) edge (U);
    \draw[->, dashed] (U) edge (YG) (inter) edge (U);
  \end{tikzpicture}.
  \]
  We have an action of $\mu_{nr}$ on $\widetilde U$ where an element $\zeta \in \mu_{nr}$ acts by $\zeta \cdot (u,v,t) \mapsto (\zeta u, \zeta^{-1} v, t)$.
  The map $\widetilde U \to \overline U$ is the geometric quotient by $\mu_n = \langle \zeta^r \rangle \subset \mu_{nr}$
  and the map $\overline U \to U$ is the stack quotient by $\mu_r = \mu_{nr}/ \mu_n$.
  Since $\widetilde U$ is simply connected, we get a map $\mu_{nr} \to G$ and a lift $\widetilde U \to Y$ of $\phi$ which is equivariant with respect to the $\mu_{nr}$ action on $\widetilde U$ and the $G$ action on $Y$.
  Since $\sigma^r = 1$ for all $\sigma \in G$, the map $\widetilde U \to Y$ is equivariant with respect to the $\mu_n$ action on $\widetilde U$ and the trivial action on $Y$.
  Hence it gives a map $\overline U \to Y$ and by composition a map $\overline U \to [Y/G]$.
  Since $\overline U \to U$ is \'etale, we have extended $\phi$ to a map at $0$ \'etale locally on $U$. 
\end{proof}

\subsection{The Fulton--MacPherson configuration space}\label{sec:fulton_macpherson}
The goal of this section is to construct a compactified configuration space of $b$ distinct points on orbifold smooth curves in the style of Fulton and MacPherson. 
We first recall (a slight generalization of) the notion of a \emph{$b$-pointed degeneration} from \cite{ful.mac:94}.
\begin{definition}[Pointed degeneration]
Let $X$ be a smooth schematic curve and let $x_1, \dots, x_n$ be distinct points of $X$.
A \emph{$b$-pointed degeneration of $(X,\{x_1, \dots, x_n\})$} is the data of 
\[(\rho \from Z \to X, \{ \sigma_1, \dots, \sigma_b \}, \{\widetilde x_1, \dots, \widetilde x_n\}),\]
where
\begin{itemize}
\item $Z$ is a nodal curve and $\{\sigma_j, \widetilde x_i \}$ are $b+n$ distinct smooth points on $Z$;
\item $\rho$ maps $\widetilde x_i$ to $x_i$
\item $\rho$ is an isomorphism on one component of $Z$, called the main component and also denoted by $X$.
  The rest of the curve, namely $\overline{Z \setminus X}$, is a disjoint union of trees of smooth rational curves, each tree meeting $X$ at one point.
\end{itemize}
We say that the degeneration is \emph{stable} if each component of $Z$ contracted by $\rho$ has at least three special points (nodes or marked points).
\end{definition}

We now define a similar gadget for an orbifold curve.
\begin{definition}[Pointed degeneration of orbifold curves]
  Let ${\orb X}$ be a smooth orbifold curve with $n$ orbifold points $x_1, \dots, x_n$.
  A \emph{$b$-pointed degeneration} of ${\orb X}$ is the data of 
  \[ (\rho \from {\orb Z} \to {\orb X}, \{ \sigma_1, \dots, \sigma_b \}),\]
  where
  \begin{itemize}
  \item ${\orb Z}$ is a twisted nodal curve, schematic away from the nodes and $n$ smooth points, say $\widetilde x_1, \dots, \widetilde x_n$;
  \item $\sigma_1, \dots, \sigma_b$ are $b$ distinct points in the smooth and schematic locus of ${\orb Z}$;
  \item $\rho$ maps $\widetilde x_i$ to $x_i$ and induces an isomorphism $\rho \from \Aut_{\widetilde x_i} {\orb Z} \to \Aut_{x_i} {\orb X}$;
  \item the data on the underlying coarse spaces
    \[(\coarse \rho \from \coarse {\orb Z} \to \coarse {\orb X}, \{\sigma_1, \dots, \sigma_b\}, \{\widetilde x_1, \dots, \widetilde x_n\})\]
    is a $b$-pointed degeneration of $(\coarse {\orb X}, \{x_1, \dots, x_n\})$.
  \end{itemize}
  We say that the degeneration is \emph{stable} if the degeneration of the underlying coarse spaces is stable.
\end{definition}

Let ${\orb X}$ be an orbifold curve with $n$ orbifold points $x_1, \dots, x_n$.
Let $U \subset {\orb X}^b$ be the complement of all the diagonals and orbifold points.
Let $\pi \from {\orb X} \times U \to U$ be the second projection and $\sigma_j \from U \to {\orb X} \times U$ the section of $\pi$ corresponding to the $j$th factor, namely
\[ \sigma_j (u_1, \dots, u_b) = (u_j, u_1, \dots, u_b).\]
Let $\rho \from {\orb X} \times U \to {\orb X}$ be the first projection.
\begin{proposition}\label{prop:fm}
  There exists a smooth Deligne--Mumford stack ${\orb X}[b]$ along with a family of twisted curves $\pi \from {\orb Z} \to {\orb X}[b]$ such that the following hold:
  \begin{enumerate}
  \item\label{compactification}
    ${\orb X}[b]$ contains $U$ as a dense open substack and $\pi \from {\orb Z} \to {\orb X}[b]$ restricts over $U$ to ${\orb X} \times U \to U$.
  \item \label{snc}
    ${\orb X}[b] \setminus U$ is a divisor with simple normal crossings and the total space $\orb Z$ is smooth.
  \item\label{extension}
    The sections $\sigma_j \from U \to {\orb X} \times U$ and the map $\rho \from {\orb X} \times U \to {\orb X}$ extend to sections $\sigma_j \from {\orb X}[b] \to \orb Z$ and a map $\rho \from {\orb Z} \to {\orb X}$. 
  \item \label{fibers_are_expansions}
    For every point $t$ in ${\orb X[b]}$, the datum $(\rho \from {\orb Z}_t \to {\orb X}, \{\sigma_j\})$ is a $b$-pointed stable degeneration of ${\orb X}$.
  \item \label{deep nodes}
    We may arrange ${\orb X}[b]$ and $\pi \from {\orb Z} \to {\orb X}[b]$ so that the orders of the orbinodes of the fibers of $\pi$ are sufficiently divisible.
  \end{enumerate}
\end{proposition}
For the proof, we need a slight variation of the Fulton--MacPherson space of $b$ points on $X$.
Let $X[b;x_1, \dots ,x_n]$ be the space of $b$ points on $X$ where the points are required to remain distinct and also distinct from the $x_i$.
To ease notation, we abbreviate $X[b; x_1, \dots, x_n]$ by $X[b]$.
It is a smooth projective variety containing $U$ as a dense open subset with a normal crossings complement.
It carries a universal family of nodal curves $\pi \from Z \to X[b]$ with smooth total space $Z$ along with $b+n$ sections $\sigma_j \from X[b] \to Z$ and $\widetilde x_i \from X[b] \to Z$, and a map $\rho \from Z \to X$.
The universal family extends the constant family $X \times U \to U$; the sections $\sigma_j$ extend the tautological sections; the sections $\widetilde x_i$ extend the constant sections $x_i$; and the map $\rho$ extends the projection $X \times U \to X$.
The fibers of $Z \to X[b]$ along with the points $\sigma_j$, $\widetilde x_i$, and the map $\rho$ form a stable $b$-pointed degeneration of $(X, x_1, \dots, x_n)$.

We can construct $X[b]$ following the method of \cite{ful.mac:94}, which we recall briefly.
We start with $X^b$ and the constant family $X \times X^b \to X$.
We then successively blow up the proper transforms of the strata where the sections $\sigma_j$ coincide among themselves or coincide with the points $x_i$ to arrive at $X[b]$ and the family $Z \to X[b]$.
We summarize the features of $X[b]$ in the following diagram:
\[
  \begin{tikzpicture}
    \draw node (XU) {$X \times U$} (XU)+(2,0) node (Z) {$Z$} (Z)+(2,0) node (X) {$X$};
    \draw node[below of=XU] (U) {$U$} (U)+(2,0) node (Xb) {$X[b]$};
    \draw[right hook->]
    (XU) edge (Z) 
    (U) edge (Xb);
    \draw[->] (XU) edge (U) (Z) edge (Xb) (Z) edge node[above]{\small $\rho$} (X);
    \draw[->, bend right=30]
    (U) edge node[right] {\small $\sigma_j, x_i$} (XU) 
    (Xb) edge node[right] {\small $\sigma_j, \widetilde x_i$} (Z);
  \end{tikzpicture}.
\]

\begin{proof}[Proof of \autoref{prop:fm}]
  Let $X$ be the coarse space of ${\orb X}$.
  Let $X[b]$ be the Fulton--MacPherson space of $b$ distinct points on $X$ that remain distinct from the $x_i$, as described above.
  We now use the results of Olsson \cite{ols:07} to modify $X[b]$ and $Z \to X[b]$ in a stacky way to obtain the claimed ${\orb X}[b]$ and ${\orb Z} \to {\orb X}$.
  Our argument follows the proof of \cite[Theorem~1.9]{ols:07}.
  Fix a positive integer $d$ that is divisible by $a_i = |\Aut_{x_i}{\orb X}|$ for all $i$.
  The simple normal crossings divisor $X[b] \setminus U$ gives a canonical log structure $\mathcal M$ on $X[b]$.
  This log structure agrees with the log structure that $X[b]$ gets from the family of nodal curves $Z \to X[b]$ as explained in \cite[\S~3]{ols:07}. 
  Denote by $r$ the number of irreducible components of $X[b] \setminus U$.
  Let $\alpha$ be the vector $(d, \dots, d)$ of length $r$.
  Let ${\orb X}[b]$ be the stack $\stack F(\alpha)$ constructed in \cite[Lemma~5.3]{ols:07}.
  The defining property of $\stack F(\alpha)$ is the following: a map $T \to \stack F(\alpha)$ corresponds to a map $T \to X[b]$ along with an extension of log structures $\mathcal M_T \to \mathcal M'_T$ which locally has the form $\N^r \xrightarrow{\times d} \N^r$.
  Thus, ${\orb X}[b]$ maps to $X[b]$ and comes equipped with a tautological extension $\mathcal M \to \mathcal M'$, where we have used the same symbol $\mathcal M$ to denote the pullback to ${\orb X}[b]$ of the log structure $\mathcal M$ on $X[b]$.
  By \cite[Theorem~1.8]{ols:07}, the data $(Z \times_{X[b]} {\orb X}[b], \{x_i, a_i\}, \mathcal M \to \mathcal M')$ defines a twisted curve ${\orb Z} \to {\orb X}[b]$ with a map ${\orb Z} \to Z$ which is a purely stacky modification (an isomorphism on coarse spaces).

  Before we proceed, let us describe the modification ${\orb X}[b] \to X[b]$ and ${\orb Z} \to Z$ explicitly in local coordinates.
  Let $0 \in X[b]$ be a point such that $Z_0$ is an $l$-nodal curve.
  \'Etale locally around $0$, the pair $(X[b], X[b]\setminus U)$ is isomorphic to
  \begin{equation}\label{Xb etale}
    (\spec \C[x_1, \dots, x_b], x_1 \cdots x_l).
  \end{equation}
  \'Etale locally around a node of $Z_0$, the map $Z_0 \to X[b]$ is isomorphic to
  \begin{equation}\label{Z etale}
    \spec \C[x_1, \dots, x_b, y, z]/(yz - x_1) \to \spec \C[x_1, \dots, x_b].
  \end{equation}
  In the coordinates of \eqref{Xb etale}, the map ${\orb X}[b] \to X[b]$ is given by
  \begin{equation}\label{oXb etale}
    \begin{split}
      \left[ \spec \C[u_1, \dots, u_l, x_{l+1}, \dots, x_b] \middle/ \mu_d^l \right] &\to \spec \C[x_1, \dots, x_b] \\
      (u_1, \dots, u_l, x_{l+1}, \dots, x_b) &\mapsto (u_1^d, \dots, u_l^d, x_{l+1}, \dots, x_b).
    \end{split}
  \end{equation}
  Here $\mu_d^l$ acts by multiplication on $(u_1, \dots, u_l)$ and trivially on the $x_i$. 
  Having described ${\orb X}[b] \to X[b]$, we turn to ${\orb Z} \to Z$.
  Let 
  \[V = \spec \C[u_1, \dots, u_l, x_{l+1}, \dots, x_b]\]
  be the \'etale local chart of ${\orb X}[b]$ from \eqref{oXb etale}.
  The map ${\orb Z}_V \to Z_V$ is an isomorphism except over the points $\widetilde x_i$ and the nodes of $Z_0$.
  Around the point $\widetilde x_i$ of $Z_0$, the map ${\orb Z}_V \to Z_V$ is given by
  \begin{equation}\label{Z etale marked point}
    \begin{split}
      [\spec \O_V[s] / \mu_{a_i}] &\to \spec \O_V[t], \\
      s &\mapsto s^{a_i},
    \end{split}
  \end{equation}
  where $\zeta \in \mu_{a_i}$ acts by $\zeta \cdot s = \zeta s$.
  Around the node of $Z_0$ from \eqref{Z etale}, the map ${\orb Z}_V \to Z_V$ is given by
  \begin{equation}\label{Z etale node}
    \begin{split}
      \left[\spec \O_V[a,b]/(ab-u_1) \middle/ \mu_d\right] &\to \spec \O_V[y,z]/(yz-u^d_1)\\
      (a,b) &\mapsto (a^d, b^d),
    \end{split}
  \end{equation}
  where $\zeta \in \mu_d$ acts by $\zeta \cdot (a,b) = (\zeta a, \zeta^{-1} b)$.

  We now check that our construction has the claimed properties.
  From \eqref{oXb etale}, it follows that ${\orb X}[b] \to X[b]$ is an isomorphism over $U$.
  Therefore, ${\orb X}[b]$ contains $U$ as a dense open.
  From \eqref{oXb etale}, we also see that the complement is simple normal crossings.
  From \eqref{Z etale marked point}, we see that the map ${\orb Z}_U \to Z_U$ is the root stack of order $a_i$ at $x_i \times U$.
  Therefore, ${\orb Z}_U \to U$ is isomorphic to ${\orb X} \times U \to U$.
  From the local description, we see that $\orb Z$ is smooth.
  The sections $\sigma_j \from {X[b]} \to Z$ give sections $\sigma_j \from {\orb X}[b] \to Z \times_{X[b]}{\orb X}[b]$.
  But ${\orb Z} \to Z \times_{X[b]} {\orb X}[b]$ is an isomorphism around $\sigma_j$.
  So we get sections $\sigma_j \from {\orb X}[b] \to {\orb Z}$.
  To get $\rho \from {\orb Z} \to {\orb X}$, we start with the map $\coarse\rho \from {\orb Z} \to X$ obtained by composing ${\orb Z} \to Z$ with $ \rho \from Z \to X$.
  To lift $\coarse \rho$ to $\rho \from {\orb Z} \to {\orb X}$, we must show that the divisor ${\coarse \rho}^{-1}(x_i) \subset {\orb Z}$ is $a_i$ times a Cartier divisor.
  The divisor $\rho^{-1}(x_i) \subset Z$ consists of multiple components: a main component $\widetilde x_i(X[b])$ and several other components that lie in the exceptional locus of $Z \to X \times X[b]$.
  From \eqref{Z etale marked point}, we see that the multiplicity of the preimage in $\orb Z$ of $\widetilde x_i(X[b])$ is precisely $a_i$.
  In the coordinates of \eqref{Z etale}, the exceptional components are cut out by powers of $y$, $z$, or $x_i$.
  In any case, we see from \eqref{Z etale node} that their preimage in $\orb Z$ is divisible by $d$, which is in turn divisible by $a_i$.
  Therefore, $\coarse \rho \from {\orb Z} \to X$ lifts to $\rho \from {\orb Z} \to {\orb X}$.
  For a point $t$ in ${\orb X}[b]$, the datum $(\rho \from Z_t \to X, \{\sigma_j\}, \{\widetilde x_i\})$ is a $b$-pointed stable degeneration of $(X, \{\widetilde x_i\})$.
  From the description of $\orb Z \to Z$, it follows that $(\rho \from {\orb Z}_t \to {\orb X}, \{\sigma_j\})$ is a $b$-pointed stable degeneration of ${\orb X}$.
  Finally, we see from \eqref{Z etale node} that the order of the orbinodes of the fibers of ${\orb Z} \to {\orb X}[b]$ is $d$, which we can take to be sufficiently divisible.
\end{proof}

\subsection{Moduli of branched covers of a stacky curve}\label{sec:combine}
The goal of this section to combine the results of the previous two sections and accommodate generic stabilizers.

Let $\stack X$ be a smooth stacky curve.
We can express $\stack X$ as an \'etale gerbe $\stack X \to \orb X$, where $\orb X$ is an orbifold curve.
Fix a positive integer $b$ and let ${\orb X}[b]$ be a Fulton--MacPherson space of $b$ distinct points constructed in \autoref{prop:fm}.
Let $\pi \from {\orb Z} \to {\orb X}[b]$, and $\sigma_j \from {\orb X}[b] \to {\orb Z}$, and $\rho \from {\orb Z} \to {\orb X}$ be as in \autoref{prop:fm}.
We think of ${\orb X}[b]$ as the space of $b$-pointed stable degenerations of ${\orb X}$, and the data $(\orb Z, \rho \from {\orb Z} \to {\orb X}, \sigma_j)$ as the universal object.

Let the divisor $\Sigma \subset \orb Z$ be the union of the images of the sections $\sigma_j$.
Set $\stack Z = {\orb Z} \times_\rho \stack X$.
\begin{definition}[$\BrCov_d(\stack X, b)$]
Define $\BrCov_d(\stack X, b)$ as the category fibered in groupoids over $\Schemes$ whose objects over a scheme $T$ are
\[ (T \to {\orb X}[b], \phi \from {\orb P} \to {\stack Z}_T),\]
such that $\psi \from {\orb P} \to \orb Z_T$ induced by $\phi$ is representable, flat, and finite of degree $d$ with $\br \psi = \Sigma_T$.
\end{definition}

\begin{theorem}\label{prop:technical_main}
  $\BrCov_d(\stack X, b)$ is a Deligne--Mumford stack smooth and separated over $\C$ of dimension $b$.
  If the orders of the orbinodes of the fibers of $\orb Z \to {\orb X}[b]$ are sufficiently divisible, then it is also proper.
\end{theorem}
\begin{remark}
  Strictly speaking, $\BrCov_d(\stack X, b)$ is an abuse of notation since this object depends on the choice of ${\orb X}[b]$.
  However, as long as the nodes of $\orb Z \to {\orb X}[b]$ are sufficiently divisible, this choice will not play any role.
\end{remark}
\begin{proof}
  We have a natural transformation $\BrCov_d(\stack X, b) \to \BrCov_d(\orb Z / {\orb X}[b], \Sigma)$ defined by $\phi \mapsto \psi$.
  Let $S$ be a scheme with a map $\mu \from S \to \BrCov_d(\orb Z / {\orb X}[b], \Sigma)$.
  Such $\mu$ is equivalent to $(\pi \from {\orb P} \to S, \psi \from \orb P \to \orb Z_S)$.
  Then $S \times_\mu \BrCov_d(\stack X, b)$ is just the stack of lifts of $\psi$ to $\stack Z_S$.
  Equivalently, setting $\stack P = \orb P \times_{\phi} \stack Z$, the objects of $S \times_\mu \BrCov_d(\stack X, b)$ over $T/S$ are sections ${\orb P}_T \to {\stack P}_T$ of ${\stack P}_T \to {\orb P}_T$.
  We denote the latter by $\Sect_S(\stack P \to \orb P)$.

  That $S \times_\mu \BrCov_d(\stack X, b) = \Sect_S(\stack P \to \orb P)$ is a separated Deligne--Mumford stack of finite type over $S$ follows from \cite[Theorem~1.4.1]{abr.vis:02}.
  That $\Sect_S(\stack P \to \orb P) \to S$ is \'etale follows from the property that an \'etale morphism (here $\stack P \to \orb P$) is formally \'etale (that is, its sections have a unique infinitesimal extension property).
  This is standard for \'etale morphisms of schemes and not hard to check for Deligne--Mumford stacks (see \cite[Corollary~B.9]{ryd:11}).
  
  Assume that the orders of the nodes of $\orb Z_S \to S$ are sufficiently divisible.
  Let $p \in \orb P$ be a node of a fiber of $\orb P \to S$ and set $z = \psi(p)$.
  Since $\psi$ is representable, it induces an inclusion of stabilizers $\Aut_p \orb P \to \Aut_z \orb Z$.
  Since $\psi$ is finite of degree $d$, the quotient $\Aut_z \orb Z/\Aut_p \orb P$ has order $c$ with $c \leq d$.
  In other words, the order of $\Aut_p \orb P$ is $1/c$ times the order of $\Aut_z \orb Z$ where $c \leq d$.
  So we may assume that the orders of the nodes of $\orb P \to S$ are also sufficiently divisible.
  To check that $\Sect_S(\stack P \to \orb P) \to S$ is proper, let $S = \Delta$ be a DVR and let a section $s \from \orb P \to \stack P$ be given over the generic point of $\Delta$.
  First, $s$ extends to a section over the generic points of $\orb P_0$ after replacing $\Delta$ by a finite cover.
  Second, $s$ extends to a section over the smooth points and the generic nodes of $\orb P_0$ since $\orb P$ is locally simply connected and $S_2$ at these points.
  Finally, $s$ extends over the non-generic nodes by \autoref{prop:deep_orbinode} (the extension of the section on the coarse spaces follows from normality).
  
  We have thus proved that $\BrCov_d(\stack X, b) \to \BrCov_d(\orb Z/\orb X[b], \Sigma)$ is representable by a separated \'etale morphism of Deligne--Mumford stacks which is also proper if the orders of the nodes of $\orb Z \to \orb X[b]$ are sufficiently divisible.
  By combining this with \autoref{prop:fm}, we complete the proof.
\end{proof}

Let $\Maps(\stack X,d)$ be the Abramovich--Vistoli space of twisted stable maps to $\stack X$.
This is the moduli space of $\phi \from \orb P \to \stack X$, where $\orb P$ is a twisted curve and $\phi$ is a representable morphism such that the map on the underlying coarse spaces is a Kontsevich stable map that maps the fundamental class of $\coarse{\orb P}$ to $d$ times the fundamental class of $\coarse{\stack X}$.

\begin{proposition}\label{prop:map_to_av}
  $\BrCov_d(\stack X, b)$ admits a morphism to $\Maps(\stack X, d)$.
\end{proposition}
\begin{proof}
  On $\BrCov_d(\stack X, b)$ we have a universal twisted curve $\orb P \to \BrCov_d(\stack X, b)$ with a morphism $\orb P \to \stack X$. This morphism is obtained by composing the universal $\orb P \to \stack Z$ with $\rho \from \stack Z \to \stack X$.
  By \cite[Proposition~9.11]{abr.vis:02}, there exists a factorization $\orb P \to \orb P' \to \stack X$, where $\orb P' \to \BrCov_d(\stack X, b)$ is a twisted curve and $\orb P' \to \stack X$ is a twisted stable map.
  On coarse spaces, this factorization is the contraction of unstable rational components.
  The twisted stable map $\orb P' \to \stack X$ gives the morphism $\BrCov_d(\stack X, b) \to \Maps(\stack X, d)$.
\end{proof}

\section{Moduli of tetragonal curves on Hirzebruch surfaces}\label{sec:hirz}
In this section, we apply our results to $\stack X = [\overline M_{0,4}/\sym_4]$.
Set
\[ \widetilde{\stack M}_{0,4} := [\overline M_{0,4}/\sym_4].\]
Interpret the quotient as the moduli space of stable marked rational curves where the marking consists of a divisor of degree 4.
Let 
\[\pi \from (\widetilde{\orb S}, \widetilde{\orb C})  \to \widetilde{\stack M}_{0,4}\]
be the universal family, where $\widetilde{\orb S} \to \widetilde{\stack M}_{0,4}$ is a nodal curve of genus $0$ and $\widetilde{\orb C} \subset \widetilde{\orb S}$ a divisor of relative degree 4.

The action of $\sym_4$ on $\overline M_{0,4}$ has a kernel: the Klein four group
\[ K = \{ \id, (12)(34), (13)(24), (14)(23)\},\]
acts trivially.
Therefore, a generic $t \to \widetilde{\stack M}_{0,4}$ has automorphism group $K$.
The action of $K$ on $\widetilde{\orb S}_t$ and $\widetilde{\orb C}_t$ is faithful.
There are three special points $t$ at which $\Aut_t \widetilde{\stack M}_{0,4}$ jumps.
The first, which we label $t = 0$, is specified by
\[ (\widetilde{\orb S}_t, \widetilde{\orb C}_t) \cong (\P^1, \{1, i, -1, -i\}),\]
with $\Aut_0\widetilde{\stack M}_{0,4} = D_4 \subset \sym_4$.
The second, which we label $t = 1$, is specified by
\[ (\widetilde{\orb S}_t, \widetilde{\orb C}_t) \cong (\P^1, \{0, 1, e^{2\pi i/3}, e^{-2\pi i/3}\}),\]
with $\Aut_1\widetilde{\stack M}_{0,4} = A_4 \subset \sym_4$.
The third, which we label $t = \infty$, is specified by
\[ (\widetilde{\orb S}_t, \widetilde{\orb C}_t) \cong (\P^1 \cup \P^1, \{0,1; 0,1\})\]
where the two $\P^1$s are attached at a node (labeled $\infty$ on both).
We have $\Aut_\infty\widetilde{\stack M}_{0,4} = D_4 \subset \sym_4$. 

The quotient $\sym_4/K$ is isomorphic to $\sym_3$.
Therefore, the orbifold curve underlying $\widetilde{\stack M}_{0,4}$ is the quotient $[\overline M_{0,4}/\sym_3]$.
Consider the inclusion $\sym_3 \subset \sym_4$ as permutations acting only on the first three elements.
The inclusion $\sym_3 \to \sym_4$ is a section of the projection $\sym_4 \to \sym_3$.
We can thus think of $\sym_3$ as acting on $\overline M_{0,4}$ by permuting three of the four points and leaving the fourth fixed.
Set $\widetilde{\orb M}_{0,1+3} := [\overline M_{0,4}/\sym_3]$.
We can interpret this quotient as the moduli space of stable marked rational curves, where the marking consists of a point and a divisor of degree 3 (hence the notation ``1+3'').

We have the fiber product diagram
\[
\begin{tikzpicture}
  \node (M1) {$\widetilde{\stack M}_{0,4}$};
  \draw (M1)+(0,-1) node (M2) {$\widetilde{\orb M}_{0,1+3}$};
  \draw (M1)+(2,0) node (S4) {$B\sym_4$};
  \node[below of=S4] (S3) {$B\sym_3$};
  \draw [->] (M1) edge (M2) (S4) edge (S3) (M1) edge (S4) 
  (M2) edge (S3);
\end{tikzpicture}.
\]
Since $\sym_4 = K \rtimes \sym_3$, the map $B\sym_4 \to B\sym_3$ is the trivial $\orb K$ gerbe $B\orb K$, where $\orb K$ is the sheaf of groups on $B \sym_3$ obtained by the action of $\sym_3$ on $K$ by conjugation.
Therefore, we get that $\widetilde{\stack M}_{0,4} = B {\orb K} \times_{B\sym_3} \widetilde{\orb M}_{0,1+3}$.

\subsection{The tetragonal-trigonal correspondence}
The relation $\sym_4 = K \rtimes \sym_3$ gives a stacky perspective on a classical relation between quadruple and triple covers explored by Recillas \cite{rec:73}.
\begin{proposition}\label{prop:recillas}
  Let $\orb P$ be a Deligne--Mumford stack.
  We have a natural bijection between $\{\phi \from \orb C \to \orb P\},$ where $\phi$ is a finite \'etale cover of degree 4 and $\{(\psi \from \orb D \to \orb P, \orb L)$ where $\psi$ is a finite \'etale cover of degree 3 and $\orb L$ a line bundle on $\orb D$ with $\orb L^2 = \O_{\orb D}$ and $\Norm_\psi \orb L = \O_{\orb P}$.
  Furthermore, under this correspondence we have $\phi_* \O_{\orb C} = \O_{\orb P} \oplus \psi_* \orb L$.
\end{proposition}
\begin{proof}
  This is essentially the content of \cite{rec:73}.
  We sketch a proof using stacks.

  An \'etale cover of degree $d$ is equivalent to a map to $B\sym_d$.
  From $\sym_4 = K \rtimes \sym_3$, we get that an \'etale cover $\orb C \to \orb P$ of degree 4 is equivalent to a map $\mu \from \orb P \to B\sym_3$ and a section of $\orb K \times_\mu \orb P$.
  Such a section is in turn equivalent to an element of $H^1(\orb P, \orb K)$.
  Let $\psi \from \orb D \to \orb P$ be the \'etale cover of degree 3 corresponding to $\mu$.
  Denote $\orb K\times_\mu \orb P$ by $\orb K$ for brevity.
  We have the following exact sequence on $\orb P$ (pulled back from an analogous exact sequence on $B\sym_3$):
  \[ 0 \to \orb K \to \psi_* \left({\Z_2}\right)_{\orb D} \to \left(\Z_2\right)_{\orb P} \to 0.\]
  The associated long exact sequence gives
  \[ H^1(\orb P, \orb K) = \ker\left(H^1(\orb D, \Z_2) \to H^1(\orb P, \Z_2)\right).\]
  If we interpret $H^1(-,\Z_2)$ as two-torsion line bundles, then the map $H^1(\orb D, \Z_2) \to H^1(\orb P, \Z_2)$ is the norm map.
  The bijection follows.

  In the rest of the proof, we view the data of the line bundle $\orb L$ as the data of an \'etale double cover $\tau \from \widetilde{\orb D} \to \orb D$.
  The double cover and the line bundle are related by $\tau_* \O_{\widetilde {\orb D}} = \O_{\orb D} \oplus \orb L$.

  It suffices to prove the last statement universally on $B\sym_4$---it will then follow by pullback.
  A cover of $B\sym_4$ is just a set with an $\sym_4$ action and a sheaf on $B\sym_4$ is just an $\sym_4$-representation.
  Consider the 4-element $\sym_4$-set $\orb C = \{1,2,3,4\}$.
  The corresponding 3-element $\sym_4$-set $\orb D$ with an \'etale double cover $\tau \from \widetilde {\orb D} \to {\orb D}$ is given by
  \[ \widetilde {\orb D} = \{(12),(13),(14),(23),(24),(34)\} \xrightarrow{\tau} {\orb D} = \{(12)(34), (13)(24), (14)(23)\}.\]
  It is easy to check that we have an isomorphism of $\sym_4$-representations
  \[ \C \oplus \C [\widetilde {\orb D} ] = \C [ {\orb D}] \oplus \C[ \orb C ].\]
  In terms of the map $\phi \from \orb C \to B\sym_4$, the map $\psi \from {\orb D} \to B\sym_4$, and the map $\tau \from \widetilde {\orb D} \to {\orb D}$, this isomorphism can be written as an isomorphism of sheaves on $B\sym_4$:
  \[ \O \oplus \psi_* \O_{\orb D} \oplus \psi_* {\orb L} = \psi_*\O_{\orb D} \oplus \phi_*\O_{\orb C}.\]
  Canceling $\psi_* \O_{\orb D}$ gives the statement we want.
\end{proof}

\subsection{Tetragonal curves on Hirzebruch surfaces and covers of $\widetilde {\stack M}_{0,4}$}
Let $f \from S \to \P^1$ be a $\P^1$-bundle and $C \subset S$ a smooth curve such that $f \from C \to \P^1$ is a finite simply branched map of degree $4$.
Away from the $b = 2 g(C) + 6$ branch points $p_1, \dots, p_b$ of $C \to \P^1$, we get a morphism
\[ \phi \from \P^1 \setminus \{p_1, \dots, p_b\} \to \widetilde{\stack M}_{0,4} \setminus \{\infty\}.\]
Set $\orb P = \P^1(\sqrt{p_1}, \dots, \sqrt p_b)$.
Abusing notation, denote the point of $\orb P$ over $p_i$ by the same letter.
Then $\phi$ extends to a morphism $\phi \from \orb P \to \widetilde{\stack M}_{0,4}$, which maps $p_1, \dots, p_b$ to $\infty$, is \'etale over $\infty$, and the underlying map of orbifolds $\orb P \to \widetilde{\orb M}_{0;1+3}$ is representable of degree $b$.
We can construct the family of 4-pointed rational curves that gives $\phi$ as follows.
Consider the $\P^1$-bundle $S \times_{\P^1} \orb P \to \orb P$ and the curve $C \times_{\P^1} \orb P \subset S \times_{\P^1} \orb P$.
Since $C \to \P^1$ was simply branched, $C \times_{\P^1} \orb P$ has a node over each $p_i$.
Let $\widetilde S \to S \times_{\P^1} \orb P$ be the blow up at these nodes and $\widetilde C$ the proper transform of $C$.
The pair $(\widetilde S, \widetilde C)$ over $\orb P$ gives the map $\phi \from \orb P \to \widetilde{\stack M}_{0,4}$.
The geometric fiber of $(\widetilde S, \widetilde C)$ over $p_i$ is isomorphic to $(\P^1 \cup \P^1, \{0,1; 0,1\})$ where we think of the $\P^1$s as joined at $\infty$.
The action of $\Z_2 = \Aut_{p_i} \orb P$ is trivial on one component and is given by $x \mapsto 1-x$ on the other component.

Conversely, let $\phi \from \orb P \to \widetilde{\stack M}_{0,4}$ be a morphism that maps $p_1, \dots, p_b$ to $\infty$, is \'etale over $\infty$, and the underlying map of orbifolds $\orb P \to \widetilde{\orb M}_{0;1+3}$ is representable of degree $b$.
Let $f \from (\widetilde S, \widetilde C) \to \orb P$ be the corresponding family of 4-pointed rational curves.
Away from $p_1, \dots, p_b$, the map $f \from  \widetilde S \to \orb P$ is a $\P^1$-bundle.
Locally near $p_i$, we have
\[ \orb P = [\spec \C[t] / \Z_2].\]
Set $U = \spec \C[t]$.
Since $\phi$ is \'etale at $t = 0$, the total space $\widetilde S_U$ is smooth.
Since $t = 0$ maps to $\infty$, the fiber of $f$ over $0$ is isomorphic to $(\P^1 \cup \P^1, \{0,1,0,1\})$ where we think of the $\P^1$s as joined at $\infty$.
Consider the map $\Z_2 = \Aut_{p_i} \orb P \to D_4 = \Aut_\infty \widetilde{\stack M}_{0,4}$.
Since the map induced by $\phi$ on the underlying orbifolds is representable, the image of the generator of $\Z_2$ is an element of order 2 in $D_4$ not contained in the Klein four subgroup.
The only possibility is a 2-cycle, whose action on the fiber is trivial on one component and $x \mapsto 1-x$ on the other.
Let $\widetilde S_U \to S'_U$ be obtained by blowing down the component on which the action is non-trivial and let $C'_U \subset S'_U$ be the image of $\widetilde C_U$.
Note that the $\Z_2$ action on $(\widetilde S_U, \widetilde C_U)$ descends to an action on $(S'_U, C'_U)$ which is trivial on the central fiber.
Thus $S'_U/\Z_2 \to U/\Z_2$ is a $\P^1$ bundle and $C'_U/\Z_2 \to U/\Z_2$ is simply branched (the quotients here are geometric quotients, not stack quotients).
Let $(S, C)$ be obtained from $(\widetilde S, \widetilde C)$ by performing this blow down and quotient operation around every $p_i$.
Then $f \from S \to \P^1$ is a $\P^1$-bundle and $C \subset S$ is a smooth curve such that $C \to \P^1$ is a finite, simply branched cover of degree 4.
We call the construction of $(S,C)$ from $(\widetilde S, \widetilde C)$ the \emph{blow-down construction}.

We thus have a natural bijection
\begin{equation}\label{eqn:basic_bijection}
  \{f \from (S, C) \to \P^1 \} \leftrightarrow \{ \phi \from \orb P \to \widetilde{\stack M}_{0,4}\},
\end{equation}
where on the left we have a $\P^1$-bundle $S \to \P^1$ and a smooth curve $C \subset S$ such that $f \from C \to \P^1$ is a finite, simply branched cover of degree 4 and on the right we have $\orb P = \P^1(\sqrt p_1, \dots, \sqrt p_b)$ and $\phi$ that maps $p_i$ to $\infty$, is \'etale over $\infty$, and induces a representable finite map of degree $b$ to $\widetilde{\orb M}_{0,1+3}$.

Assume that $C \subset S$ on the left is general so that the induced map $\orb P \to \widetilde{\orb M}_{0,1+3}$ is simply branched over distinct points away from $0$, $1$, or $\infty$.
Then the monodromy of $\orb P \to \widetilde{\orb M}_{0,1+3}$ over $0$, $1$, and $\infty$ is given by a product of $2$-cycles, a product of $3$-cycles, and identity, respectively.
Said differently, the map $\coarse{\phi} \from \P^1 \to \P^1$ has ramification $(2,2,\dots)$ over $0$; ramification $(3,3,\dots)$ over $1$, and ramification $(1,1,\dots)$ over $\infty$.
In particular, the degree $b$ of $\coarse\phi$ is divisible by $6$.
Taking $b = 6d$, the map $\phi \from \orb P \to \widetilde{\orb M}_{0;1+3}$ has $5d-2$ branch points.

\begin{definition}[$\orb Q_d$ and $\orb T_d$]
Denote by $\orb Q_d$ the open and closed substack of $\BrCov_{6d}(\widetilde{\stack M}_{0,4}, 5d-2)$ that parametrizes covers with connected domain and ramification $(2,2,\dots)$ over $0$, ramification $(3,3,\dots)$ over $1$, and ramification $(1,1,\dots)$ over $\infty$.
Likewise, denote by $\orb T_d$ the open and closed substack of $\BrCov_{6d}(\widetilde{\orb M}_{0,1+3}, 5d-2)$ defined by the same two conditions.
\end{definition}
Let $\overline{\orb H}_{d,g}$ be the space of admissible covers of degree $d$ of genus 0 curves by genus $g$ curves as in \cite{abr.cor.vis:03}.
Recall that the \emph{directrix} of $\F_n$ is the unique section of $\F_n \to \P^1$ of negative self-intersection.
\begin{proposition}\label{prop:components}
$\orb Q_d$ is a smooth and proper Deligne--Mumford stack of dimension $5d-2$.
For $d \geq 2$,  it has three connected ($=$ irreducible) components $\orb Q_d^0$, $\orb Q_d^{\rm odd}$, and $\orb Q_d^{\rm even}$.
Via \eqref{eqn:basic_bijection}, general points of these components correspond to the following $f \from (S, C) \to \P^1$:
  \begin{itemize}
  \item [\quad $\orb Q_d^0$:] $S = \F_d$ and $C = $ a disjoint union of the directrix $\sigma$ and a general curve of class $3(\sigma + d F)$.
  \item [\quad $\orb Q_d^{\rm odd}$:] $S = \F_1$ and $C = $ a general curve of class $4 \sigma + (d+2)F$.
  \item [\quad $\orb Q_d^{\rm even}$:] $S = \F_0$ and $C = $ a general curve of class $(4,d)$.
  \end{itemize}
  The components of $\orb Q_d$ admit morphisms to the corresponding spaces of admissible covers, namely $\orb Q_d^0 \to \overline {\orb H}_{3,3d-2}$, $\orb Q_d^{\rm odd} \to \overline {\orb H}_{4, 3d-3}$, and $\orb Q_d^{\rm even} \to \overline {\orb H}_{4,3d-3}$.
\end{proposition}
\begin{proof}
  That $\orb Q_d$ is a smooth and proper Deligne--Mumford stack of dimension $5d-2$ follows from \autoref{prop:technical_main}.
  That the components admit morphisms to the spaces of admissible covers follows from the same argument as in \autoref{prop:map_to_av}.

  Recall that $\orb Q_d$ parametrizes covers of $\widetilde{\stack M}_{0,4}$ and its degenerations.
  Let $U \subset \orb Q_d$ be the dense open subset of non-degenerate covers.
  It suffices to show that $U$ has three connected components.
  Via \eqref{eqn:basic_bijection}, the points of $U$ parametrize $f \from (S, C) \to \P^1$, where $S \to \P^1$ is a $\P^1$-bundle and $C \subset S$ is a smooth curve such that $C \to \P^1$ is simply branched of degree 4.
  Say $S = \F_n$.
  Since $C \to \P^1$ is degree 4 and ramified at $6d$ points, we get
  \[ [C] = 4\sigma + (d+2n) F.\]

  Let $U^0 \subset U$ be the open and closed subset where $C$ is disconnected.
  Since $C$ is smooth, the only possibility is $n = d$ and $C$ is the disjoint union of $\sigma$ and a curve in the class $3(\sigma+dF)$.
  As a result, $U^0$ is irreducible and hence a connected component of $U$.

  Let $U^{\rm even} \subset U$ be the open and closed subset where $n$ is even.
  Since $C$ is smooth, we must have $d+2n \geq 4d$.
  In particular, $H^1(S, \O_S(C)) = H^2(S, \O_S(C)) = 0$, and hence $(S, C)$ is the limit of $(\F_0, C_{\rm gen})$, where $C_{\rm gen} \subset \F_0$ is a curve of type $(4, d)$.
  Therefore, $U^{\rm even}$ is irreducible, and hence a connected component of $U$.

  By the same reasoning, the open and closed subset $U^{\rm odd} \subset U$ where $n$ is odd is the third connected component of $U$.
\end{proof}  

\subsection{The odd and even components of $\orb Q_d$ and theta characteristics}\label{sec:parity}
There is a second explanation for the connected components of $\orb Q_d$, which involves the theta characteristics of the trigonal curve $D \to \P^1$ associated to the tetragonal curve $C \to \P^1$ via the Recillas correspondence (see \cite{vak:01}).
Let $V \subset \orb T_d$ be the open set parametrizing non-degenerate covers of $\widetilde{\orb M}_{0,1+3}$.
It is easy to check that $V$ is irreducible and the map $\orb Q_d \to \orb T_d$ is representable, finite, and \'etale over $V$.
Therefore, the connected components of $\orb Q_d$ correspond to the orbits of the monodromy of $\orb Q_d \to \orb T_d$ over $V$.
Let $v$ be a point of $V$; let $\psi_v \from \orb P \to \widetilde{\orb M}_{0,1+3}$ the corresponding map; let $(\widetilde {\orb T}, \widetilde\sigma \sqcup \widetilde {\orb D}_v) \to \orb P$ the corresponding (1+3)-pointed family of rational curves; and let $f \from (S, \sigma \sqcup D_v) \to \P^1$ the family obtained by the blow-down construction as in \eqref{eqn:basic_bijection}.
Note that $D_v$ is the coarse space of $\widetilde{\orb D}_v$.

By \autoref{prop:recillas}, the points of $\orb Q_d$ over $v$ are in natural bijection with the norm-trivial two-torsion line bundles $\orb L$ on $\widetilde {\orb D}_v$.
Since $\coarse{\orb P}$ has genus $0$, a line bundle on $\orb P$ is trivial if and only if it has degree 0 and the automorphism groups at the orbifold points act trivially on its fibers.
Let $p_1, \dots, p_{6d}$ be the orbifold points of $\orb P$.
Note that $\widetilde{\orb D}_v$ also has the same number of orbifold points, say $q_1, \dots, q_{6d}$, with $q_i$ lying over $p_i$.
All the orbifold points, $\{p_i\}$ and $\{q_j\}$, have order 2.
Since $q_i$ is the only orbifold point over $p_i$, the action of $\Aut_{p_i}{\orb P}$ on the fiber of $\Norm \orb L$ is trivial if and only if the action of $\Aut_{q_i}{\orb D}_v$ on the fiber of $\orb L$ is trivial.
If this is the case for all $i$, then $\orb L$ is a pullback from the coarse space $D_v$.
Thus, norm-trivial two-torsion line bundles on $\orb{\widetilde D}_v$ are just pullbacks of two-torsion line bundles on $D_v$.
The component ${\orb Q}_d^0$ corresponds to the trivial line bundle.
The non-trivial ones split into two orbits because of the natural theta characteristic $\theta = f^*\O_{\P^1}(d-1)$ on $D_v$.

We can summarize the above discussion in the following sequence of bijections:
\begin{equation}\label{eqn:theta_bijection}
  \left\{\parbox{8em}{Points in $\orb Q_d$ over a general $v \in \orb T_d$}\right\}
  \leftrightarrow 
  \left\{\parbox{8em}{Two torsion line bundles on $D_v$}\right\}
  \stackrel{\otimes\theta}{\leftrightarrow} 
  \left\{\parbox{9em}{Theta characteristics on $D_v$}\right\}.
\end{equation}
\begin{proposition}
  Under the bijection in \eqref{eqn:theta_bijection}, the points of $\orb Q_d^{\rm even}$ correspond to even theta characteristics and the points of $\orb Q_d^{\rm odd}$ correspond to odd theta characteristics.
\end{proposition}
\begin{proof}
  Let $u \in \orb Q_d$ be a point over $v$.
  Let $f \from (S, C) \to \P^1$ be the corresponding 4-pointed curve on a Hirzebruch surface and $L$ the corresponding two-torsion line bundle on $D_v$.
  By \autoref{prop:recillas}, we get 
  \[\O_{\P^1} \oplus f_* L  = f_* \O_C.\]
  Tensoring by $\O_{\P^1}(d-1)$ gives
  \[ \O_{\P^1}(d-1) \oplus f_* \theta = f_* \O_C \otimes \O_{\P^1}(d-1).\]
  Thus the parity of $\theta$ is the parity of $h^0(C, f^*\O_{\P^1}(d-1)) - d$.
  It is easy to calculate that for $C$ on $\F_0$ of class $(4,d)$, this quantity is $0$ and for $C$ on $\F_1$ of class $(4\sigma + (d+2)F)$, this quantity is $1$.
\end{proof}

It will be useful to understand the theta characteristic $\theta$ on $D_v$ in terms of the map to $\widetilde{\orb M}_{0,1+3}$.
Let $(\orb T, \sigma \sqcup \orb D) \to \widetilde{\orb M}_{0,1+3}$ be the universal $(1+3)$-pointed curve.
The curve $\orb D$ has genus 0 and has two orbifold points, both of order 2, one over $0$ and one over $\infty$.
Let $\widetilde{\orb M}_{0,1+3} \to \widetilde{\orb M}_{0;1+3}'$ be the coarse space around $\infty$ and $\orb D \to \orb D'$ the coarse space around the orbifold point over $\infty$.
Then $\orb D'$ is a genus 0 orbifold curve with a unique orbifold point of order $2$.
Furthermore, $\orb D' \to \widetilde{\orb M'}_{0,1+3}$ is simply branched over $\infty$ and the line bundle $\O(1/2)$ on $\orb D'$ is the square root of the relative canonical bundle of $\orb D' \to \widetilde{\orb M'}_{0,1+3}$.
We have the fiber diagram
\begin{equation}\label{eqn:theta}
  \begin{tikzpicture}
    \node (D) {$D_v$};
    \draw (D)+(2,0) node (Dp) {$\orb D'$};
    \draw (D)+(0,-1) node (P) {$\P^1$};
    \draw (D)+(2,-1) node (M) {$\widetilde{\orb M}_{0,1+3}'$};
    \path[->] (D) edge node[above]{\scriptsize $\mu$} (Dp) edge (P) (Dp) edge (M) (P) edge (M);
  \end{tikzpicture}.
\end{equation}
Thus $\theta_{\rm rel} = \mu^*\O(1/2)$ is a natural relative theta characteristic on $D_v$.
With the unique theta characteristic $\O(-1)$ on $\P^1$, we get the theta characteristic $\theta = \theta_{\rm rel} \otimes \O(-1)$.

\section{Limits of plane quintics}\label{sec:quintics}
In this section, we fix $d = 3$ and write $\orb Q$ for $\orb Q_3$.
By \autoref{prop:components}, a general point of $\orb Q^{\rm odd}$ corresponds to a curve of class $(4\sigma + 5F)$ on $\F_1$.
Such a curve is the proper transform of a plane quintic under a blow up $\F_1 \to \P^2$ at a point on the quintic.
Therefore, the image in $\overline{\orb M}_6$ of $\orb Q^{\rm odd}$ is the closure of the locus $Q$ of plane quintic curves. 
The goal of this section is to describe the elements in the closure.
More specifically, we will determine the stable curves corresponding to the generic points of the irreducible components of $\Delta \cap \overline Q$, where $\Delta \subset \overline{\orb M}_6$ is the boundary divisor.

We have the sequence of morphisms
\[ \orb Q^{\rm odd} \xrightarrow{\alpha} \overline{\orb H}_{4,6} \xrightarrow{\beta} \overline{\orb M}_6.\]
Set $\widetilde Q = \alpha(\orb Q^{\rm odd})$.
We then get the sequence of surjections 
\[\orb Q^{\rm odd} \xrightarrow{\alpha} \widetilde Q \xrightarrow{\beta} \overline Q.\]
Let $U \subset \orb Q$ be the locus of non-degenerate maps.
Call the irreducible components of $\orb Q \setminus U$ the \emph{boundary divisors} of $\orb Q$.
\begin{proposition}\label{prop:prune_boundary}
  Let $B$ be an irreducible component of $\overline Q \cap \Delta$.
  Then $B$ is the image of a boundary divisor $\orb B$ of $\orb Q^{\rm odd}$ that satisfies
  \begin{enumerate}
  \item $\dim \alpha(\orb B) = \dim \orb B = 12$,
  \item $\dim \beta \circ \alpha (\orb B) = 11$, and
  \item $\beta \circ \alpha (\orb B) \subset \Delta$.  
  \end{enumerate}
\end{proposition}
\begin{proof}
  Note that $\dim {\orb Q}^{\rm odd} = \dim \widetilde Q = 13$ and $\dim \overline Q = 12$.
  Since $\Delta$ is a Cartier divisor, we have $\codim(B, \overline Q) = 1$.
  Let $\widetilde B \subset \widetilde Q$ be an irreducible component of $\beta^{-1}(\widetilde B)$ that surjects onto $B$.
  Then $\codim(\widetilde B, \widetilde Q) = 1$.
  Let $\orb B \subset \orb Q^{\rm odd}$ be an irreducible component of $\alpha^{-1}(\widetilde B)$ that surjects onto $\widetilde B$.
  Then $\orb B$ is the required boundary divisor of $\orb Q$.
\end{proof}

Recall that points of $\orb Q$ correspond to certain finite maps
$\phi \from \orb P \to \orb Z$, where $\orb Z \to \widetilde{\stack M}_{0,4}$ is a pointed degeneration.
Set $P = \coarse{\orb P}$ and $Z = \coarse{\orb Z}$. 
The map $P \to Z$ is an admissible cover with ramification $(2,2,\dots)$ over $0$, ramification $(3,3,\dots)$ over $1$, and ramification $(1,1,\dots)$ over $\infty$.
We encode the topological type of the admissible cover by its dual graph $(\Gamma_\phi \from \Gamma_P \to \Gamma_Z)$.
Here the graph $\Gamma_Z$ is the dual graph of $(Z, \{0,1,\infty\})$, the graph $\Gamma_P$ is the dual graph of $P$, and $\Gamma_\phi$ is a map that sends the vertices and edges of $\Gamma_P$ to the corresponding vertices and edges of $\Gamma_Z$.
We decorate each vertex of $\Gamma_P$ by the degree of $\phi$ on that component and each edge of $\Gamma_P$ by the local degree of $\phi$ at that node.
We indicate the main component of $Z$ by a doubled circle.
For the generic points of divisors, $Z$ has two components, the main component and a `tail'.
In this case, we will omit writing the vertices of $\Gamma_P$ corresponding to the `redundant components'---these are the components over the tail that are unramified except possibly over the node and the marked point.
These can be filled in uniquely.
\autoref{fig:adm_cover_graph} shows an example of an admissible cover and its dual graph, with and without the redundant components.
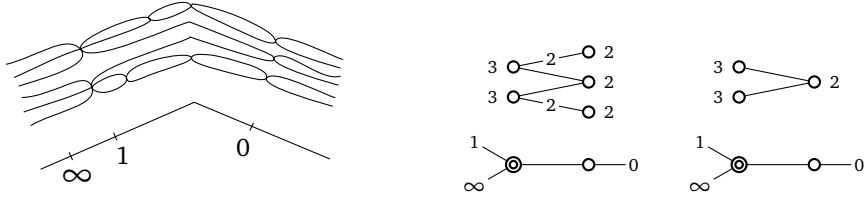
\begin{figure}
  \begin{center}
    \begin{tikzpicture}
      \begin{scope}[shift={(3.5,1.5)}]
        \node {

\begin{tikzpicture}[y=0.80pt, x=0.80pt, yscale=-1.000000, xscale=1.000000, inner sep=0pt, outer sep=0pt]
\begin{scope}[shift={(-158.95462,-314.71943)}]
  \begin{scope}[cm={{0.52175,0.0,0.0,0.52175,(75.91434,190.57269)}}]
    \path[draw=black] (190.7143,390.2193) -- (330.0000,329.5050) --
      (330.0000,329.5050) -- (452.8571,380.9336)(402.9090,272.3738) .. controls
      (400.2541,266.9565) and (395.5978,264.9814) .. (391.5665,262.4235) .. controls
      (372.9311,251.7681) and (351.8264,244.1348) .. (331.8588,240.7165) .. controls
      (330.1414,240.8912) and (327.0436,238.6002) .. (327.4606,242.8601) .. controls
      (331.4385,246.4634) and (336.1058,248.1605) .. (340.4957,250.5132) .. controls
      (355.1296,257.5313) and (369.8524,264.1830) .. (384.5489,270.9354) .. controls
      (389.9950,273.0990) and (395.4902,275.6890) .. (401.2190,275.7759) .. controls
      (402.6385,275.9822) and (404.6331,274.3710) .. (402.9090,272.3738) --
      cycle(344.9889,286.6155) .. controls (341.2628,286.4247) and
      (332.6261,284.9891) .. (329.1055,287.2303) .. controls (325.3377,288.6285) and
      (328.8808,292.4879) .. (330.8749,292.9208) .. controls (338.7857,296.9467) and
      (347.1860,298.5494) .. (355.3922,300.9939) .. controls (366.3280,303.8871) and
      (377.2654,306.8782) .. (388.3175,308.8387) .. controls (390.4850,309.1127) and
      (392.7080,309.0873) .. (394.8390,308.2607) .. controls (397.4043,304.3503) and
      (392.2985,301.4561) .. (390.2554,300.1641) .. controls (382.1759,295.4413) and
      (373.4083,293.9280) .. (364.9318,291.1675) .. controls (356.6894,288.8311) and
      (353.3469,288.0182) .. (344.9889,286.6155) -- cycle(326.0359,256.7599) ..
      controls (326.0359,256.7599) and (400.1989,285.1699) .. (401.1812,287.8808) ..
      controls (402.1636,290.5916) and (404.1658,293.8507) .. (390.8345,290.5260) ..
      controls (368.6565,284.9952) and (325.9030,270.2124) ..
      (325.9030,270.2124)(464.2890,290.2561) .. controls (465.4841,293.2782) and
      (413.0985,264.6714) .. (404.4805,272.5660) .. controls (395.8626,280.4606) and
      (462.3600,298.4788) .. (462.3600,298.4788)(462.5612,305.7630) .. controls
      (444.8892,311.2239) and (415.5271,281.4360) .. (402.9280,288.3818) .. controls
      (394.3101,296.2763) and (460.6322,313.9857) ..
      (460.6322,313.9857)(457.3834,324.8648) .. controls (458.5785,327.8869) and
      (404.6214,299.0879) .. (396.0035,306.9825) .. controls (387.3855,314.8771) and
      (455.4545,333.0875) .. (455.4545,333.0875)(268.6492,308.2292) .. controls
      (265.1845,294.1201) and (322.6802,282.2416) .. (327.0000,288.6096) .. controls
      (329.4258,292.1857) and (310.7104,297.2356) .. (294.1915,302.9904) .. controls
      (281.2939,307.4836) and (270.1683,314.4153) .. (268.6492,308.2292) --
      cycle(175.3245,338.3460) .. controls (188.2696,333.4713) and
      (217.9566,325.6445) .. (237.1704,315.8187) .. controls (236.1282,313.7280) and
      (266.0068,296.3065) .. (268.6492,308.2292) .. controls (269.6408,312.7029) and
      (258.8131,319.0980) .. (249.7404,320.3860) .. controls (242.6668,321.3902) and
      (236.8773,318.1690) .. (237.1704,315.8187) .. controls (228.9339,298.7363) and
      (329.1992,273.6094) .. (325.8944,270.0500)(181.7279,350.6398) .. controls
      (203.2316,337.7209) and (242.7066,332.8656) .. (237.1704,315.8187) .. controls
      (231.5570,298.5340) and (203.5281,324.0004) ..
      (173.6378,325.1956)(288.1692,252.7338) .. controls (290.6684,246.2344) and
      (320.1678,233.8974) .. (326.4138,241.0591) .. controls (326.6179,241.2930) and
      (326.7971,241.5478) .. (326.9497,241.8242) .. controls (327.0402,241.9880) and
      (326.9270,242.7673) .. (326.5809,243.1157) .. controls (322.6386,247.0839) and
      (292.6080,261.0515) .. (288.1692,252.7338) -- cycle(166.6987,307.0020) ..
      controls (179.4331,301.6950) and (205.5269,287.3969) .. (226.8315,281.1644) ..
      controls (239.3291,274.4456) and (302.2343,257.4455) .. (288.1692,252.7338) ..
      controls (277.8339,249.2716) and (223.4714,271.5133) .. (226.8315,281.1644) ..
      controls (236.8638,296.4514) and (315.2039,257.7056) ..
      (326.1785,256.6014)(159.3534,295.0806) .. controls (184.9525,289.3748) and
      (214.4554,266.9559) .. (226.7577,281.3275) .. controls (239.2316,295.8995) and
      (191.2485,308.2527) .. (170.6734,314.8668);
    \path[draw=black] (216.5017,374.8122) .. controls (217.3636,377.2539) and
      (218.6381,379.5193) .. (219.6522,381.8962);
    \path[draw=black] (256.8451,357.0008) .. controls (257.6635,359.4574) and
      (258.8975,361.7451) .. (259.8692,364.1396);
    \path[draw=black] (383.5257,348.1183) .. controls (382.6035,350.5379) and
      (381.2732,352.7710) .. (380.2004,355.1220);
    \path[fill=black] (223.8214,396.4018) node (text3040) {$\infty$};
    \path[fill=black] (265.2561,377.8749) node (text3040-7) {$1$};
    \path[fill=black] (374.3989,370.5891) node (text3040-3) {$0$};
  \end{scope}
\end{scope}

\end{tikzpicture}
 };
      \end{scope}
      \begin{scope}[shift={(8,0.5)}]
        \draw
        (0,0) node[curve, double] (P1) {}
        +(210:.6) node[int] (I) {$\infty$} edge (P1)
        +(150:.6) node[int] (1) {$1$} edge (P1)
        (1,0) node[curve] (P2) {}
        +(0:.6) node[int] {0} edge (P2);
      \end{scope}
      \begin{scope}[shift={(8,1.2)}]
        \draw
        (1,0) node[curve, label={0:2}] (C21) {}
        (1,0.4) node[curve, label={0:2}] (C22) {}
        (1,0.8) node[curve, label={0:2}] (C23) {}
        (0,0.2) node[curve, label={180:3}] (C11) {}
        (0,0.6) node[curve, label={180:3}] (C12) {};
        \draw
        (P1) edge (P2)
        (C11) edge node[int] {2} (C21)
        (C11) edge (C22)
        (C12) edge (C22)
        (C12) edge node[int] {2} (C23);
      \end{scope};
      \begin{scope}[shift={(11,0.5)}]
        \draw
        (0,0) node[curve, double] (P1) {}
        +(210:.6) node[int] (I) {$\infty$} edge (P1)
        +(150:.6) node[int] (1) {$1$} edge (P1)
        (1,0) node[curve] (P2) {}
        +(0:.6) node[int] (0) {$0$} edge (P2);
      \end{scope}
      \begin{scope}[shift={(11,1.2)}]
        \draw
        (1,.4) node[curve, label={0:2}] (C22) {}
        (0,.6) node[curve, label={180:3}] (C11) {}
        (0,.2) node[curve, label={180:3}] (C12) {};
        \draw
        (P1) edge (P2)
        (C11) edge (C22)
        (C12) edge (C22);
      \end{scope};
    \end{tikzpicture}
  \end{center}
  \caption{An admissible cover, and its dual graph with and without the redundant components}
  \label{fig:adm_cover_graph}
\end{figure}

\begin{proposition}\label{prop:not_contracted_1}
  Let $\orb B \subset \orb Q$ be a boundary divisor such that $\dim \alpha(\orb B) = \dim \orb B$ and $\alpha(\orb B) \subset \overline {\orb H}_{4,6} \setminus {\orb H}_{4,6}$.
  Then the generic point of $\orb B$ has one of the following dual graphs (drawn without the redundant components).
  \begin{multicols}{3}
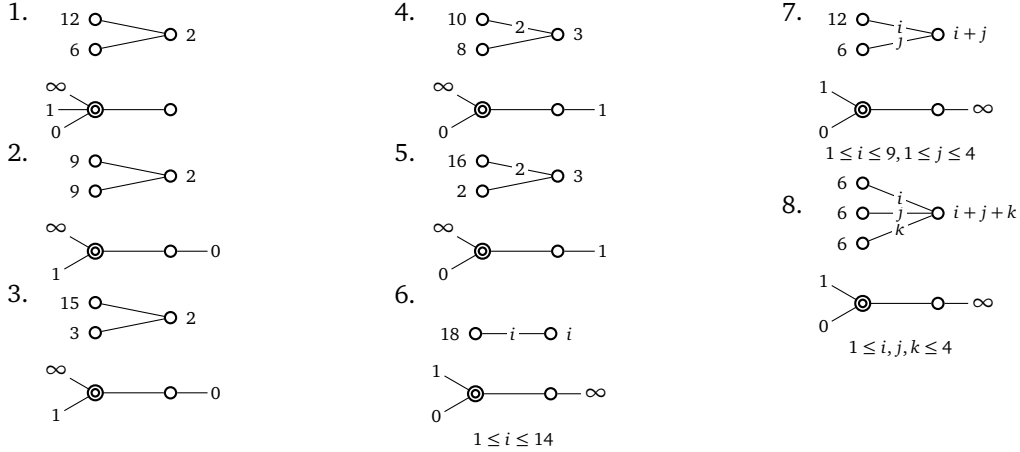
    \begin{enumerate}[label=\arabic*., ref=\arabic*]
    \item
      \begin{tikzpicture}[baseline=(C2)]
        \draw
        (0,0) node[curve, double] (P1) {}
        +(150:.6) node[int]  {$\infty$} edge (P1)
        +(180:.6) node[int]  {$1$}  edge (P1)
        +(210:.6) node[int]  {$0$}  edge (P1)
        (1,0) node[curve] (P2) {};
        \begin{scope}[shift={(0,.8)}]
          \draw
          (0,0) node[curve, label={180:6}] (C1) {}
          (0,.4) node[curve, label={180:12}] (C2) {}
          (1,.2) node[curve, label={0:2}] (E) {};
        \end{scope};
        \draw (C1) edge (E)  (C2) edge (E)      (P1) edge (P2)       ;
      \end{tikzpicture}
    \item
      \begin{tikzpicture}[baseline=(C2)]
        \draw
        (0,0) node[curve, double] (P1) {}
        +(150:.6) node[int]  {$\infty$} edge (P1)
        +(210:.6) node[int]  {$1$}  edge (P1)
        (1,0) node[curve] (P2) {}
        +(0:.6) node[int]  {$0$}  edge (P2);
        \begin{scope}[shift={(0,.8)}]
          \draw
          (0,0) node[curve, label={180:9}] (C1) {}
          (0,.4) node[curve, label={180:9}] (C2) {}
          (1,.2) node[curve, label={0:2}] (E) {};
        \end{scope};
        \draw (C1) edge (E)  (C2) edge (E)      (P1) edge (P2)       ;
      \end{tikzpicture}
    \item
      \begin{tikzpicture}[baseline=(C2)]
        \draw
        (0,0) node[curve, double] (P1) {}
        +(150:.6) node[int]  {$\infty$} edge (P1)
        +(210:.6) node[int]  {$1$}  edge (P1)
        (1,0) node[curve] (P2) {}
        +(0:.6) node[int]  {$0$}  edge (P2);
        \begin{scope}[shift={(0,.8)}]
          \draw
          (0,0) node[curve, label={180:3}] (C1) {}
          (0,.4) node[curve, label={180:15}] (C2) {}
          (1,.2) node[curve, label={0:2}] (E) {};
        \end{scope};
        \draw (C1) edge (E)  (C2) edge (E)      (P1) edge (P2)       ;
      \end{tikzpicture}
    \item 
      \begin{tikzpicture}[baseline=(C2)]
        \draw
        (0,0) node[curve, double] (P1) {}
        +(150:.6) node[int]  {$\infty$} edge (P1)
        +(210:.6) node[int]  {$0$}  edge (P1)
        (1,0) node[curve] (P2) {}
        +(0:.6) node[int]  {$1$}  edge (P2);
        \begin{scope}[shift={(0,.8)}]
          \draw
          (0,0) node[curve, label={180:8}] (C1) {}
          (0,.4) node[curve, label={180:10}] (C2) {}
          (1,.2) node[curve, label={0:3}] (E) {};
        \end{scope};
        \draw (C1) edge (E)  (C2) edge node[int] {2} (E)  (P1) edge (P2)       ;
      \end{tikzpicture}
    \item
      \begin{tikzpicture}[baseline=(C2)]
        \draw
        (0,0) node[curve, double] (P1) {}
        +(150:.6) node[int]  {$\infty$} edge (P1)
        +(210:.6) node[int]  {$0$}  edge (P1)
        (1,0) node[curve] (P2) {}
        +(0:.6) node[int]  {$1$}  edge (P2);
        \begin{scope}[shift={(0,.8)}]
          \draw
          (0,0) node[curve, label={180:2}] (C1) {}
          (0,.4) node[curve, label={180:16}] (C2) {}
          (1,.2) node[curve, label={0:3}] (E) {};
        \end{scope};
        \draw (C1) edge (E)  (C2) edge node[int] {2} (E)  (P1) edge (P2)       ;
      \end{tikzpicture}
    \item 
      \begin{tikzpicture}[baseline=(C2)]
        \draw
        (0,0) node[curve, double] (P1) {}
        +(150:.6) node[int]  {$1$} edge (P1)
        +(210:.6) node[int]  {$0$}  edge (P1)
        (1,0) node[curve] (P2) {}
        +(0:.6) node[int]  {$\infty$}  edge (P2);
        \begin{scope}[shift={(0,.8)}]
          \draw
          (0,0) node[curve, label={180:18}] (C1) {}
          (1,0) node[curve, label={0:$i$}] (E) {};
        \end{scope};
        \draw (C1) edge node[int] {$i$} (E)  (P1) edge (P2);
        \draw (.5,-.6) node[int] {$1 \leq i \leq 14$};
      \end{tikzpicture}
      
    \item 
      \begin{tikzpicture}[baseline=(C2)]
        \draw
        (0,0) node[curve, double] (P1) {}
        +(150:.6) node[int]  {$1$} edge (P1)
        +(210:.6) node[int]  {$0$}  edge (P1)
        (1,0) node[curve] (P2) {}
        +(0:.6) node[int]  {$\infty$}  edge (P2);
        \begin{scope}[shift={(0,.8)}]
          \draw
          (0,0) node[curve, label={180:6}] (C1) {}
          (0,.4) node[curve, label={180:12}] (C2) {}
          (1,.2) node[curve, label={0:$i+j$}] (E) {};
        \end{scope};
        \draw
        (C1) edge node[int] {$j$} (E)
        (C2) edge node[int] {$i$} (E)
        (P1) edge (P2);
        \draw (.5,-.6) node[int] {$1 \leq i \leq 9, 1 \leq j \leq 4$};
      \end{tikzpicture}

    \item 
      \begin{tikzpicture}[baseline=(C2)]
        \draw
        (0,0) node[curve, double] (P1) {}
        +(150:.6) node[int]  {$1$} edge (P1)
        +(210:.6) node[int]  {$0$}  edge (P1)
        (1,0) node[curve] (P2) {}
        +(0:.6) node[int]  {$\infty$}  edge (P2);
        \begin{scope}[shift={(0,.8)}]
          \draw
          (0, 0) node[curve, label={180:6}] (C1) {}
          (0,.4) node[curve, label={180:6}] (C2) {}
          (0,.8) node[curve, label={180:6}] (C3) {}
          (1,.4) node[curve, label={0:$i+j+k$}] (E) {};
        \end{scope};
        \draw
        (C1) edge node[int] {$k$} (E)
        (C2) edge node[int] {$j$} (E)
        (C3) edge node[int] {$i$} (E)
        (P1) edge (P2);
        \draw (.5,-.6) node[int] {$1 \leq i,j,k \leq 4$};
      \end{tikzpicture}
    \end{enumerate}
  \end{multicols}
\end{proposition}
\begin{proof}
  Consider the finite map $\br \from \overline{\orb H}_{4,6} \to \widetilde{\orb M}_{0,18}$ that sends a branched cover to the branch points.
  Under this map, the preimage of $\orb M_{0,18}$ is $\orb H_{4,6}$.
  So it suffices to prove the statement with $\gamma = \br \circ \alpha$ instead of $\alpha$ and $\widetilde{\orb M}_{0,18}$ instead of $\overline{\orb H}_{4,6}$.
  Notice that $\gamma$ sends $(\phi \from \orb P \to \orb Z)$ to the stabilization of $(P, \phi^{-1}(\infty))$.

  Assume that $\orb B \subset \orb Q$ is a boundary divisor satisfying the two conditions.
  Let $\phi \from \orb P \to \orb Z$ be a generic point of $\orb B$.
  The dual graph of $Z$ has the following possibilities:
  \begin{center}
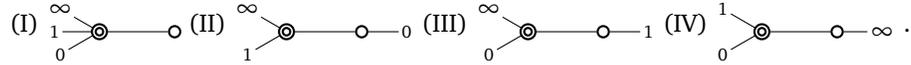
    \begin{inparaenum}[(I)]
    \item
      \begin{tikzpicture}[baseline=(P1)]
        \begin{scope}[shift={(0,0)}]
          \draw
          (0,0) node[curve, double] (P1) {}
          +(150:.6) node[int]  {$\infty$} edge (P1)
          +(180:.6) node[int]  {$1$}  edge (P1)
          +(210:.6) node[int]  {$0$}  edge (P1)
          (1,0) node[curve] (P2) {} edge (P1);
        \end{scope};
      \end{tikzpicture}
    \item
      \begin{tikzpicture}[baseline=(P1)]
        \begin{scope}[shift={(2,0)}]
          \draw
          (0,0) node[curve, double] (P1) {}
          +(150:.6) node[int]  {$\infty$} edge (P1)
          +(210:.6) node[int]  {$1$}  edge (P1)
          (1,0) node[curve] (P2) {} edge (P1)
          +(0:.6) node[int]  {$0$}  edge (P2);
        \end{scope};
      \end{tikzpicture}
    \item
      \begin{tikzpicture}[baseline=(P1)]
        \begin{scope}[shift={(4,0)}]
          \draw
          (0,0) node[curve, double] (P1) {}
          +(150:.6) node[int]  {$\infty$} edge (P1)
          +(210:.6) node[int]  {$0$}  edge (P1)
          (1,0) node[curve] (P2) {} edge (P1)
          +(0:.6) node[int]  {$1$}  edge (P2);
        \end{scope};
      \end{tikzpicture}
    \item
      \begin{tikzpicture}[baseline=(P1)]
        \begin{scope}[shift={(6,0)}]
          \draw
          (0,0) node[curve, double] (P1) {}
          +(150:.6) node[int]  {$1$} edge (P1)
          +(210:.6) node[int]  {$0$}  edge (P1)
          (1,0) node[curve] (P2) {} edge (P1)
          +(0:.6) node[int]  {$\infty$}  edge (P2);
        \end{scope};  
      \end{tikzpicture}
    \end{inparaenum}.
  \end{center}

  Let $M \subset Z$ be the main component, $T \subset Z$ the tail and set $t = M \cap T$.

  Suppose $Z$ has the form (I), (II), or (III).
  Let $E \subset P$ be a component over $T$ that has at $s$ points over $t$ where $s \geq 2$.
  Since $\gamma(\phi)$ does not lie in ${\orb M}_{0,18}$, such a component must exist.
  The contribution of $E$ towards the moduli of $\gamma(\phi)$ is due to $(E, \phi^{-1}(t))$, whose dimension is bounded above by $\max(0, s-3)$.
  The contribution of $E$ towards the moduli of $\phi$ is due to the branch points of $E \to T$.
  Let $e$ be the degree and $b$ the number of branch points of $E \to T$ away from $t$ (counted without multiplicity).
  Then $b$ equals $e+s-2$ in case (I), $e/2+s-1$ in case (II), and $e/3+s-1$ in case (III).
  Since $\gamma$ is generically finite on $\orb B$, we must have $b - \dim \Aut(T,t) = b-2 \leq \max(0,s-3)$.
  The last inequality implies that $(s,e) = (2,2)$ in cases (I) and (II), and $(s,e) = (2,3)$ in case (III).
  We now show that all other components of $P$ over $T$ are redundant.
  Suppose $E' \subset P$ is a non-redundant component over $T$ different from $E$.
  This means that $E' \to T$ has a branch point away from $t$ and the marked point (which is present only in cases (II) and (II)). 
  Composing $E' \to T$ with an automorphism of $T$ that fixes $t$ and the marked point (if any) gives another $\phi$ with the same $\gamma(\phi)$.
  Since there is a positive dimensional choice of such automorphisms and $\alpha$ is generically finite on $\orb B$, such $E'$ cannot exist.
  We now turn to the picture of $P$ over $M$.
  Since $s = 2$, the curve $P$ has two components over $M$.
  We also know the ramification profile over $0$, $1$, $\infty$, and $t$.
  This information restricts the degrees of the two components modulo 6: in case (I), they must both be $0 \pmod 6$; in case (II), they must both be $3 \pmod 6$; and in case (III), they must be $4$ and $2 \pmod 6$.
  Taking these possibilities gives the pictures (1)--(5).
  
  Suppose $Z$ has the form (IV).
  By the same argument as above, $P$ can have at most one non-redundant component over $T$.
  On the other side, we see from the ramification profile over $0$ and $1$ that the components of $P$ over $M$ have degree divisible by 6.
  We get the three possibilities (6), (7), or (8) corresponding to whether $P$ has 1, 2, or 3 components over $M$.
\end{proof}

The next step is to identify the images in $\overline{\orb M}_6$ of the boundary divisors of the form listed in \autoref{prop:not_contracted_1}.
Recall that the map $\orb Q \to \overline{\orb M}_6$ factors through the stabilization map $\orb Q \to \Maps(\widetilde{\stack M}_{0,4})$, where $\Maps(\widetilde{\stack M}_{0,4})$ is the Abramovich--Vistoli space of twisted stable maps.
The \emph{type} of a divisor refers to the dual graph of its generic point as enumerated in \autoref{prop:not_contracted_1}.
The flavor of the analysis in cases (1)--(5) versus cases (6)--(8) is quite different.

\subsection{Divisors of type (1)--(5)}\label{sec:1-5}
\begin{proposition}\label{thm:1-5}
  There are $5$ irreducible components of $\overline Q \cap \Delta$ which are the images of the divisors of $\orb Q^{\rm odd}$ of type (1)--(5). 
  Their generic points correspond to one of the following stable curves:
  \begin{itemize}
  \item With the dual graph   
    \tikz[baseline={(0,-.1)}]{
      \draw 
      (0,0) node (X) [curve] {}
      (X) edge [loop right] (X);
    }
    \begin{enumerate}
    \item
      \label{15:1}
      A nodal plane quintic.
    \end{enumerate}

  \item With the dual graph 
    \tikz[baseline={(0,-.1)}]{
      \draw 
      (0,0) node (X) [curve, label={180:$X$}] {}
      (1,0) node (Y) [curve, label={0:$Y$}] {}
      (X) edge node[ int] {$p$} (Y);
    }
    \begin{enumerate}[resume]
    \item
      \label{15:2}
      $X$ hyperelliptic of genus 3, $Y$ a plane quartic, and $p \in Y$ a hyperflex ($K_Y = 4p$).
    \end{enumerate}

  \item With the dual graph
    \tikz[baseline={(0,-.1)}]{
      \draw 
      (0,0) node (X) [curve, label={180:$X$}] {}
      (1,0) node (Y) [curve, label={0:$Y$}] {}
      (X) edge[bend left=45] node[ int] {$p$} (Y)
      (X) edge[bend right=45] node[ int] {$q$} (Y);
    }
    \begin{enumerate}[resume]
    \item
      \label{15:3}
      $X$ Maroni special of genus 4, $Y$ of genus 1, and $p, q \in X$ in a fiber of the degree 3 map $X \to \P^1$.
    \item
      \label{15:4}
      $X$ hyperelliptic of genus 3, $Y$ of genus 2, and $p \in Y$ a Weierstrass point.
    \end{enumerate}

  \item With the dual graph
    \tikz[baseline={(0,-.1)}]{
      \draw
      (0,0) node (X) [curve, label={180:$X$}] {}
      (1,0) node (Y) [curve, label={0:$Y$}] {}
      (X) edge [bend right=45] (Y)
      (X) edge  (Y)
      (X) edge [bend left=45] (Y);
    }
    \begin{enumerate}[resume]
    \item
      \label{15:5}
      $X$ hyperelliptic of genus 3, and $Y$ of genus 1.
    \end{enumerate}
  \end{itemize}
\end{proposition}
Recall that a \emph{Maroni special} curve of genus 4 is a curve that lies on a singular quadric in its canonical embedding in $\P^3$.

The rest of this section is devoted to the proof of \autoref{thm:1-5}.

The map $\orb Q \to \overline{\orb M}_6$ factors via the space $\Maps(\widetilde{\stack M}_{0,4})$ of twisted stable maps.
Let $(\widetilde\phi \from \widetilde P \to \orb Z, \orb Z \to \widetilde{\stack M}_{0,4})$ correspond to a generic point of type (1)--(5).
Under the morphism to $\Maps(\widetilde{\stack M}_{0,4})$, all the components of $\widetilde{\orb P}$ over the tail of $\orb Z$ are contracted.
The resulting twisted stable map $\phi \from \orb P \to \widetilde{\stack M}_{0,4}$ has the following form: $\orb P$ is a twisted curve with two components joined at one node; $\phi$ maps the node to a general point in case (1), to $0$ in cases (2) and (3), and to $1$ in cases (4) and (5).
In all the cases, $\phi$ is \'etale over $\infty$.
Let $(\widetilde S, \widetilde C) \to \orb P$ be the pullback of the universal family of 4-pointed rational curves.
Let $\orb P \to P$ be the coarse space at the 18 points $\phi^{-1}(\infty)$ and let $f \from (S, C) \to P$ be the family obtained from $(\widetilde S, \widetilde C)$ by the blow-down construction as in \eqref{eqn:basic_bijection} on page~\pageref{eqn:basic_bijection}.
Then $S \to P$ is a $\P^1$ bundle and $C \to P$ is simply branched over 18 smooth points.

Every $\P^1$-bundle over $P$ is the projectivization of a vector bundle (see, for example, \cite{pom:13}).
It is easy to check that vector bundles on $P$ split as direct sums of line bundles and line bundles on $P$ have integral degree.
Therefore, $S = \P V$ for some vector bundle $V$ on $P$ of rank two.
The degree of $V$ (modulo 2) is well-defined and determines whether $(S, C)$ comes from $\orb Q^{\rm odd}$ or $\orb Q^{\rm even}$.
The normalization of $P$ is the disjoint union of two orbicurves $P_1$ and $P_2$, both isomorphic to $\P^1(\sqrt[r]0)$.
The number $r$ is the order of the orbinode of $P$.
Since $\phi \from P \to \widetilde{\stack M}_{0,4}$ is representable, the possible values for $r$ are $1$ and $2$ in case (1), $2$ and $4$ in cases (2) and (3), and $3$ in cases (4) and (5).
Set $V_i = V |_{P_i}$ and $C_i = f^{-1}(P_i) \subset \P V_i$.
The number of branch points of $C_i \to P_i$ is $\deg \phi|_{P_i}$ and $C_i \to P_i$ is \'etale over $0$.
Let $[C_i] = 4 \sigma_i + m_i F$, where $\sigma_i \subset \P V_i$ is the class of the directrix. 
Using the description of curves in $\P^1$-bundles over $\P^1(\sqrt[r]0)$ from \autoref{sec:orbiscrolls} (\autoref{prop:a bound} and \autoref{prop:orbiscroll_4}), we can list the possibilities for $V_i$ and $m_i$.
These are enumerated in \autoref{table:1to5}.
\begin{table}
  \centering
  \rowcolors{2}{gray!15}{white}
  \begin{tabular}{lllllllll}
    \hline
    \rowcolor{gray!25}
    Number & Type & $r$ &$V$ & $m_1$ & $m_2$ & $g(C_1)$ & $g(C_2)$\\
    \hline
    1 & 1 & 1 & $(0,0), (0,1)$ & 2 & 3 & $3$ & $0^*$ \\
    2 & 1 & 1 & $(0,1), (0,0)$ & 4 & 1 & $3$ & $0$ \\
    3 & 1 & 1 & $(0,2), (0,-1)$ & 6 & 3 & $3^*$ & $0^*$ \\
    4 & 1 & 2 & $(0,1/2), (0,1/2)$ & 3 & 2 & $4$ & $1$ \\
    5 & 2 & 2 & $(0,1/2), (0,-3/2)$ & 5/2 & 9/2 & $2$ & $2^*$ \\
    6 & 2 & 2 & $(0,1/2), (0,1/2)$ & 5/2 & 5/2 & $2$ & $2$ \\
    7 & 2 & 4 & $(0,1/4), (0,3/4)$ & 2 & 3 & $3$ & $3$ \\
    8 & 3 & 2 & $(0,1/2), (0,1/2)$ & 7/2 & 3/2 & $5$ & $-1^*$ \\
    9 & 3 & 4 & $(0,3/4), (0,1/4)$ & 4 & 1 & $6$ & $0$ \\
    10 & 3 & 4 & $(0,5/4), (0,-1/4)$ & 5 & 1 & $6$ & $0$ \\
    11 & 4 & 3 & $(0,1/3), (0,-4/3)$ & 7/3 & 4 & $3$ & $2^*$ \\
    12 & 4 & 3 & $(0,1/3), (0,2/3)$ & 7/3 & 8/3 & $3$ & $2$ \\
    13 & 4 & 3 & $(0,2/3), (0,1/3)$ & 3 & 2 & $3$ & $2$ \\
    14 & 4 & 3 & $(0,5/3), (0,-2/3)$ & 5 & 8/3 & $3^*$ & $2$ \\
    15 & 5 & 3 & $(0,2/3), (0,1/3)$ & 4 & 1 & $6$ & $-1^*$ \\
    16 & 5 & 3 & $(0,4/3), (0,-1/3)$ & 16/3 & 1 & $6$ & $-1^*$ \\
    \hline
  \end{tabular}
  \caption{Possibilities for the divisors of type (1)--(5)}
  \label{table:1to5}
\end{table}
An asterisk in front of the (arithmetic) genus means that the curve is disconnected.
In these disconnected cases, it is the disjoint union $\sigma \sqcup D$, where $D$ is in the linear system $3 \sigma + m_i F$.
The notation $(0,a_1),(0,a_2)$ represents the vector bundle $\O \oplus L$, where $L$ is the line bundle on $P$ whose restriction to $P_i$ is $\O(a_i)$.

We must identify in classical terms (as in \autoref{thm:1-5}) the curves $C_i$ appearing in \autoref{table:1to5}.
Let $C$ be a general curve in the linear system $4 \sigma + mF$ on $\F_a$ for a fractional $a$.
Let $X = \coarse{\F_a}$ and let $\hat X \to X$ be the minimal resolution of singularities.
Denote also by $C$ the proper transform of $C \subset X$ in $\hat X$.
From \autoref{prop:orbisingularities}, we can explicitly describe the pair $(\hat X, C)$.
By successively contracting exceptional curves on $\hat X$, we then transform $(\hat X, C)$ into a pair where the surface is a minimal rational surface.
We describe these modifications diagrammatically using the dual graph of the curves involved, namely  the components of the fiber of $\hat X \to \P^1$ over $0$, and the proper transforms in $\hat X$ of the directrix $\sigma$ and the original curve $C$.
We draw the components of $\hat X \to \P^1$ over $0$ in the top row, and $\sigma$ and $C$ in the bottom row.
We label a vertex by the self-intersection of the corresponding curve and an edge by the intersection multiplicity of the corresponding intersection.
We represent coincident intersections by a 2-cell.
The edges emanating from $C$ are in the same order before and after.

We can read-off the classical descriptions in \autoref{thm:1-5} from the resulting diagrams.
For example, diagram~\ref{F25/2} implies that a curve of type $4 \sigma + (5/2)F$ on $\F_{1/2}$ is of genus 2; it has three points on the fiber over $0$, namely $\sigma(0)$ (the leftmost edge), $\tau(0)$ (the rightmost edge), and $x$ (the middle edge), of which $\sigma(0)$ and $x$ are hyperelliptic conjugates.
Likewise, diagram~\ref{F23} implies that a curve of type $4 \sigma + 3F$ on $\F_{1/2}$ is Maroni special of genus 4 and its two points over $0$ lie on a fiber of the unique map $C \to \P^1$ of degree 3.
We leave the remaining such interpretations to the reader.

\label{resolution_contraction_diagrams_1}
\begin{multicols}{2}
  \begin{enumerate}[series=rc, label=\arabic*., ref=\arabic*]
  \item $4 \sigma + 2F$ on $\F_{1/2}$:
    \begin{equation*}\label{F22}
      \begin{tikzpicture}[xscale=.75, yscale=.6, baseline={(0,-.5)}]
        \draw 
        (1,-1) node[curve, label={-90:$\sigma$}, label={180:$-1$}] (S) {}
        (1,0) node[curve, label=$-2$] (E1) {}
        (2,0) node[curve, label=$-1$] (F) {}
        (3,0) node[curve, label=$-2$] (E2) {}
        (2,-1) node[curve, label={-90:$C$}] (C) {}
        (S) edge (E1) (E1) edge (F) (F) edge (E2)
        (C) edge[bend right=30] (F)
        (C) edge[bend left=30] (F)
        ;
      \end{tikzpicture}
      \leadsto
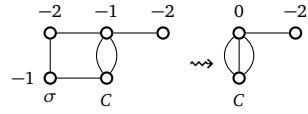
      \begin{tikzpicture}[xscale=.75, yscale=.6, baseline={(0,-.5)}]
        \draw 
        (2,0) node[curve, label=$0$] (F) {}
        (3,0) node[curve, label=$-2$] (E2) {}
        (2,-1) node[curve, label={-90:$C$}] (C) {}
        (F) edge (E2)
        (C) edge[bend right=30] (F)
        (C) edge[bend left=30] (F)
        ;
      \end{tikzpicture}
    \end{equation*}
  \item $4 \sigma + (5/2) F$ on $\F_{1/2}$:
    
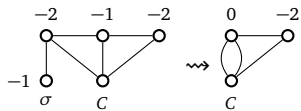
\begin{equation*}\label{F25/2}
      \begin{tikzpicture}[xscale=.75, yscale=.6, baseline={(0,-.5)}]
        \draw 
        (1,-1) node[curve, label={-90:$\sigma$}, label={180:$-1$}] (S) {}
        (1,0) node[curve, label=$-2$] (E1) {}
        (2,0) node[curve, label=$-1$] (F) {}
        (3,0) node[curve, label=$-2$] (E2) {}
        (2,-1) node[curve, label={-90:$C$}] (C) {}
        (E1) edge (F) (F) edge (E2)
        (S) edge (E1)  
        (C) edge (E1) (C) edge (F) (C) edge (E2)
        ;
      \end{tikzpicture}
      \leadsto
      \begin{tikzpicture}[xscale=.75, yscale=.6, baseline={(0,-.5)}]
        \draw 
        (2,0) node[curve, label=$0$] (F) {}
        (3,0) node[curve, label=$-2$] (E2) {}
        (2,-1) node[curve, label={-90:$C$}] (C) {}
        (F) edge (E2)
        (C) edge[bend right=30] (F) (C) edge[bend left=30] (F) (C) edge (E2)
        ;
      \end{tikzpicture}
    \end{equation*}
  \item $4 \sigma + 3F$ on $\F_{1/2}$:
    \begin{equation*}\label{F23}
      \begin{tikzpicture}[xscale=.75, yscale=.6, baseline={(0,-.5)}]
        \draw 
        (1,-1) node[curve, label={-90:$\sigma$}, label={180:$-1$}] (S) {}
        (1,0) node[curve, label=$-2$] (E1) {}
        (2,0) node[curve, label=$-1$] (F) {}
        (3,0) node[curve, label=$-2$] (E2) {}
        (2,-1) node[curve, label={-90:$C$}] (C) {}
        (E1) edge (F) (F) edge (E2)
        (S) edge (E1)  
        (C) edge (S) (C) edge[bend left=30] (F) (C) edge[bend right=30] (F)
        ;
      \end{tikzpicture}
      \leadsto
      \begin{tikzpicture}[xscale=.75, yscale=.6, baseline={(0,-.5)}]
        \draw 
        (2,0) node[curve, label=$0$] (F) {}
        (3,0) node[curve, label=$-2$] (E2) {}
        (2,-1) node[curve, label={-90:$C$}] (C) {}
        (F) edge (E2)
        (C) edge[bend left=45] (F) (C) edge[bend right=45] (F) (C) edge (F)
        ;
      \end{tikzpicture}
    \end{equation*}
  \item $4 \sigma + (7/2)F$ on $\F_{1/2}$:
    
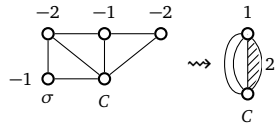
\begin{equation*}\label{F27/2}
      \begin{tikzpicture}[xscale=.75, yscale=.6, baseline={(0,-.5)}]
        \draw 
        (1,-1) node[curve, label={-90:$\sigma$}, label={180:$-1$}] (S) {}
        (1,0) node[curve, label=$-2$] (E1) {}
        (2,0) node[curve, label=$-1$] (F) {}
        (3,0) node[curve, label=$-2$] (E2) {}
        (2,-1) node[curve, label={-90:$C$}] (C) {}
        (E1) edge (F) (F) edge (E2)
        (S) edge (E1)  
        (C) edge (E1) (C) edge (F) (C) edge (E2) (C) edge (S)
        ;      
      \end{tikzpicture}
      \leadsto
      \begin{tikzpicture}[xscale=.75, yscale=.8, baseline={(0,-.5)}]
        \draw
        (1,0) node[curve, label=$1$] (E1) {}
        (1,-1) node[curve, label={-90:$C$}] (C) {}
        (C) edge [bend left=90] (E1)
        (C) edge [bend left=45] (E1)
        (C) edge (E1)
        (C) edge [bend right=45] node[int, right] {2} (E1)
        ;
        \fill [pattern=north east lines]
        (C.center) to [bend right=45] (E1.center) -- cycle;
      \end{tikzpicture}
    \end{equation*}

  \item $4 \sigma + 2F$ on $\F_{1/3}$:
    
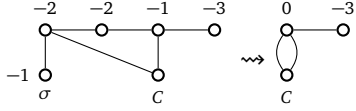
\begin{equation*}\label{F32}
      \begin{tikzpicture}[xscale=.75, yscale=.6, baseline={(0,-.5)}]
        \draw 
        (1,-1) node[curve, label={-90:$\sigma$}, label={180:$-1$}] (S) {}
        (1,0) node[curve, label=$-2$] (E1) {}
        (2,0) node[curve, label=$-2$] (E2) {}
        (3,0) node[curve, label=$-1$] (F) {}
        (4,0) node[curve, label=$-3$] (E3) {}
        (3,-1) node[curve, label={-90:$C$}] (C) {}
        (E1) edge (E2) (E2) edge (F) (F) edge (E3)
        (S) edge (E1) 
        (C) edge (E1) (C) edge (F)
        ;
      \end{tikzpicture}
      \leadsto
      \begin{tikzpicture}[xscale=.75, yscale=.6, baseline={(0,-.5)}]
        \draw 
        (3,0) node[curve, label=$0$] (F) {}
        (4,0) node[curve, label=$-3$] (E3) {}
        (3,-1) node[curve, label={-90:$C$}] (C) {}
        (F) edge (E3)
        (C) edge[bend right=30] (F) (C) edge[bend left=30] (F)
        ;
      \end{tikzpicture}
    \end{equation*}

  \item $4 \sigma + (7/3)F$ on $\F_{1/3}$:
    
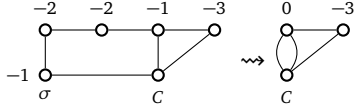
\begin{equation*}\label{F13}
      \begin{tikzpicture}[xscale=.75, yscale=.6, baseline={(0,-.5)}]
        \draw 
        (1,-1) node[curve, label={-90:$\sigma$}, label={180:$-1$}] (S) {}
        (1,0) node[curve, label=$-2$] (E1) {}
        (2,0) node[curve, label=$-2$] (E2) {}
        (3,0) node[curve, label=$-1$] (F) {}
        (4,0) node[curve, label=$-3$] (E3) {}
        (3,-1) node[curve, label={-90:$C$}] (C) {}
        (E1) edge (E2) (E2) edge (F) (F) edge (E3)
        (S) edge (E1) 
        (C) edge (S) (C) edge (F) (C) edge (E3)
        ;
      \end{tikzpicture}
      \leadsto
      \begin{tikzpicture}[xscale=.75, yscale=.6, baseline={(0,-.5)}]
        \draw 
        (3,0) node[curve, label=$0$] (F) {}
        (4,0) node[curve, label=$-3$] (E3) {}
        (3,-1) node[curve, label={-90:$C$}] (C) {}
        (F) edge (E3)
        (C) edge[bend left=30] (F) (C) edge[bend right=30](F) (C) edge (E3)
        ;
      \end{tikzpicture}
    \end{equation*}

  \item $4 \sigma + (8/3)F$ on $\F_{2/3}$:
    
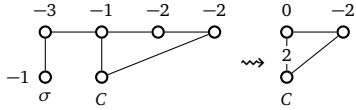
\begin{equation*}\label{F38/3}
      \begin{tikzpicture}[xscale=.75, yscale=.6, baseline={(0,-.5)}]
        \draw 
        (1,-1) node[curve, label={-90:$\sigma$}, label={180:$-1$}] (S) {}
        (1,0) node[curve, label=$-3$] (E1) {}
        (2,0) node[curve, label=$-1$] (F) {}
        (3,0) node[curve, label=$-2$] (E2) {}
        (4,0) node[curve, label=$-2$] (E3) {}
        (2,-1) node[curve, label={-90:$C$}] (C) {}
        (E1) edge (F) (F) edge (E2) (E2) edge (E3)
        (S) edge (E1) 
        (C) edge (F) (C) edge (E3)
        ;
      \end{tikzpicture}
      \leadsto
      \begin{tikzpicture}[xscale=.75, yscale=.6, baseline={(0,-.5)}]
        \draw 
        (3,0) node[curve, label=$0$] (E2) {}
        (4,0) node[curve, label=$-2$] (E3) {}
        (3,-1) node[curve, label={-90:$C$}] (C) {}
        (E2) edge (E3)
        (C) edge node[int]{2} (E2) (C) edge (E3) 
        ;
      \end{tikzpicture}
    \end{equation*}

  \item $4 \sigma + 3F$ on $\F_{2/3}$:
    
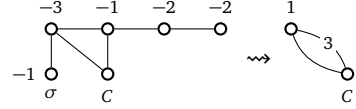
\begin{equation*}\label{F33}
      \begin{tikzpicture}[xscale=.75, yscale=.6, baseline={(0,-.5)}]
        \draw 
        (1,-1) node[curve, label={-90:$\sigma$}, label={180:$-1$}] (S) {}
        (1,0) node[curve, label=$-3$] (E1) {}
        (2,0) node[curve, label=$-1$] (F) {}
        (3,0) node[curve, label=$-2$] (E2) {}
        (4,0) node[curve, label=$-2$] (E3) {}
        (2,-1) node[curve, label={-90:$C$}] (C) {}
        (E1) edge (F) (F) edge (E2) (E2) edge (E3)
        (S) edge (E1) 
        (C) edge (E1) (C) edge (F) 
        ;
      \end{tikzpicture}
      \leadsto
      \begin{tikzpicture}[xscale=.75, yscale=.6, baseline={(0,-.5)}]
        \draw 
        (1,0) node[curve, label=$1$] (E1) {}
        (2,-1) node[curve, label={-90:$C$}] (C) {}
        (C) edge[bend right=30] node[int]{3} (E1) (C) edge[bend left=30] (E1) 
        ;
      \end{tikzpicture}
    \end{equation*}

  \item $4 \sigma + 2F$ on $\F_{1/4}$:
    
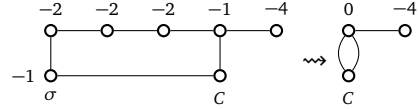
\begin{equation*}\label{F14}
      \begin{tikzpicture}[xscale=.75, yscale=.6, baseline={(0,-.5)}]
        \draw 
        (1,-1) node[curve, label={-90:$\sigma$}, label={180:$-1$}] (S) {}
        (1,0) node[curve, label=$-2$] (E1) {}
        (2,0) node[curve, label=$-2$] (E2) {}
        (3,0) node[curve, label=$-2$] (E3) {}
        (4,0) node[curve, label=$-1$] (F) {}
        (5,0) node[curve, label=$-4$] (E4) {}
        (4,-1) node[curve, label={-90:$C$}] (C) {}
        (E1) edge (E2) (E2) edge (E3) (E3) edge (F) (F) edge (E4)
        (S) edge (E1)
        (C) edge (S) (C) edge (F)
        ;
      \end{tikzpicture}
      \leadsto
      \begin{tikzpicture}[xscale=.75, yscale=.6, baseline={(0,-.5)}]
        \draw 
        (3,0) node[curve, label=$0$] (E3) {}
        (4,0) node[curve, label=$-4$] (E4) {}
        (3,-1) node[curve, label={-90:$C$}] (C) {}
        (E3) edge (E4)
        (C) edge [bend left=30] (E3) (C) edge [bend right=30] (E3)
        ;
      \end{tikzpicture}
    \end{equation*}

  \item $4 \sigma + 3F$ on $\F_{3/4}$:
    
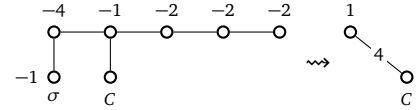
\begin{equation*}\label{F34}
      \begin{tikzpicture}[xscale=.75, yscale=.6, baseline={(0,-.5)}]
        \draw 
        (1,-1) node[curve, label={-90:$\sigma$}, label={180:$-1$}] (S) {}
        (1,0) node[curve, label=$-4$] (E1) {}
        (2,0) node[curve, label=$-1$] (F) {}
        (3,0) node[curve, label=$-2$] (E2) {}
        (4,0) node[curve, label=$-2$] (E3) {}
        (5,0) node[curve, label=$-2$] (E4) {}
        (2,-1) node[curve, label={-90:$C$}] (C) {}
        (E1) edge (F) (F) edge (E2) (E2) edge (E3) (E3) edge (E4)
        (S) edge (E1)
        (C) edge (F)
        ;
      \end{tikzpicture}
      \leadsto
      \begin{tikzpicture}[xscale=.75, yscale=.6, baseline={(0,-.5)}]
        \draw 
        (1,0) node[curve, label=$1$] (E1) {}
        (2,-1) node[curve, label={-90:$C$}] (C) {}
        (C) edge node[int] {4} (E1)
        ;
      \end{tikzpicture}
    \end{equation*}
  \end{enumerate}
\end{multicols}

We similarly analyze the curves $\sigma \sqcup C$ where $C$ is of type $3 \sigma + mF$.
\label{resolution_contraction_diagrams_2}
\begin{multicols}{2}
  \begin{enumerate}[resume=rc, label=\arabic*., ref=\arabic*]

  \item $3 \sigma + (9/2) F$ on $\F_{3/2}$:
    \begin{equation*}\label{F29/2}
      \begin{tikzpicture}[xscale=.75, yscale=.6, baseline={(0,-.5)}]
        \draw 
        (1,-1) node[curve, label={-90:$\sigma$}, label={180:$-2$}] (S) {}
        (1,0) node[curve, label=$-2$] (E1) {}
        (2,0) node[curve, label=$-1$] (F) {}
        (3,0) node[curve, label=$-2$] (E2) {}
        (2,-1) node[curve, label={-90:$C$}] (C) {}
        (E1) edge (F) (F) edge (E2)
        (S) edge (E1)  
        (C) edge (F) (C) edge (E2)
        ;
      \end{tikzpicture}
      \leadsto
      \begin{tikzpicture}[xscale=.75, yscale=.6, baseline={(0,-.5)}]
        \draw 
        (3,0) node[curve, label=$1$] (E2) {}
        (2,-1) node[curve, label={-90:$C$}] (C) {}
        (C) edge[bend left=30] node[int]{3} (E2) (C) edge[bend right=30] (E2)
        ;
      \end{tikzpicture}
    \end{equation*}

  \item $3 \sigma + 4F$ on $\F_{4/3}$:
    
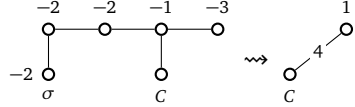
\begin{equation*}\label{F43}
      \begin{tikzpicture}[xscale=.75, yscale=.6, baseline={(0,-.5)}]
        \draw 
        (1,-1) node[curve, label={-90:$\sigma$}, label={180:$-2$}] (S) {}
        (1,0) node[curve, label=$-2$] (E1) {}
        (2,0) node[curve, label=$-2$] (E2) {}
        (3,0) node[curve, label=$-1$] (F) {}
        (4,0) node[curve, label=$-3$] (E3) {}
        (3,-1) node[curve, label={-90:$C$}] (C) {}
        (E1) edge (E2) (E2) edge (F) (F) edge (E3)
        (S) edge (E1) 
        (C) edge (F)
        ;
      \end{tikzpicture}
      \leadsto
      \begin{tikzpicture}[xscale=.75, yscale=.6, baseline={(0,-.5)}]
        \draw 
        (4,0) node[curve, label=$1$] (E3) {}
        (3,-1) node[curve, label={-90:$C$}] (C) {}
        (C) edge node[int]{4} (E3)
        ;
      \end{tikzpicture}
    \end{equation*}

  \item $3 \sigma + 5F$ on $\F_{5/3}$:
    \begin{equation*}\label{F53}
      \begin{tikzpicture}[xscale=.75, yscale=.6, baseline={(0,-.5)}]
        \draw 
        (1,-1) node[curve, label={-90:$\sigma$}, label={180:$-2$}] (S) {}
        (1,0) node[curve, label=$-3$] (E1) {}
        (2,0) node[curve, label=$-1$] (F) {}
        (3,0) node[curve, label=$-2$] (E2) {}
        (4,0) node[curve, label=$-2$] (E3) {}
        (2,-1) node[curve, label={-90:$C$}] (C) {}
        (E1) edge (F) (F) edge (E2) (E2) edge (E3)
        (S) edge (E1) 
        (C) edge (F)
        ;
      \end{tikzpicture}
      \leadsto
      \begin{tikzpicture}[xscale=.75, yscale=.6, baseline={(0,-.5)}]
        \draw 
        (1,-1) node[curve, label={-90:$\sigma$}, label={180:$-2$}] (S) {}
        (1,0) node[curve, label=$0$] (E1) {}
        (2,-1) node[curve, label={-90:$C$}] (C) {}
        (S) edge (E1) 
        (C) edge node[int]{3} (E1)
        ;
      \end{tikzpicture}
    \end{equation*}
  \end{enumerate}
\end{multicols}

The proof of \autoref{thm:1-5} now follows from \autoref{table:1to5} and the diagrams above. 
We discard rows 9, 10, 15, and 16 of \autoref{table:1to5} since they map to the interior of $\overline{\orb M}_6$.
For the remaining ones, we read off the description of $C_1$ and $C_2$ from the corresponding diagram and get $C = C_1 \cup C_2$.
While attaching $(C_1, S_1)$ to $(C_2, S_2)$, we must take into account whether the directrices of the $S_i$ meet each other or whether the directrix of one meets a co-directrix of the other.
From $S = \P V$ and $V = \O \oplus L$, we see that if the restrictions of $L$ to $\orb P_i$ have the same sign, then the two directrices meet and if they have the opposite sign, then a directrix meets a co-directrix.
Proceeding in this way, we get that rows 2, 5, 6, 13, and 14  map to loci in $\overline {\orb M}_6$ of dimension at most 10, and hence do not give divisors of $\overline Q$.
Row 1 gives divisor~\eqref{15:5}; rows 3 and 4 give divisor~\eqref{15:3}; rows 7 and 11 give divisor~\eqref{15:2}; row 8 gives divisor~\eqref{15:1}; and row 12 gives divisor~\eqref{15:4}.
The proof of \autoref{thm:1-5} is now complete.

\subsection{Towards divisors of type (6)--(8)}
To handle boundary divisors of type (6)--(8), we need to do some preparatory work.
First, we need to understand the tetragonal curves arising from finite maps to $\widetilde{\stack M}_{0,4}$ ramified over $\infty$.
Second, we need to understand the tetragonal curves arising from maps that contract the domain to the point $\infty \in \widetilde{\stack M}_{0,4}$.
Third, we need to understand the parity of the curve obtained by putting these together.

First, we consider finite maps to $\widetilde{\stack M}_{0,4}$ possibly ramified over $\infty$.
Away from the points mapping to $\infty$, a map to $\widetilde{\stack M}_{0,4}$ gives a fiberwise degree 4 curve in a $\P^1$-bundle.
The question that remains is then local around the points that map to $\infty$.
Let $D$ be a disk, set $\orb D = D(\sqrt[r]{0})$, and let $\phi \from \orb D \to \widetilde{\stack M}_{0,4}$ be a representable finite map that sends $0$ to $\infty$.
Let $n$ be the local degree of the map $D \to \P^1$ of the underlying coarse spaces.
Let $f \from (\orb S, \orb C) \to \orb D$ be the pullback of the universal family of 4-pointed rational curves.
Then $(\orb S_0, \orb C_0) \cong (\P^1 \cup \P^1, \{1,2,3,4\})$, where the $\P^1$'s meet in a node and $1, 2$ lie on on component and $3,4$ lie on the other.
Let $\pi \in \sym_4$ be the image in $\Aut_\infty \widetilde{\stack M}_{0,4}$ of a generator of $\Aut_0 \orb D$.

In the following proposition, an $A_n$ singularity over a disk with uniformizer $t$ is the singularity with the formal local equation $x^2-t^{n+1}$.
Thus an $A_0$ singularity is to be interpreted as a smooth ramified double cover and an $A_{-1}$ singularity as a smooth unramified double cover.
\begin{proposition}\label{thm:An_singularities}
  With the notation above, the curve $\coarse{\orb C}$ is the normalization of $\coarse {\orb C'}$, where $\orb C'$ is a curve of fiberwise degree 4 on a $\P^1$ bundle $\orb S'$ over an orbifold disk $\orb D'$ of one of the following forms.
  \begin{enumerate}
  \item [Case 1]: $\pi$ preserves the two components of $\orb S_0$.
    Then $\orb D' = D$ and $\orb S' = \P^1 \times D'$.
    On the central fiber of $\orb S' \to \orb D'$, the curve $\orb C'$ has an $A_i$ and an $A_j$ singularity over $\orb D'$ with $i+j=n-2$.
    If $n$ is even and $i$ is even, then $\pi$ is trivial; if $n$ is even and $i$ is odd, then $\pi$ has the cycle type $(2,2)$; and if $n$ is odd then $\pi$ has the cycle type $(2)$.
  \item [Case 2]: $\pi$ switches the two components of $\orb S_0$.
    Then $\orb D' = D(\sqrt[2]{0})$ and $\orb S' = \P(\O \oplus \O(1/2))$.
    Let $u \from \widetilde{D}' \to \orb D'$ be the universal cover.
    On the central fiber of $S' \times_u \widetilde {D}' \to \widetilde D'$, the curve $\orb C' \times_u \widetilde {D}'$ has two $A_{n-1}$ singularities over $\widetilde{D}'$ that are conjugate under the natural action of $\Z_2$.
    If $n$ is even then $\pi$ has the cycle type $(2,2)$, and if $n$ is odd then $\pi$ has the cycle type $(4)$.
  \end{enumerate}
\end{proposition}
\begin{remark}\label{rem:ij}
  The apparent choice of $i$ and $j$ in the first case is not a real ambiguity.
  By an elementary transformation centered on the $A_i$ singularity, we can transform $(i,j)$ to $(i-2,j+2)$.
\end{remark}
\begin{proof}
  Since $\phi \from \orb D \to \widetilde{\stack M}_{0,4}$ is representable, the map $\Aut_0 \orb D \to  \Aut_\infty \widetilde{\stack M}_{0,4} = D_4$ is injective.
  So the order $r$ of $0 \in \orb D$ is $1, 2,$ or $4$.
  Let $\widetilde D \to \orb D$ be the universal cover.
  Set $\widetilde S = \orb S \times_{\orb D} {\widetilde D}$ and $\widetilde C = \orb C \times_{\orb D}{\widetilde D}$.
  Then $\widetilde S$ is a surface with an action of $\Z_r$ compatible with the action of $\Z_r$ on $\widetilde D$.
  The action of the generator of $\Z_r$ on the central fiber of $\widetilde S \to \widetilde D$ is given by $\pi$.
  Note that $rn$ is the local degree of the map $\widetilde D \to \widetilde{\stack M}_{0,4}$.
  Therefore, the surface $\widetilde S$ has an $A_{m-1}$ singularity at the node of the central fiber $\widetilde S_0$, where $m = rn/2$.
  We take the minimal desingularization of $\widetilde S$, successively blow down the $-1$ curves on the central fiber compatibly with the action of $\Z_r$ until we arrive at a $\P^1$-bundle, and then take the quotient by the induced $\Z_r$ action.
  The resulting surface $\orb S'$ and curve $\orb C'$ are as claimed in the proposition.

  We illustrate the process for Case 2 and odd $n$, in which case $r = 4$.
  Let $t$ be a uniformizer on $\widetilde D$.
  In suitable coordinates, $\widetilde S \to \widetilde D$ has the form
  \[ \C[x,y,t]/(xy-t^{m}) \leftarrow \C[t],\]
  where $m = 2n$.
  A generator $\zeta \in \Z_4$ acts by
  \[ \zeta \cdot t = it, \quad \zeta \cdot x = y, \quad \zeta \cdot y = -x.\]
  Let $\hat S \to \widetilde S$ be the minimal desingularization.
  Then $\hat S_0$ is a chain of $\P^1$s, say 
  \[\hat S_0 = P_0 \cup P_1 \cup \dots \cup P_n \cup \dots \cup P_{m-1} \cup P_m,\]
  where $P_i$ meets $P_{i+1}$ nodally at a point.
  Under the induced action of $\Z_4$, the generator $\zeta$ sends $P_i$ to $P_{m-i}$.
  Contract $P_0, \dots, P_{n-1}$ and $P_m, \dots, P_{n+1}$ successively, leaving only $P_n$.
  Let $\overline S$ (resp. $\overline C$) be the image of $\hat S$ (resp. $\hat C$) under the contraction.
  Then $\overline C$ has two $A_{m-1}$ singularities on $P_n$, say at $0$ and $\infty$.
  The $\Z_4$ action descends to an action on $\overline S$ and the generator exchanges $0$ and $\infty$ on $P_n$.
  Note that the $\Z_2 \subset \Z_4$ acts trivially on the central fiber.
  Let us replace $\overline S$ (resp. $\overline C$, $\overline D$) by its geometric quotient under the $\Z_2$ action.
  Then $\overline S \to \overline D$ is a $\P^1$ bundle and $\overline C \subset \overline S$ has two $A_{n-1}$ singularities on the central fiber.
  The group $\Z_2$ acts compatibly on $(\overline S, \overline C)$ and $\overline D$ and exchanges the two singularities of $C$.
  Setting $\orb S' = [\overline S / \Z_2]$, $\orb C' = [\overline C/\Z_2]$, and $\orb D' = [\overline D/\Z_2]$ gives the desired claim.

  The other cases are analogous.
\end{proof}

Second, we consider maps that contract the domain to $\infty \in \widetilde{\stack M}_{0,4}$.
Let us denote the geometric fiber of the universal family $(\widetilde{\orb S}, \widetilde{\orb C}) \to \widetilde{\stack M}_{0,4}$ over $\infty$ by
\[ (\widetilde{\orb S}, \widetilde {\orb C})_\infty = (P_A \cup P_B, \{1,2,3,4\}) \text{ where } 1, 2 \in P_A \cong \P^1 \text{ and } 3, 4 \in P_B \cong \P^1.\]
We have
\[D_4 \cong \Aut_\infty \widetilde{\stack M}_{0,4} = \Stab(\{\{1,2\}, \{3,4\}\}) \subset \sym_4.\]
Over the $BD_4 \subset \widetilde{\stack M}_{0,4}$ based at $\infty$, the universal family is given 
by 
\[ (\widetilde {\orb S}, \widetilde {\orb C})_{BD_4} = [(P_A \cup P_B, \{1,2,3,4\}) / D_4].\]
We have the natural map
\begin{equation}\label{eqn:DihedralHyperelliptic}
  [\{1,2,3,4\} / D_4] \to [\{A, B\} / \Z_2].
\end{equation}
Let $\orb P$ be a smooth connected orbifold curve and let $\phi \from \orb P \to BD_4 \subset \widetilde{\stack M}_{0,4}$ be a representable morphism, where the $BD_4$ is based at $\infty$.
Set $\widetilde{\orb C}_\phi = \widetilde {\orb C} \times_\phi \widetilde{\stack M}_{0,4}$.
From \eqref{eqn:DihedralHyperelliptic}, we see that the degree 4 cover $\widetilde{\orb C}_\phi \to \orb P$ factors as a sequence of two degree 2 covers
\begin{equation}\label{eqn:CGP}
  \widetilde{\orb C}_\phi \to \orb G_\phi \to \orb P.
\end{equation}
The two points of ${\orb G}_\phi$ over a point of $\orb P$ are identified with the two components of $\widetilde {\orb S}$ over that point.

Let us analyze this factorization from the point of view of the tetragonal-trigonal correspondence (~\autoref{prop:recillas}).
Consider the induced map $\psi \from \orb P \to B\Z_2 \subset \widetilde{\orb M}_{0,1+3}$, where the $B\Z_2$ is based at $\infty$.
It is important to distinguish between the $4$ numbered points for $\orb M_{0,4}$ and those for ${\orb M}_{0,1+3}$.
We denote the latter by $I, II, III, IV$ with the convention that the $\sym_3 = \sym_4/K$ action is given by conjugation via the identification
\[ I \leftrightarrow (13)(24),\quad II \leftrightarrow (14)(32),\quad III \leftrightarrow (12)(34),\quad IV \leftrightarrow \id.\]
Then the $\Z_2 = D_4/K$ action switches $I$ and $II$ and leaves $III$ and $IV$ fixed.
As a result, the trigonal curve $\orb D = \orb D_\psi$ of \autoref{prop:recillas} is the disjoint union
\begin{equation}\label{eqn:DEP}
  \orb D_\psi = \orb P \sqcup \orb E_\psi,
\end{equation}
where $\orb E_\psi \to \orb P$ is a double cover.
(Caution: the double cover $\orb G_\phi \to \orb P$ of \eqref{eqn:CGP} is \emph{different} from the double cover $\orb E_\psi \to \orb P$ of \eqref{eqn:DEP}).
Let $\orb L$ be the norm-trivial two-torsion line bundle on $\orb D_\psi$ corresponding to the lift $\phi \from \orb P \to \widetilde{\stack M}_{0,4}$ of $\psi \from \orb P \to \widetilde{\orb M}_{0,1+3}$.
Since $\orb L^2$ is trivial, the action of the automorphism groups of points of $\orb D_\psi$ on the fibers of $\orb L$ is either trivial or by multiplication by $-1$.
The following proposition relates the ramification of $\coarse{\orb G}_\phi \to \coarse{\orb P}$ with this action.
\begin{proposition}\label{thm:CGP_action}
  Identify $\orb P$ with its namesake connected component in $\orb D_\psi = \orb P \sqcup \orb E_\psi$.
  Let $p \in \orb P$ be a point.
  Then $\coarse{\orb G}_\phi \to \coarse{\orb P}$ is ramified over $p$ if and only if the action of $\Aut_p \orb P$ on $\orb L_p$ is nontrivial.
\end{proposition}
\begin{proof}
  Write $\orb G_\phi = \orb G$, $\widetilde{\orb C}_\phi = \widetilde{\orb C}$, and so on.
  The map $\coarse{\orb G} \to \coarse{\orb P}$ is ramified over $p$ if and only if the action of $\Aut_p \orb P$ on the fiber $\orb G_p$ is non-trivial.
  Denote the fiber of $\widetilde{\orb C} \to \orb P$ over $p$ by $\{1,2,3,4\}$, considered as a set with the action of $\Aut_p \orb P$.
  Then the fiber of $\orb D \to \orb P$ over $p$ is
  \[\{I, II, III \} = \{(13)(24), (14)(32), (12)(34)\},\]
  among which $\{(13)(24), (14)(32)\}$ comprise points of $\orb E$ and $\{(12)(34)\}$ the point of $\orb P$.
  From the proof of \autoref{prop:recillas}, we know that the two-torsion line bundle $\orb L$ on $\orb D$ corresponds to the \'etale double cover $\tau \from \widetilde{\orb D} \to \orb D$ where the fiber of $\widetilde{\orb D}$ over $p$ is $\{(12),(13), (14), (23), (24), (34)\}$.
  The action of $\Aut_p \orb P$ on $\orb L_p$ is non-trivial if and only if the action of $\Aut_p\orb P$ on $\{(12),(34)\}$ is non-trivial.
  But we can identify the $\Aut_p\orb P$ set $\{(12),(34)\}$ with the $\Aut_p\orb P$ set $\{A,B\}$, which is precisely the fiber of $\orb G \to \orb P$ over $p$.
\end{proof}

We have all the tools to determine the stable images of the divisors of $\orb Q$ of type (6)--(8), but to be able to separate $\orb Q^{\rm odd}$ from $\orb Q^{\rm even}$, we need some further work.

We need to extend the blow-down construction in \eqref{eqn:basic_bijection}, which we recall.
Let $\orb P$ be a smooth orbifold curve of with $b$ orbifold points of order 2 and let $\phi \from \orb P \to \widetilde{\stack M}_{0,4}$ be a finite map of degree $b$ such that the underlying map $\orb P \to \widetilde{\orb M}_{0,1+3}$ is representable and has ramification type $(1,1, \dots)$ over $\infty$.
Let $f \from (\widetilde S, \widetilde C) \to \orb P$ be the pullback of the universal family of 4-pointed rational curves.
Let $p \in \orb P$ be an orbifold point.
Then we have $\widetilde S_p \cong \P^1 \cup \P^1$ and the $\Z_2 = \Aut_p \orb P$ acts trivially on one $\P^1$ and by an involution on the other.
By blowing down the component with the non-trivial action and taking the coarse space, we get a family $f \from (S, C) \to P$, where $P = \coarse{\orb P}$, the map $S \to P$ is a $\P^1$ bundle, and the curve $C \subset S$ is simply branched over $P$.

Now assume that $\orb P$ is reducible and $p \in \orb P$ is a smooth orbifold point of order 2 lying on a component that is contracted to $\infty$ by $\phi$.
(We still require that the underlying map $\phi \from \orb P \to \widetilde{\orb M}_{0,1+3}$ be representable at $p$.)
Locally near $p$, the curve $\orb P$ has the form $[\spec \C[x] / \Z_2]$.
Set $U = \spec \C[x]$.
Then $\widetilde S_U \cong P_1 \cup P_2$, where $P_i \cong \P^1 \times U$ and the $P_i$ are joined along sections $s_i \from U \to P_i$ (see \autoref{fig:blowup_blowdown}).
Like the case before, the action of $\Z_2$ on the fiber $P_1|_0 \cup P_2 |_0$ must be trivial on one component, say the first, and an involution on the other.
Unlike the case before, we cannot simply blow down the $\P^1$ with the non-trivial action.
Let $\hat S_U$ be the blow up of $\widetilde S_U$ along the (non-Cartier) divisor $P_2|_0$  (see \autoref{fig:blowup_blowdown}).
Then $\hat S_U = \hat P_1 \cup P_2$, where $\hat P_1$ is the blow up of $P_1$ at the point $s_1(0) = P_1 \cap P_2|_0$, and $\hat P_1$ and $P_2$ are joined along the proper transform $\hat s_1$ of $s_1$ and $s_2$.
The proper transform of $P_1|_0$ is a $-1$ curve on the $\hat P_1$ component of $\hat S_U$.
Let $\hat S_U \to S'_U$ be the blow-down along this $-1$ curve.
Then $S'_U \to U$ is a $\P^1 \cup \P^1$-bundle with a trivial $\Z_2$ action on the central fiber $S'_0$.
Therefore the quotient $S = S'_U / \Z_2$ is a $\P^1 \cup \P^1$ bundle over the coarse space $P$ of $\orb P$ at $p$.
The image $C \subset S$ of $\widetilde C \subset \widetilde S$ is simply branched over $p \in P$ and is disjoint from the singular locus of $S$.
We call the construction of $(S,C)$ from $(\widetilde S, \widetilde C)$ the \emph{blow-up blow-down construction}.

\begin{figure}
  \centering
  \begin{tikzpicture}
    \draw[thick]
    plot coordinates {(-1.2,1) (-1,0) (-0.98,-0.1)}
    plot coordinates {(-1.2,-1) (-1,0) (-0.98,0.1)}
    plot coordinates {(0.8,-1) (1,0) (1.02,0.1)}
    plot coordinates {(0.8,1) (1,0) (1.02,-0.1)}
    plot coordinates {(0.8,-1) (-1.2, -1)}
    plot coordinates {(0.8,1) (-1.2, 1)}
    plot coordinates {(-1,0) (1, 0)}
    plot coordinates {(-0.2,1) (0,0) (0.02,-0.1)}
    plot coordinates {(-0.2,-1) (0,0) (0.02,0.1)}
    ;

    \begin{scope}[shift={(4,0)}]
      \draw[thick]
      plot coordinates {(-1.2,1) (-1,0) (-0.98,-0.1)}
      plot coordinates {(-1.2,-1) (-1,0) (-0.98,0.1)}
      plot coordinates {(0.8,-1) (1,0) (1.02,0.1)}
      plot coordinates {(0.8,1) (1,0) (1.02,-0.1)}
      plot coordinates {(0.8,-1) (-1.2, -1)}
      plot coordinates {(0.8,1) (-1.2, 1)}
      plot coordinates {(-1,0) (1, 0)}
      plot coordinates {(-0.2,-1) (0,0) (0.02,0.1)}

      plot coordinates {(-0.2,1)  (0.1,0.5) }
      ;
      \draw[thick, smooth, dashed]
      plot coordinates {(0,0)  (-0.1, .4) (0.15, 0.8) }
      ;
    \end{scope}

    \begin{scope}[shift={(8,0)}]
      \draw[thick]
      plot coordinates {(-1.2,1) (-1,0) (-0.98,-0.1)}
      plot coordinates {(-1.2,-1) (-1,0) (-0.98,0.1)}
      plot coordinates {(0.8,-1) (1,0) (1.02,0.1)}
      plot coordinates {(0.8,1) (1,0) (1.02,-0.1)}
      plot coordinates {(0.8,-1) (-1.2, -1)}
      plot coordinates {(0.8,1) (-1.2, 1)}
      plot coordinates {(-1,0) (1, 0)}
      plot coordinates {(-0.2,-1) (0,0) (0.02,0.1)}
      ;
      \draw[dashed]
      plot coordinates {(-0.2,1) (0,0) (0.02,-0.1)}
      ;
    \end{scope}

    \draw
    (0, -1.5) node {$\widetilde S_U$}
    (4, -1.5) node {$\hat S_U$}
    (8, -1.5) node {$S'_U$}
    (2, 0) node {$\longleftarrow$}
    (6, 0) node {$\longrightarrow$};
  \end{tikzpicture}
  \caption{The blow-up blow-down construction}\label{fig:blowup_blowdown}
\end{figure}
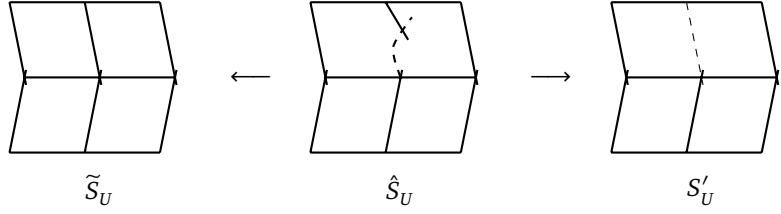

Let us verify that the blow-up blow-down construction is compatible in a one-parameter family with the blow-down construction.
This verification is local around the point $p$.
Let $\Delta$ be a DVR and $P \to \Delta$ a smooth (not necessarily proper) curve with a section $p \from \Delta \to P$.
Set $\orb P = P(\sqrt p)$.
Let $\phi \from \orb P \to \widetilde{\stack M}_{0,4}$ be a map such that the underlying map $\orb P \to \widetilde{\orb M}_{0,1+3}$ is representable.
Assume that for a generic point $t \in \Delta$, the map $\phi_t$ maps $p$ to $\infty$ and is \'etale around $p$ but $\phi_0$ contracts $\orb P_0$ to $\infty$.
Let $f \from (\widetilde S, \widetilde C) \to \orb P$ be the pullback by $\phi$ of the universal family of 4-pointed rational curves.
\begin{proposition}
  There exists a (flat) family $S \to P$ over $\Delta$ such that the generic fiber $S_t \to P_t$ is the $\P^1$ bundle obtained from $\widetilde S_t \to \widetilde P_t$ by the blow-down construction and the special fiber $S_0 \to P_0$ is the $\P^1 \cup \P^1$ bundle obtained from $\widetilde S_0 \to \widetilde P_0$ by the blow-up blow-down construction.
\end{proposition}
\begin{proof}
  We may take $\orb P = [U / \Z_2]$, where $U \to \Delta$ is a smooth curve and $\Z_2$ acts freely except along a section $p \from \Delta \to U$.
  Say $\widetilde S|_p = P_1 \cup P_2$, where $P_i \to \Delta$ are $\P^1$-bundles meeting along a section and the $\Z_2$ acts trivially on $P_2$ and by an involution on $P_1$.
  Note that $\widetilde S_U|_t$ is a smooth surface for a generic $t$ and $P_1|_t \subset \widetilde S_U|_t$ is a $-1$ curve.
  Let $\beta \from \hat S_U \to S_U$ be the blow-up along $P_2$.
  Then $\beta_t$ is an isomorphism for a generic $t \in \Delta$.
  We claim that $\beta_0$ is is the blow up of $P_2|_0$ in $\widetilde S_U|_0$.
  To check the claim, we do a local computation.
  Locally around $p(0)$, we can write $U$ as 
  \[ \spec \C[x,t],\]
  where $p$ is cut out by $x$.
  Now $\widetilde S_U \to U$ is a family of curves whose generic fiber is $\P^1$, and whose discriminant locus (where the fiber is singular) is supported on $xt = 0$.
  Furthermore, we know that the multiplicity of the discriminant along $(x = 0)$ is $1$.
  Therefore, around the node of $\widetilde S|_0$, we can write $\widetilde S_U$ as
  \[ \spec \C[x,y,z,t]/(yz - xt^n).\]
  In these coordinates, say $P_2 \subset \widetilde S_U$ is cut out by the ideal $(x,y)$.
  Direct calculation shows that the specialization of $\Bl_{(x,y)} \spec \C[x,y,z,t]/(yz - xt^n)$ at $t = 0$ is $\Bl_{(x,y)}\spec \C[x,y,z]/(yz)$, as claimed.
  Let $\hat P_1 \subset \hat S_U$ be the proper transform of $P_1$.
  Then $\hat P_1 |_t \subset \hat S_U|_t$ is a $-1$ curve for all $t$.
  Let $\hat S_U \to S'_U$ be the blow-down.
  Then the action of $\Z_2$ on $S'_U|_p$ is trivial.
  The quotient $S = S'_U / \Z_2$ with the map $S \to P$ is the required family.
\end{proof}

Let $\phi \from \orb P \to \widetilde{\stack M}_{0,4}$ be an Abramovich--Vistoli stable map arising from a generic point of a divisor in $\orb Q$ of type (6)--(8).
Then $\orb P$ has 18 smooth orbifold points of order $2$.
Let $\orb P \to P$ be the coarse space at these 18 points.
Let $f \from (\widetilde S, \widetilde C) \to \orb P$ be the pullback of the universal family of 4-pointed rational curves and let $(S, C) \to P$ be the family obtained by the blow-up blow-down construction.
Then the surface $S$ is a degeneration of $\F_0$ or $\F_1$.
The following observation lets us distinguish the two cases.
\begin{proposition}\label{thm:section_parity}
  Suppose $s \from P \to S$ is a section lying in the smooth locus of $S \to P$.
  Then the self-intersection $s^2$ is an integer.
  If it is even (resp. odd), then $S$ is a degeneration of $\F_0$ (resp. $\F_1$).
\end{proposition}
\begin{proof}
  Note that $s(P) \subset S$ is a Cartier divisor.
  Let $\orb L$ be the associated line bundle.
  Then $s^2 = \deg(s^* \orb L)$.
  Since degrees of line bundles on $P$ are integers, $s^2$ is an integer.
  Suppose $S$ is a degeneration of $\F_i$.
  Then a smoothing of $\orb L$ is a line bundle on $\F_i$ of fiberwise degree 1.
  Its self-intersection determines the parity of $i$.
\end{proof}
Such a section $s$ does not always exist, however.
For example, the $\P^1 \cup \P^1$ bundle over a contracted component of $P$ may have non-trivial monodromy that exchanges the two components.
To distinguish odd and even in these cases, we must understand the parity of the limiting theta characteristic on the associated trigonal curve.

We quickly recall the theory of limiting theta characteristic from \cite{chi:08}.
Consider a one-parameter family of smooth curves degenerating to a nodal curve $C$.
Suppose we have a theta-characteristic on this family away from the central fiber.
Then, after possibly making a base change and replacing the nodes by orbifold nodes of order two, the theta-characteristic extends uniquely to a (locally free) theta-characteristic on the central fiber.
Note that the limit theta-characteristic may not be a line bundle on $C$ itself, but on $\orb C$, where $\orb C \to C$ is an orbinodal modification.
By a \emph{limiting theta characteristic} on $C$, we mean a theta characteristic on an orbinodal modification of $C$.
Suppose $\orb L$ is a theta characteristic on $\orb C$ and $x \in \orb C$ is an orbinode.
Then $\Aut_x\orb C$ acts on $\orb L_x$ by $\pm 1$.
Suppose the action is non-trivial.
Let $\nu \from \hat {\orb C} \to \orb C$ be the normalization at $x$ and $c \from \hat {\orb C} \to {\orb C}'$ the coarse space at the two points of $\hat {\orb C}$ over $x$.
Then $c_* \nu^* \orb L$ is a theta characteristic on ${\orb C}'$ and
\begin{equation}\label{eqn:theta-orbinode}
  h^0(\orb C, \orb L) = h^0\left({\orb C}', c_* \nu^* \orb L\right).
\end{equation}
Suppose the action is trivial.
Then $\orb L$ is a pullback from the coarse space around $x$, so we may assume that $\Aut_x \orb C$ is trivial.
Let $\nu \from \hat {\orb C} \to \orb C$ be the normalization at $x$, as before, and $x_1, x_2$ the two points of $\hat{\orb C}$ over $x$.
Let $\epsilon_x$ be the two-torsion line bundle on $\orb C$ obtained by taking the trivial line bundle on $\hat{\orb C}$ and gluing the fibers over $x_1$ and $x_2$ by $-1$.
Then $\orb L \otimes \epsilon_x$ is another theta characteristic on $\orb C$, and by \cite[Theorem~2.14]{har:82}  we have
\begin{equation}\label{eqn:theta-switch}
  h^0(\orb C, \orb L \otimes \epsilon_x) = h^0(\orb C, \orb L) \pm 1.
\end{equation}

Let $\orb Z \to \widetilde{\orb M}_{0,1+3}$ be a pointed degeneration and $\psi \from \orb P \to \orb Z$ a finite cover corresponding to a generic point of a divisor of type (6), (7), or (8).
Let $f \from \orb D_\psi \to \orb P$ the corresponding \'etale triple cover.
We assume that the orders of the orbinodes of $\orb Z$ and therefore $\orb D_\psi$ are sufficiently divisible.
Therefore, we have a limiting theta characteristic $\theta$ on $\coarse{\orb D}_\psi$.
Denote by the same symbol its pullback to $\orb D_\psi$.
Note that the action on $\theta_x$ of $\Aut_x \orb D_\psi$ is trivial for all $x$ except possibly the nodes.

Let $\orb P^{\rm main} \subset \orb P$ (resp. $\orb P^{\rm tail}$) be the union of the components that lie over the main (resp. tail) component of $\orb Z$.
Denote by $\orb D^{\rm main}_\psi$ (resp. $\orb D_\psi^{\rm tail}$) the pullback of $\orb D_\psi$ to $\orb P^{\rm main}$ (resp. $\orb P^{\rm tail}$).
Then $\orb D^{\rm tail}_\psi $ is the disjoint union $\orb P^{\rm tail} \sqcup \orb E_\psi$, where $\orb E_\psi \to \orb P^{\rm tail}$ is a double cover.
\begin{proposition}\label{thm:limit_theta}
  Let $x$ be a node of the $\orb P^{\rm tail}$ component of $\orb D^{\rm tail}_\psi$.
  Then the action of $\Aut_x {\orb D}_\psi$ on $\theta_x$ is non-trivial.
\end{proposition}
\begin{proof}
  We look at the limiting relative theta characteristic on the universal family.
  Let $\orb D \to \orb Z$ be the pullback along $\orb Z \to \widetilde{\orb M}_{0,1+3}$ of the universal triple cover on $\widetilde{\orb M}_{0,1+3}$.
  Note that $\orb D$ has three components, say $\orb D_1$, $\orb E_1$, and $\orb E_2$, and the dual graph of $\orb D \to \orb Z$ is as follows
  
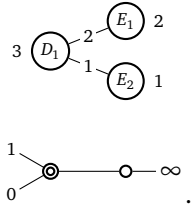
\[
  \begin{tikzpicture}
    \draw
    (0,0) node[curve, double] (P1) {}
    +(150:.6) node[int]  {$1$} edge (P1)
    +(210:.6) node[int]  {$0$}  edge (P1)
    (1,0) node[curve] (P2) {}
    +(0:.6) node[int]  {$\infty$}  edge (P2);
    \begin{scope}[shift={(0,1.6)}]
      \draw
      (0, 0) node[curve, label={180:3}] (C1) {\scriptsize $D_1$}
      (1,.4) node[curve, label={0:2}] (E1) {\scriptsize $E_1$}
      (1,-.4) node[curve, label={0:1}] (E2) {\scriptsize $E_2$};
    \end{scope};
    \draw
    (C1) edge node[int] {$2$} (E1)
    (C1) edge node[int] {$1$} (E2)
    (P1) edge (P2);
  \end{tikzpicture}.
  \]
  Let $\orb Z'$ be obtained from $\orb Z$ by taking the coarse space at $\infty$ and $\orb D'$ (resp. $\orb E_1'$, $\orb E'_2$) from $\orb D$ (resp. $\orb E_1$, $\orb E_2$) by taking the coarse space at the points over $\infty$.
  Then the cover $\orb D' \to \orb Z'$ is simply branched over $\infty$, and it is a degeneration of the cover $\orb D' \to \orb M'_{0,1+3}$ in \eqref{eqn:theta}.
  The relative dualizing sheaf of $\orb D' \to \orb Z'$ has degree $0$ on $\orb E'_2$.
  Let $\theta_{\rm rel}$ be the limiting theta characteristic on $\coarse{\orb D'}$.
  Since $x_2 = \orb E_2' \cap \orb D_1$ is the unique orbifold point on $\orb E_2'$, the action of $\Aut_{x_2}\orb E_2'$ on $\theta_{\rm rel}$ at $x_2$ must be trivial.

  The map $\psi \from \orb D_\psi \to \orb D$ maps a node $x$ on the $\orb P^{\rm tail}$ component to $x_2$.
  Therefore, the action of $\Aut_x \orb D_\psi$ on $\psi^*\theta_{\rm rel}|_x$ is trivial.
  Let $\theta_{\orb P}$ be the unique limiting theta characteristic on $\coarse{\orb P}$.
  Then the action of $\Aut_{f(x)}\orb P$ on $\theta_{\orb P}$ at $f(x)$ is by $-1$ and $f \from \Aut_x \orb D_\psi \to \Aut_{f(x)} \orb P$ is an isomorphism.
  Since $\theta = \psi^*\theta_{\rm rel} \otimes f^*\theta_{\orb P}$, we get the assertion.
\end{proof}
\begin{remark}
  Let $\psi' \from \orb P' \to \widetilde{\orb M}_{0,1+3}$ be the Abramovich--Vistoli stable map obtained from $\psi \from \orb P \to \orb Z$ by contracting the unstable (= redundant) components of $\orb P^{\rm tail}$.
  Let $\orb D'_\psi \to \orb P'$ be the corresponding triple cover and let $\theta'$ be the limiting theta characteristic.
  The statement of \autoref{thm:limit_theta} holds also for $\orb D'_\psi \to \orb P'$ and $\theta'$.
  Indeed, in a neighborhood of the node $x$, the pairs $(\orb D_\psi, \theta)$ and $(\orb D'_\psi, \theta')$ are isomorphic. 
\end{remark}

We now have all the tools to determine the images in $\overline {\orb M}_6$ of the boundary components of $\orb Q^{\rm odd}$ of type (6), (7), and (8).

\subsection{Divisors of type (6)}\label{sec:6}
\begin{proposition}\label{thm:6}
  There are $10$ irreducible components of $\overline Q \cap \Delta$ which are images of divisors of type (6) in $\orb Q_6^{\rm odd}$.
  Their generic points correspond to the following stable curves:
  \begin{itemize}
  \item With the dual graph
    \tikz[baseline={(0,-0.1)}]{
      \draw(0,0) node (X) [curve, label={180:$X$}] {}
      (X) edge [loop right] (X);
    }
    \begin{enumerate}
    \item \label{6:11}
      A nodal plane quintic.
    \end{enumerate}
  \item With the dual graph
    \tikz[baseline={(0,-0.1)}]{
      \draw 
      (0,0) node (X) [curve, label={180:$X$}] {}
      (1,0) node (Y) [curve, label={0:$Y$}] {}
      (X) edge node[int] {$p$} (Y);
    }
    \begin{enumerate}[resume]
    \item
      \label{6:21}
      $(X,p)$ the normalization of a cuspidal plane quintic and $Y$ of genus 1.
    \item 
      \label{6:22}
      $X$ Maroni special of genus $4$, $Y$ of genus $2$, $p \in X$ a ramification point of the unique degree 3 map $X \to \P^1$, and $p \in Y$ a Weierstrass point.
    \item
      \label{6:23}
      $X$ a plane quartic, $Y$ hyperelliptic of genus 3, $p \in X$ a point on a bitangent, and $p \in Y$ a Weierstrass point.
    \item
      \label{6:24}
      $X$ of genus $2$, $Y$ hyperelliptic of genus 4, $p \in Y$ a Weierstrass point.
    \item 
      \label{6:25}
      $X$ a plane quartic, $Y$ hyperelliptic of genus 3, and $p \in X$ a hyperflex ($K_X = 4p$).
    \item 
      \label{6:26}
      $X$ of genus 1, $Y$ hyperelliptic of genus 5.
    \end{enumerate}
  \item With the dual graph
    \tikz[baseline={(0,-0.1)}]{
      \draw 
      (0,0) node (X) [curve, label={180:$X$}] {}
      (1,0) node (Y) [curve, label={0:$Y$}] {}
      (X) edge [bend right=45] node[int] {$q$} (Y)
      (X) edge [bend left=45] node[int] {$p$} (Y);
    }
    \begin{enumerate}[resume]
    \item 
      \label{6:31}
      $X$ Maroni special of genus $4$, $Y$ of genus $1$, and $p, q \in X$ on a fiber of the unique degree 3 map $X \to \P^1$.
    \item 
      \label{6:32}
      $X$ a plane quartic, $Y$ of genus 2, the line through $p, q \in X$ tangent to $X$ at a third point, and $p, q  \in Y$ hyperelliptic conjugate.
    \item
      \label{6:33}
      $X$ a curve of genus $2$, $Y$ hyperelliptic of genus 3, and $p, q \in Y$ hyperelliptic conjugate.
    \end{enumerate}
  \end{itemize}
\end{proposition}
The rest of this section is devoted to the proof of \autoref{thm:6}.

Recall that type (6) corresponds to $(\phi \from \orb P \to \orb Z, \orb Z \to \widetilde{\stack{M}}_{0,4})$ where $\phi$ has the following dual graph:
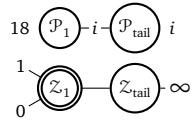
\[
\begin{tikzpicture}[baseline=(C2)]
  \draw
  (0,0) node[curve, double] (P1) {$\orb Z_1$}
  +(150:.6) node[int]  {$1$} edge (P1)
  +(210:.6) node[int]  {$0$}  edge (P1)
  (1,0) node[curve] (P2) {$\orb Z_{\rm tail}$}
  +(0:.6) node[int]  {$\infty$}  edge (P2);
  \begin{scope}[shift={(0,.8)}]
    \draw
    (0,0) node[curve, label={180:18}] (C1) {$\orb P_1$}
    (1,0) node[curve, label={0:$i$}] (E) {$\orb P_{\rm tail}$};
  \end{scope};
  \draw (C1) edge node[int] {$i$} (E)  (P1) edge (P2);
\end{tikzpicture}
\]
Let $(\widetilde S, \widetilde C) \to \orb P$ be the pullback of the universal family of 4-pointed rational curves.
Let $\orb P \to P$ be the coarse space away from the node.
Let $f \from (S, C) \to P$ be obtained from $(\widetilde S, \widetilde C)$ by the blow-up blow-down construction.
Define $C_1 = f^{-1}(P_1)$ and $C_{\rm tail} = f^{-1} (P_{\rm tail})$, and similarly for $S_1$ and $S_{\rm tail}$.
Set $x = P_1 \cap P_{\rm tail}$.
Denote the fiber of $S \to P$ over $x$ by $(\P^1 \cup \P^1, \{1,2,3,4\})$, where $1,2$ lie on one component and $3,4$ on the other.

\begin{figure}
  \centering
  \begin{tikzpicture}
    \draw[thick, smooth, tension=1]
    plot coordinates {(-1,1) (-.8, 0) (-1,-1) }
    plot coordinates {(1,1) (1.3,-0.2) }
    plot coordinates {(1,-1) (1.3,0.2) }
    plot coordinates {(-1,1) (1,1) }
    plot coordinates {(-1,-1) (1,-1) }
    ;
    \draw
    (1.1, .6) edge (.8, .6) (.8, .6) edge[dotted] (.4, .6)
    (1.2, .2) edge (.8, .2) (.8, .2) edge[dotted] (.4, .2)
    (1.1, -.6) edge (.8, -.6) (.8, -.6) edge[dotted] (.4, -.6)
    (1.2, -.2) edge (.8, -.2) (.8, -.2) edge[dotted] (.4,-.2)
    ;
    \draw
    (0,0) node {$C_1$};
    
    \begin{scope}[shift={(3,0)}]
      \draw[thick]
      plot coordinates {(-1,1) (-0.7,-0.2) }
      plot coordinates {(-1,-1) (-0.7,0.2) }
      plot coordinates {(1,1) (1.3,-0.2) }
      plot coordinates {(1,-1) (1.3,0.2) }
      plot coordinates {(-1,1) (1,1) }
      plot coordinates {(-1,-1) (1,-1) }
      plot coordinates {(-0.7,-0.2) (1.3,-0.2) }
      plot coordinates {(-0.7,0.2) (1.3,0.2) };
      \draw
      (1.1, .7) edge (.8, .7) (.8, .7) edge[dotted] (.5, .7)
      (1.2, .4) edge (.8, .4) (.8, .4) edge[dotted] (.5, .4)
      (1.1, -.7) edge (.8, -.7) (.8, -.7) edge[dotted] (.5, -.7)
      (1.2, -.4) edge (.8, -.4) (.8, -.4) edge[dotted] (.5,-.4)

      (-.9, .7) edge (-.5, .7) (-.5, .7) edge[dotted] (-.2, .7)
      (-.8, .4) edge (-.4, .4) (-.4, .4) edge[dotted] (-.1, .4)
      (-.9, -.7) edge (-.5, -.7) (-.5, -.7) edge[dotted] (-.2, -.7)
      (-.8, -.4) edge (-.4, -.4) (-.4, -.4) edge[dotted] (-.1,-.4)

      (0.2, .55) node {$C_{\rm tail}^1$}
      (0.2, -.55) node {$C_{\rm tail}^2$}
      ;
    \end{scope}
  \end{tikzpicture}
  \caption{A sketch of $C_1 \subset S_1$ and $C_{\rm tail} \subset S_{\rm tail}$ as in type (6)}\label{fig:6}
\end{figure}
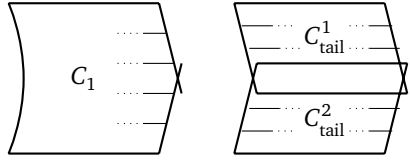
\subsubsection{Analyzing $P_{\rm tail}$}
The map $S_{\rm tail} \to P_{\rm tail}$ is a $\P^1 \cup \P^1$ bundle.
Recall the \'etale double cover $\orb G \to \orb P_{\rm tail}$ in \eqref{eqn:CGP} on page~\pageref{eqn:CGP}, whose fiber over $t$ corresponds to the two components of $S_{\rm tail}|_t$.
Since the action of $\Aut_t \orb P_{\rm tail}$ on the two components is trivial for all $t$ except possibly the node, $\orb G \to \orb P_{\rm tail}$ descends to an \'etale double cover $G \to P_{\rm tail}$.
Since $P_{\rm tail}$ has only one orbifold point, it is simply connected, and hence $G$ is the trivial cover $P_{\rm tail} \sqcup P_{\rm tail}$.
The degree 4 cover $C_{\rm tail} \to P_{\rm tail}$ factors as $C_{\rm tail} \to G \to P_{\rm tail}$.
Hence, it is a disjoint union $C_{\rm tail} = C^1_{\rm tail} \sqcup C^2_{\rm tail}$.
Both $C^i_{\rm tail}$ are hyperelliptic curves, each contained in a component of $S_{\rm tail}$ and lying away from the singular locus (See a sketch in \autoref{fig:6}).
We claim that if both $C^1_{\rm tail} \to P_{\rm tail}$ and $C^2_{\rm tail} \to P_{\rm tail}$ are nontrivial covers, then the boundary divisor maps to a locus of codimension at least 2 in $\overline Q$.
Indeed, compose $C^2_{\rm tail} \to P_{\rm tail}$ by an automorphism of $P_{\rm tail}$ that fixes $x$.
The resulting cover also represents an element of the same boundary divisor and has the same image if $\overline {\orb M}_6$.
The claim follows from the fact that there is a 2-dimensional choice of moduli for $\Aut(P_{\rm tail}, x)$.
We may thus assume that $C^2_{\rm tail} = P_{\rm tail} \sqcup P_{\rm tail}$.
Without loss of generality, take $C^2_{\rm tail}|_x = \{3,4\}$.
Then the monodromy of $\{1,2,3,4\}$ at $x$ is either trivial or $(12)$.
The map $C^1_{\rm tail} \to P_{\rm tail}$ is ramified at $i$ points.
The component of $S_{\rm tail}$ containing $C^1_{\rm tail}$ is the bundle $\P(\orb O \oplus \orb O(i/2))$.
The component of $S_{\rm tail}$ containing $C^2_{\rm tail}$ is the trivial bundle $\P^1 \times P_{\rm tail}$ (See \autoref{fig:6}).

\subsubsection{Analyzing $P_1$}
The map $S_1 \to P_1$ is a $\P^1$-bundle away from $x$; over $x$ the fiber is isomorphic to $\P^1 \cup \P^1$ (See \autoref{fig:6}).
By blowing down the component containing $1$ and $2$ as in the proof of \autoref{thm:An_singularities} , we see that $\coarse {C_1}$ is the normalization of a curve $C_1'$ on a Hirzebruch surface $\F_l$ which has an $A_{i-1}$ singularity over the fiber over $x$ along with two smooth unramified points, namely $\{3,4\}$.
Note that $S \to P$ admits a section of self-intersection $l \pmod 2$, which consists of a horizontal section of the component of $S_{\rm tail}$ containing $C^2_{\rm tail}$ and a section of $S_1 \to P_1$ that only intersects the component of $S_1|_x$ containing $\{3,4\}$.
Also note that $C_1' \subset \F_l$ is of class $4 \sigma + (3+2l) F$.
Since $C$ and hence $C_1'$ is connected, the only possible odd choice of $l$ is $l = 1$.

\subsubsection{Putting the components together}
By the analysis above, we see that the (pre)-stable images of generic points of divisors of type (6) are of the following two forms:
First, for odd $i$ we get $C_1 \cup_p C_{\rm tail}$, where $(C_1, p)$ is the normalization of a curve of class $4\sigma + 5F$ on $\F_1$ with an $A_{i-1}$ singularity, $C_{\rm tail}$ is a hyperelliptic curve of genus $(i-1)/2$, and $p \in C_{\rm tail}$ is a Weierstrass point.
Second, for even $i$ we get $C_1 \cup_{p,q} C_{\rm tail}$ where $(C_1, \{p,q\})$ is the normalization of a curve of class $4\sigma + 5F$ on $\F_1$ with an $A_{i-1}$ singularity, $C_{\rm tail}$ is a hyperelliptic curve of genus $i/2$ and $p, q \in C_{\rm tail}$ are hyperelliptic conjugate.
In both cases, we have $1 \leq i \leq 14$.
The case of $i = 1$ gives a smooth stable curve so we discard it.
The cases $i = 3, 5, 7, 9$ give the divisors \eqref{6:21}, \eqref{6:22}, \eqref{6:23}, \eqref{6:24} of \autoref{thm:6}.
The case of $i = 11$ yields a codimension 2 locus.
The case of irreducible $C_1'$ and $i = 2, 4, 6, 8$ give the divisors \eqref{6:11}, \eqref{6:31}, \eqref{6:32}, \eqref{6:33}, respectively.
The cases of $i = 10, 12$ yield codimension 2 loci.
We also have cases with reducible $C_1'$ for $i = 2, 4, 6$.
For $i = 2$, we can have $C_1'$ be the union of $\sigma$ with a tangent curve of class $3 \sigma + 5F$, which again gives divisor~\eqref{6:21}.
For $i = 4$, we can have $C_1'$ be the union of $\sigma + F$ with a 4-fold tangent curve of class $3 \sigma + 4F$, which gives divisor~\eqref{6:25}.
For $i = 6$, we can have $C_1'$ be the union of $\sigma + 2F$ with a 6-fold tangent curve of class $3 \sigma + 3F$ or the union of $2 \sigma + 2F$ with a 6-fold tangent curve of class $2 \sigma + 3F$, both of which give divisor~\eqref{6:26}.
The proof of \autoref{thm:6} is now complete.

\subsection{Divisors of type (7)}\label{sec:7}
\begin{proposition}\label{thm:7}
  There are $8$ irreducible components of $\overline Q \cap \Delta$ which are images of divisors of type (7) in $\orb Q_6^{\rm odd}$.
  Their generic points correspond to the following stable curves:

  \begin{itemize}
  \item With the dual graph
    \tikz{
      \draw 
      (0,0) node (X) [curve, label={180:$X$}] {} 
      (X) edge [loop right] (X);
    }
    \begin{enumerate}
    \item  $X$ hyperelliptic of genus $5$.
      \label{7:11}
    \end{enumerate}
  \item With the dual graph
    \tikz[baseline={(0,-0.1)}]{
      \draw 
      (0,0) node (X) [curve, label={180:$X$}] {} 
      (1,0) node (Y) [curve, label={0:$Y$}] {}
      (X) edge node[int]{$p$} (Y);
    }
    \begin{enumerate}[resume]
    \item $X$ of genus 2, $Y$ Maroni special of genus $4$, $p \in X$ a Weierstrass point, $p \in Y$ a ramification point of the unique degree 3 map $Y \to \P^1$.
      \label{7:21}
    \item $X$ hyperelliptic of genus $3$, $Y$ of genus 3, $p \in X$ a Weierstrass point, $p \in Y$ a point on a bitangent.
      \label{7:22}
    \item $X$ hyperelliptic of genus $4$, $Y$ of genus 2, $p \in X$ a Weierstrass point.
      \label{7:23}
    \end{enumerate}
  \item With the dual graph
    \tikz[baseline={(0,-0.1)}]{
      \draw 
      (0,0) node (X) [curve, label={180:$X$}] {}
      (1,0) node (Y) [curve, label={0:$Y$}] {}
      (X) edge [bend right=45]node[int] {$q$} (Y)
      (X) edge [bend left=45] node[int] {$p$} (Y);
    } 
    \begin{enumerate}[resume]
    \item $X$ hyperelliptic of genus 3, $Y$ of genus $2$, and $p \in Y$ a Weierstrass point.
      \label{7:31}
    \item $X$ of genus $2$, $Y$ a plane quartic, $p, q \in X$ hyperelliptic conjugate, the line through $p$, $q$ tangent to $Y$ at a third point.
      \label{7:32}
    \item  $X$ hyperelliptic of genus $3$, $Y$ of genus $2$, $p, q \in X$ hyperelliptic conjugate.
      \label{7:33}
    \end{enumerate}

    \item With the dual graph
      \tikz[baseline={(0,-0.1)}]{
      \draw 
      (0,0) node (X) [curve, label={180:$X$}] {}
      (1,0) node (Y) [curve, label={0:$Y$}] {}
      (X) edge [bend right=45] (Y)
      (X) edge (Y)
      (X) edge [bend left=45] (Y);
    }
      \begin{enumerate}[resume]
      \item $X$ hyperelliptic of genus $3$, $Y$ of genus $1$.
        \label{7:41}
      \end{enumerate}
  \end{itemize}
\end{proposition}
The rest of this section is devoted to the proof of \autoref{thm:7}.

Recall that type (7) corresponds to $(\phi \from \orb P \to \orb Z, \orb Z \to \widetilde{\stack{M}}_{0,4})$ where $\phi$ has the following dual graph:
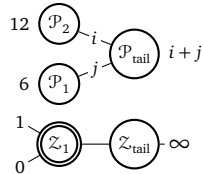
\[
\begin{tikzpicture}[baseline=(C2)]
  \draw
  (0,0) node[curve, double] (P1) {$\orb Z_1$}
  +(150:.6) node[int]  {$1$} edge (P1)
  +(210:.6) node[int]  {$0$}  edge (P1)
  (1,0) node[curve] (P2) {$\orb Z_{\rm tail}$}
  +(0:.6) node[int]  {$\infty$}  edge (P2);
  \begin{scope}[shift={(0,.8)}]
    \draw
    (0,0) node[curve, label={180:6}] (C1) {$\orb P_1$}
    (0,.8) node[curve, label={180:12}] (C2) {$\orb P_2$}
    (1,.4) node[curve, label={0:$i+j$}] (E) {$\orb P_{\rm tail}$};
  \end{scope};
  \draw
  (C1) edge node[int] {$j$} (E)
  (C2) edge node[int] {$i$} (E)
  (P1) edge (P2);
\end{tikzpicture}
\]
Let $(\widetilde S, \widetilde C) \to \orb P$ be the pullback of the universal family of 4-pointed rational curves.
Let $\orb P \to P$ be the coarse space away from the nodes.
Let $f \from (S, C) \to P$ the family obtained by the blow-up blow-down construction. 
Define $C_1 = f^{-1}(P_1)$, and similarly for $C_2$, $C_{\rm tail}$, $S_1$, $S_2$, and $S_{\rm tail}$.
Set $x_1 = P_{\rm tail} \cap P_1$ and $x_2 = P_{\rm tail} \cap P_2$.

\subsubsection{Analyzing $P_{\rm tail}$}
The map $S_{\rm tail} \to P_{\rm tail}$ is a $\P^1 \cup \P^1$ bundle.
The \'etale double cover given by the two components is $\orb G \to \orb P_{\rm tail}$ of \eqref{eqn:CGP} which induces an \'etale double cover $G \to P_{\rm tail}$ as in type (6).
We have the factorization $C_{\rm tail} \to G \to P_{\rm tail}$.
Since $P_{\rm tail}$ has two orbifold points $x_1$ and $x_2$, this cover may be nontrivial.
If $G \to P_{\rm tail}$ is trivial, then $C_{\rm tail}$ is the disjoint union $C^1_{\rm tail} \sqcup C^2_{\rm tail}$ of two double covers of $P_{\rm tail}$.
If $G \to P_{\rm tail}$ is nontrivial, then $\coarse{G}$ is a rational curve and $\coarse{C_{\rm tail}}$ is its double cover.

\subsubsection{Analyzing $P_1$}
Denote the fiber of $S_1 \to P_1$ over $x_1$ by $(P_A \cup P_B, \{1,2,3,4\})$, where $P_A \cong P_B \cong \P^1$; $1,2 \in P_A$; and $3,4 \in P_B$.
Let $\pi \in \sym_4$ be a generator of the monodromy of $C_1 \to P_1$ at $x_1$.
By \autoref{thm:An_singularities}, $\coarse{C_1}$ is the normalization of $\coarse{C_1'}$, where $C_1'$ is a fiberwise degree 4 curve on a $\P^1$-bundle $S_1' = \F_l$ over $P_1'$ where $P_1' = \P^1$ or $P_1' = \P^1(\sqrt 0)$.
In either case, $C_1'$ is of class $4\sigma + (1+2l)F$.
From \autoref{prop:orbiscroll_4}, we see that the only possibilities for $l$ are $l = 0$ and $1$ if $P_1' = \P^1$ and $l = 1/2$ if $P_1' = \P^1(\sqrt 0)$.
Also, if $l = 1$ then $C_1'$ is the disjoint union of $\sigma$ with a curve in the class $3 \sigma + 3F$.

The case $P_1' = \P^1$ occurs if $\pi$ preserves the two components $P_A$ and $P_B$.
By \autoref{thm:An_singularities}, $C_1'$ has an $A_{i-k}$ and an $A_{k-2}$ singularity over $0$ for some $k = 1, \dots, i+1$.
By $\autoref{rem:ij}$, we may assume that the singularities are $A_{i-1}$ and $A_{-1}$ or $A_{i-2}$ and $A_0$ (if $i$ is even).

The case $P_1' = \P^1(\sqrt 0)$ occurs if $\pi$ switches the two components $P_A$ and $P_B$.
By \autoref{thm:An_singularities}, over an \'etale chart around $0 \in \P^1(\sqrt 0)$, the pullback of $C_1'$ has two $A_{i-1}$ singularities over $0$ that are conjugate under the $\Z_2$ action.
To identify such a curve in more classical terms, we use the strategy of \autoref{sec:1-5}.
Indeed, by diagram~\eqref{F22} on page~\pageref{resolution_contraction_diagrams_1}, we get that $\coarse {C_1'}$ is a curve of class $2 \sigma + 4 F$ on $\F_2$ disjoint from the directrix and with an $A_{i-1}$ singularity on the fiber of $\F_2 \to \P^1$ over $0$.

We now simply enumerate the possibilities for $\coarse{C_1}$ along with its attaching data with the rest of $C$, namely the divisor $\coarse{D_1} = \coarse{f^{-1}(x_1)}$.
We list the possible dual graphs for $(\coarse{C_1}, \coarse{D_1})$, where the vertices represent connected components of $\coarse{C_1}$ labeled by their genus, and the half-edges represent points of $\coarse{D_1}$, labeled by their multiplicity in $\coarse{D_1}$.
In the case where $\pi$ preserves $A$ and $B$, we record some additional data as follows.
We make the convention that the half edges depicted on top (resp. bottom) are images of the points which lie on $P_A \subset S_1$ (resp. $P_B \subset S_1$).
We then record the self-intersection (modulo 2) of a section $\sigma_A$ (resp. $\sigma_B$) of $S_1 \to P_1$ that lies in the smooth locus of $S_1 \to P_1$  and meets $P_A$ (resp. $P_B$).
In the case where $\pi$ switches $A$ and $B$, there is no such additional information.
Here we make the convention that the half-edges are depicted on the sides.

For example, let us take $i = 1$.
For $l = 0$, we get $C_1' \subset S_1' = \F_0$ of class $4\sigma + F$ with an $A_0$ singularity (that is, a point of simple ramification) over $0 \in \P^1$.
This gives us the dual graph in \ref{c1:11}.
To get the additional data, we must reconstruct $S_1$ from $S_1'$, which we do by a stable reduction of the 4-pointed family $(S_1', C_1') \to \P^1$ of rational curves around $0$.
To do so, set $P_1 = \P^1(\sqrt 0)$.
We first pass to the base change $S_1' \times_{\P^1} P_1$, on which the curve $C_1' \times_{\P^1} P_1$ has a node.
The blow up of $S_1' \times_{\P^1} P_1$ at the node and the proper transform of $C_1' \times_{\P^1} P_1$ gives the required family $(S_1, C_1)$.
The central fiber of $S_1 \to P_1$ is $P_A \cup P_B$, where $P_A$ is the exceptional curve of the blow up and $P_B$ is the proper transform of the original fiber.
The self-intersection of a section meeting $P_A$ (resp. $P_B$) is $-1/2$ (resp. $0$).
This leads to the complete picture \ref{c1:11}.
For $l = 1$, the some procedure gives \ref{c1:12}.
For $l = 1/2$, we get that $\coarse{C_1'}$ is a curve of class $2 \sigma + 4F$ on $\F_2$ disjoint from $\sigma$ and with an $A_0$ singularity (that is, a point of simple ramification) over the fiber $F$ of $\F_2 \to \P^1$ over $0$.
The divisor $\coarse{D_1'}$ is $\coarse{C_1'} \cap 2F$.
This leads to the picture \ref{c1:13}.
We get the pictures for $i = 2, 3, 4$ analogously.

\begin{itemize}
  \label{7:shorter_ends}
\item $i = 1$
  \begin{multicols}{2}
    \begin{enumerate}[series=c1, label=1.\arabic*., ref=1.\arabic*]
    \item
      \label{c1:11}
      \tikz[baseline, thick]{
        \draw (0,0) node (1) [curve, label={180:$0$}] {} 
        +(90:.4) node[int] {2} edge (1) 
        +(-75:.4) node[int] {} edge (1) 
        +(-105:.4) node[int] {} edge (1) 
        ;}\\
            $\sigma_A^2 = -1/2$,
      $\sigma_B^2 = 0$
          \item
      \label{c1:12}
      \tikz[baseline, thick]{
        \draw (0,0) node (1) [curve, label={180:$0$}] {} 
        +(-90:.4) node[int] {} edge (1)
        (1,0) node (2) [curve, label={180:$1$}] {}
        +(+90:.4) node[int] {2} edge (2) 
        +(-90:.4) node[int] {} edge (2) 
        ;}\\
            $\sigma_A^2 \equiv 1/2$,
      $\sigma_B^2 \equiv -1$
          \item
      \label{c1:13}
      \tikz[baseline, thick]{
        \draw (0,0) node (1) [curve, label={180:$1$}] {} 
        +(0:.4) node[int] {4} edge (1)
        ;}
    \end{enumerate}
  \end{multicols}

\item $i = 2$
  \begin{multicols}{2}
    \begin{enumerate}[resume*=c1]

    \item
      \label{c1:21}
      \tikz[baseline, thick]{
        \draw
        (0,0) node (1) [curve, label={180:0}] {} 
        +(90:.4) node[int] {} edge (1)
        (1,0) node (2) [curve, label={180:0}] {}
        +(90:.4) node[int] {} edge (2) 
        +(-75:.4) node[int] {} edge (2) 
        +(-105:.4) node[int] {} edge (2) 
        ;}\\
            $\sigma_A^2 \equiv -1$,
      $\sigma_B^2 \equiv 0$

    \item
      \label{c1:22}
      \tikz[baseline, thick]{
        \draw (0,0) node (1) [curve, label={180:0}] {} 
        +(90:.4) node[int] {2} edge (1) 
        +(-90:.4) node[int] {2} edge (1) 
        ;}
      \\
            $\sigma_A^2 \equiv -1/2$,
      $\sigma_B^2 \equiv -1/2$

    \item
      \label{c1:23}
      \tikz[baseline, thick]{
        \draw (0,0) node (1) [curve, label={180:0}] {} 
        +(20:.4) node[int] {2} edge (1) 
        +(-20:.4) node[int] {2} edge (1) 
        ;}
      \\
    \end{enumerate}
  \end{multicols}

\item $i = 3$
  \begin{multicols}{2}
    
    \begin{enumerate}[resume*=c1]
    \item
      \label{c1:31}
      \tikz[baseline, thick]{
        \draw (0,0)
        node (1) [curve, label={180:0}] {}
        +(-90:.4)
        node[int] {} edge (1) (1,0) node (2) [curve, label={180:0}] {}
        +(+90:.4) node[int] {2} edge (2)
        +(-90:.4) node[int] {} edge
        (2) ; }
      \\
            $\sigma_A^2 \equiv -1/2$,
      $\sigma_B^2 \equiv 1$ 
            
    \item
      \label{c1:32}
      \tikz[baseline, thick]{
        \draw (0,0) node (1) [curve, label={180:0}] {} 
        +(0:.4) node[int] {4} edge (1)
        ;}
    \end{enumerate}
  \end{multicols}
\item $i = 4$
  \begin{multicols}{2}
    \begin{enumerate}[resume*=c1]
    \item
      \label{c1:41}
      \tikz[baseline, thick]{
        \draw
        (0,0) node (1) [curve, label={180:0}] {} 
        +(90:.4) node[int] {} edge (1)
        (1,0) node (2) [curve, label={180:0}] {}
        +(90:.4) node[int] {} edge (2) 
        +(-90:.4) node[int] {} edge (2)
        (2,0) node (3) [curve, label={180:0}]{}
        +(-90:.4) node[int] {} edge (3)
        ;
      }\\
            $\sigma_A^2 \equiv 1$,
      $\sigma_B^2 \equiv 1$ 
      
    \item
      \label{c1:42}
      \tikz[baseline, thick]{
        \draw (0,0) node (1) [curve, label={180:0}] {} 
        +(0:.4) node[int] {2} edge (1) 
        (1,0) node (2) [curve, label={180:0}] {} 
        +(0:.4) node[int] {2} edge (2) 
        ;}
    \end{enumerate}
  \end{multicols}
\end{itemize}

\subsubsection{Analyzing $P_2$}
The story here is entirely analogous to that of $P_1$, except that the curve $C_2' \subset \F_l$ is of class $4 \sigma + (2+2l)F$, and the allowed values of $l$ are $l = 0$, $1/2$, $1$, and $2$.
The case of $l = 2$ corresponds to a disjoint union of $\sigma$ and $3 \sigma + (2+2l)F$.
The case of $l = 1/2$ corresponds to diagram \ref{F23} on page~\pageref{resolution_contraction_diagrams_1}, which shows that $\coarse{C_2'}$ is a curve of class $3 \sigma + 6 F$ on $\F_2$ disjoint from $\sigma$ and with an $A_{j-1}$ singularity on the fiber over $0$.
We enumerate the possibilities with the same conventions as before.

\begin{itemize}
  \label{7:longer_ends}
\item $j$ odd, say $j = 2p+1$.
  \begin{multicols}{2}
    \begin{enumerate}[series=c2, label=2.\arabic*., ref=2.\arabic*]
    \item
      \label{c2:odd1}             
      \tikz[baseline, thick]{
        \draw (0,0) node (1) [curve, label={180:$3-p$}] {} 
        +(90:.4) node[int] {2} edge (1)
        +(-75:.4) node[int] {} edge (1)
        +(-105:.4) node[int] {} edge (1)
        ;}\\
            $\sigma_A^2 \equiv p+1/2$,
      $\sigma_B^2 \equiv 1$\\
      For $0 \leq p \leq 3$.
            \\
    \item
      \label{c2:odd2}
      \tikz[baseline, thick]{
        \draw (0,0) node (1) [curve, label={180:$3-p$}] {} 
        +(90:.4) node[int] {2} edge (1)
        +(-75:.4) node[int] {} edge (1)
        +(-105:.4) node[int] {} edge (1)
        ;}\\
            $\sigma_A^2 \equiv p-1/2$,
      $\sigma_B^2 \equiv 0$\\
      For $0 \leq p \leq 3$.

    \item
      \label{c2:odd3}
      \tikz[baseline, thick, xscale=-1]{
        \draw (0,0) node (1) [curve, label={0:$4-p$}] {} 
        +(90:.4) node[int] {2} edge (1)
        +(-90:.4) node[int] {} edge (1)
        (1,0) node (2) [curve, label={180:$0$}] {}
        +(-90:.4) node[int] {} edge (2)
        ;}\\
            $\sigma_A^2 \equiv p-1/2$,
      $\sigma_B^2 \equiv 0$ 
      
    \item
      \label{c2:odd4}
      \tikz[baseline, thick]{
        \draw (0,0) node (1) [curve, label={180:$4-p$}] {} 
        +(0:.4) node[int] {4} edge (1)
        ;}
    \end{enumerate}
  \end{multicols}
\item $j$ even, say $j = 2p$.
  \begin{multicols}{2}
    \begin{enumerate}[resume*=c2]
    \item
      \label{c2:even1}
      \tikz[baseline, thick]{
        \draw (0,0) node (1) [curve, label={180:$3-p$}] {} 
        +(75:.4) node[int] {} edge (1)
        +(105:.4) node[int] {} edge (1)
        +(-75:.4) node[int] {} edge (1)
        +(-105:.4) node[int] {} edge (1)
        ;}\\
            $\sigma_A^2 \equiv p+1$,
      $\sigma_B^2 \equiv 1$\\
      For $1 \leq p \leq 3$.
      
    \item
      \label{c2:even2}
      \tikz[baseline, thick]{
        \draw (0,0) node (1) [curve, label={180:$3-p$}] {} 
        +(75:.4) node[int] {} edge (1)
        +(105:.4) node[int] {} edge (1)
        +(-75:.4) node[int] {} edge (1)
        +(-105:.4) node[int] {} edge (1)
        ;}\\
            $\sigma_A^2 \equiv p$,
      $\sigma_B^2 \equiv 0$\\
      For $1 \leq p \leq 3$.
      
    \item
      \label{c2:even3}
      \tikz[baseline, thick, xscale=-1]{
        \draw (0,0) node (1) [curve, label={0:$4-p$}] {} 
        +(75:.4) node[int] {} edge (1)
        +(105:.4) node[int] {} edge (1)
        +(-90:.4) node[int] {} edge (1)
        (1,0) node (2) [curve, label={180:$0$}] {}
        +(-90:.4) node[int] {} edge (2)
        ;}\\
            $\sigma_A^2 \equiv p$,
      $\sigma_B^2 \equiv 0$ 
      
    \item
      \label{c2:even4}
      \tikz[baseline, thick]{
        \draw (0,0) node (1) [curve, label={180:$0$}] {} 
        +(90:.4) node[int] {} edge (1)
        +(-90:.4) node[int] {} edge (1)
        (1,0) node (2) [curve, label={180:$0$}] {}
        +(90:.4) node[int] {} edge (2)
        +(-90:.4) node[int] {} edge (2)
        ;}\\
            $\sigma_A^2 \equiv 1$,
      $\sigma_B^2 \equiv 1$ \\
      For $p = 4$.

    \item
      \label{c2:even5}
      \tikz[baseline, thick]{
        \draw (0,0) node (1) [curve, label={180:$0$}] {} 
        +(90:.4) node[int] {} edge (1)
        +(-90:.4) node[int] {} edge (1)
        (1,0) node (2) [curve, label={180:$0$}] {}
        +(90:.4) node[int] {} edge (2)
        +(-90:.4) node[int] {} edge (2)
        ;}\\
            $\sigma_A^2 \equiv 0$,
      $\sigma_B^2 \equiv 0$ \\
      For $p = 4$.

    \item
      \label{c2:even6}
      \tikz[baseline, thick, xscale=.8]{
        \draw (0,0) node (1) [curve, label={180:$0$}] {} 
        +(90:.4) node[int] {} edge (1)
        (1,0) node (2) [curve, label={180:$1$}] {}
        +(90:.4) node[int] {} edge (2)
        +(-90:.4) node[int] {} edge (2)
        (2,0) node (3) [curve, label={180:$0$}] {}
        +(-90:.4) node[int] {} edge (3)
        ;}\\
            $\sigma_A^2 \equiv 0$,
      $\sigma_B^2 \equiv 0$ \\
      For $p = 4$.

    \item
      \label{c2:even7}
      \tikz[baseline, thick]{
        \draw (0,0) node (1) [curve, label={180:$4-p$}] {} 
        +(90:.4) node[int] {2} edge (1)
        +(-90:.4) node[int] {2} edge (1)
        ;}\\
            $\sigma_A^2 \equiv p-3/2$,
      $\sigma_B^2 \equiv 3/2$
      
    \item
      \label{c2:even7.5}
      \tikz[baseline, thick]{
        \draw (0,0) node (1) [curve, label={180:$4-p$}] {} 
        +(90:.4) node[int] {2} edge (1)
        +(-90:.4) node[int] {2} edge (1)
        ;}\\
            $\sigma_A^2 \equiv p-1/2$,
      $\sigma_B^2 \equiv 1/2$
      
    \item
      \label{c2:even8}
      \tikz[baseline, thick]{
        \draw (0,0) node (1) [curve, label={180:$4-p$}] {} 
        +(20:.4) node[int] {2} edge (1)
        +(-20:.4) node[int] {2} edge (1)
        ;}\\

    \item
      \label{c2:even9}
      \tikz[baseline, thick]{
        \draw (0,0) node (1) [curve, label={180:$1$}] {} 
        +(0:.4) node[int] {2} edge (1)
        (1,0) node (1) [curve, label={180:0}] {}
        +(0:.4) node[int] {2} edge (1)
        ;}\\
            For $p = 4$.
          \end{enumerate}
  \end{multicols}
\end{itemize}

The marked curves appearing as $C_2$ above are not arbitrary in moduli. 
But it is easy to find which marked curves appear by using that they are normalizations of a singular curve $C_2'$ on a known surface of a known class and a known singularity.
We now write down these descriptions.
We denote by $a$ or $a_1,a_2$ (resp. $b$ or $b_1, b_2$) the point(s) represented by the half-edges on top (resp. bottom).
The numbering goes from the left to the right.

\begin{center}
\rowcolors{2}{gray!15}{white}
\begin{longtable}{l l p{.70\textwidth}}
  \hline
  \rowcolor{gray!25}
  Dual graph & $p$ & Description\\
  \hline
  \endhead
  \ref{c2:odd1}  & 0 & Plane quartic with $2a+b_1+b_2$ a canonical divisor \\
  \ref{c2:odd1}  & 1 & Genus 2 with $b_1$ and $b_2$ hyperelliptic conjugate\\
  \ref{c2:odd1}  & 2, 3 & Any moduli\\
  \ref{c2:odd2} & 0 & Hyperelliptic genus 3 with 3 marked points\\
  \ref{c2:odd2} & 1 & Genus 2 with $a$ a Weierstrass point\\
  \ref{c2:odd2} & 2,3 & Any moduli\\
  \ref{c2:odd3} & 0  & $\P^1 \sqcup $ Maroni special of genus 3 with $2a+b_2$ the $g^1_3$\\
  \ref{c2:odd3} & 1 & $\P^1 \sqcup $ plane quartic with $2a+2b_2$ a canonical divisor\\
  \ref{c2:odd3} & 2 & $\P^1 \sqcup $ genus 2 with $b_2$ a Weierstrass point\\
  \ref{c2:odd3} & 3 & $\P^1 \sqcup $ genus 1 with $a-b_2$ two-torsion\\
  \ref{c2:odd3} & 4 & Any moduli\\
  \ref{c2:odd4} & 0 & Maroni special genus 4 with a ramification point of the $g^1_3$\\
  \ref{c2:odd4} & 1 & Plane quartic with a point on a bitangent \\
  \ref{c2:odd4} & 2,3,4 & Any moduli\\
  \ref{c2:even1} & 1 & Genus 2 with $b_1$, $b_2$ hyperelliptic conjugate\\
  \ref{c2:even1} & 2, 3 & Any moduli\\
  \ref{c2:even2} & 1 & Genus 2 with $a_1$, $a_2$ hyperelliptic conjugate\\
  \ref{c2:even2} & 2, 3 & Any moduli\\
  \ref{c2:even3} & 1 & $\P^1 \sqcup $ plane quartic with $a_1+a_2+2b_2$ a canonical divisor\\
  \ref{c2:even3} & 2 & $\P^1 \sqcup $ genus 2 with $b_2$ a Weierstrass point\\
  \ref{c2:even3} & 3 & $\P^1 \sqcup $ genus 1 with $a_1+a_2 = 2b_2$\\
  \ref{c2:even3} & 4 & Any moduli\\
  \ref{c2:even4}, \ref{c2:even5} & -- & Any moduli\\
  \ref{c2:even6} & -- & Genus 1 with $a-b$ two-torsion\\
  \ref{c2:even7} & 1 & Hyperelliptic genus 3 with any 2 points\\
  \ref{c2:even7} & 2 & Genus 2 with $a$ a Weierstrass point\\
  \ref{c2:even7} & 3,4 & Any moduli\\
  \ref{c2:even7.5} & 1 & Plane quartic with $2a_1+2a_2$ a canonical divisor \\
  \ref{c2:even7.5} & 2 & Genus 2 with $b$ a Weierstrass point\\
  \ref{c2:even7.5} & 3,4 & Any moduli\\
  \ref{c2:even8} & 1 & Plane quartic with the line joining the two points tangent at a third\\
  \ref{c2:even8} & 2,3,4 & Any moduli\\
  \ref{c2:even9} & -- & Any moduli
\end{longtable}
\end{center}

\subsubsection{Putting the components together}
Having described $C_{\rm tail} \to P_{\rm tail}$, $C_1 \to P_1$, and $C_2 \to P_2$ individually, we now put them together.
Let us first consider the case where $G \to P_{\rm tail}$ is trivial.
Recall that in this case $C_{\rm tail}$ is a disjoint union of two double covers $C^1_{\rm tail}$ and $C^2_{\rm tail}$ of $P_{\rm tail}$.
The dual graph of the coarse space of $C = C_1 \cup C^1_{\rm tail} \cup C^2_{\rm tail} \cup C_2$ has the following form 
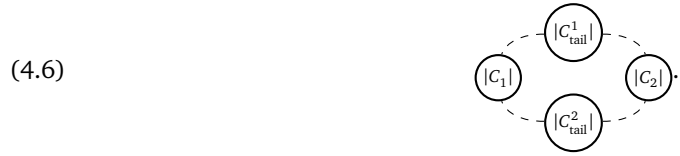
\begin{equation}\label{eqn:form7}
  \begin{tikzpicture}[yscale=.6, baseline]
    \draw
    (-1,0) node[curve] (1) {$\coarse{C_1}$}
    (1, 0) node[curve] (2) {$\coarse{C_2}$}
    (0,1) node[curve] (t1) {$\coarse{C^1_{\rm tail}}$}
    (0,-1) node[curve] (t2) {$\coarse{C^2_{\rm tail}}$}
    (t1) edge[dashed, bend right=40] (1) edge[dashed, bend left=40] (2)
    (t2) edge[dashed, bend left=40] (1) edge[dashed, bend right=40] (2)
    ;
  \end{tikzpicture}.
\end{equation}
Here a dashed line represents one or two nodes with the following admissibility criterion: In the case of one node,
the node point is a ramification point of the map to $\coarse P$ on both curves.
In the case of two nodes, the two node points are unramified points in a fiber of the map to $\coarse P$ on both curves.
The convention for drawing points of $A$ (resp. $B$) on top (resp. bottom) for $C_1$ and $C_2$ still applies, except that the $A/B$ for $C_1$ and $A/B$ for $C_2$ may be switched.
Note that $C$ comes embedded in a surface $S$ fibered over $P$ obtained by gluing the fibration $S_1 \to P_1$, the fibration $S_2 \to P_2$, and the fibration $S_{\rm tail} \to P_{\rm tail}$ (see \autoref{fig:7}).
We can determine the parity of $f \from S \to P$ using \autoref{thm:section_parity}.
We produce a section of $S \to P$ by gluing sections of $S_i \to P_i$ and of $S_{\rm tail} \to P_{\rm tail}$.
We have recorded the self-intersections of the sections $\sigma_i$ of $S_i \to P_i$ (modulo 2).
Consider a section $\sigma_{\rm tail}$ of $S_{\rm tail} \to P_{\rm tail}$ that matches with $\sigma_i$ over $x_i$ and lies in the smooth locus of $S_{\rm tail} \to P_{\rm tail}$. 
Such a section is a section of the $\P^1$ bundle $S^1_{\rm tail} \to P_{\rm tail}$ or $S^2_{\rm tail} \to P_{\rm tail}$, say the first.
Then the self-intersection of $\sigma_{\rm tail}$ (modulo 2) is $b_1/2$, where $b_1$ is the number of ramification points of $C^1_{\rm tail} \to P_{\rm tail}$.
We then get
\[ \sigma^2 = \sigma_1 ^2 + \sigma_2^2 + \sigma_{\rm tail}^2.\]
The parity of $\sigma^2$ determines the parity of $f \from S \to P$ by \autoref{thm:section_parity}.
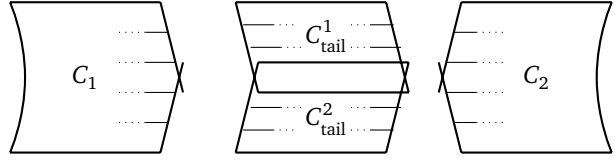
\begin{figure}
  \centering
  \begin{tikzpicture}
    \draw[thick, smooth, tension=1]
    plot coordinates {(-1,1) (-.8, 0) (-1,-1) }
    plot coordinates {(1,1) (1.3,-0.2) }
    plot coordinates {(1,-1) (1.3,0.2) }
    plot coordinates {(-1,1) (1,1) }
    plot coordinates {(-1,-1) (1,-1) }
    ;
    \draw
    (1.1, .6) edge (.8, .6) (.8, .6) edge[dotted] (.4, .6)
    (1.2, .2) edge (.8, .2) (.8, .2) edge[dotted] (.4, .2)
    (1.1, -.6) edge (.8, -.6) (.8, -.6) edge[dotted] (.4, -.6)
    (1.2, -.2) edge (.8, -.2) (.8, -.2) edge[dotted] (.4,-.2)
    ;
    \draw
    (0,0) node {$C_1$};

    \begin{scope}[shift={(6,0)}, xscale=-1]
    \draw[thick, smooth, tension=1]
    plot coordinates {(-1,1) (-.8, 0) (-1,-1) }
    plot coordinates {(1,1) (1.3,-0.2) }
    plot coordinates {(1,-1) (1.3,0.2) }
    plot coordinates {(-1,1) (1,1) }
    plot coordinates {(-1,-1) (1,-1) }
    ;
    \draw
    (1.1, .6) edge (.8, .6) (.8, .6) edge[dotted] (.4, .6)
    (1.2, .2) edge (.8, .2) (.8, .2) edge[dotted] (.4, .2)
    (1.1, -.6) edge (.8, -.6) (.8, -.6) edge[dotted] (.4, -.6)
    (1.2, -.2) edge (.8, -.2) (.8, -.2) edge[dotted] (.4,-.2)
    ;
    \draw
    (0,0) node {$C_2$};
  \end{scope}
  
    \begin{scope}[shift={(3,0)}]
      \draw[thick]
      plot coordinates {(-1,1) (-0.7,-0.2) }
      plot coordinates {(-1,-1) (-0.7,0.2) }
      plot coordinates {(1,1) (1.3,-0.2) }
      plot coordinates {(1,-1) (1.3,0.2) }
      plot coordinates {(-1,1) (1,1) }
      plot coordinates {(-1,-1) (1,-1) }
      plot coordinates {(-0.7,-0.2) (1.3,-0.2) }
      plot coordinates {(-0.7,0.2) (1.3,0.2) };
      \draw
      (1.1, .7) edge (.8, .7) (.8, .7) edge[dotted] (.5, .7)
      (1.2, .4) edge (.8, .4) (.8, .4) edge[dotted] (.5, .4)
      (1.1, -.7) edge (.8, -.7) (.8, -.7) edge[dotted] (.5, -.7)
      (1.2, -.4) edge (.8, -.4) (.8, -.4) edge[dotted] (.5,-.4)

      (-.9, .7) edge (-.5, .7) (-.5, .7) edge[dotted] (-.2, .7)
      (-.8, .4) edge (-.4, .4) (-.4, .4) edge[dotted] (-.1, .4)
      (-.9, -.7) edge (-.5, -.7) (-.5, -.7) edge[dotted] (-.2, -.7)
      (-.8, -.4) edge (-.4, -.4) (-.4, -.4) edge[dotted] (-.1,-.4)

      (0.2, .55) node {$C_{\rm tail}^1$}
      (0.2, -.55) node {$C_{\rm tail}^2$}
      ;
    \end{scope}
  \end{tikzpicture}
  \caption{A sketch of $C_1 \subset S_1$, $C_{\rm tail} \subset S_{\rm tail}$, and $C_2 \subset S_2$ as in type (7)}\label{fig:7}
\end{figure}

For example, taking the curve $C_1$ as in $\ref{c1:41}$, the curve $C_2$ as in $\ref{c2:odd2}$, the curve $C^1_{\rm tail}$ of genus $0$, and the curve $C^2_{\rm tail}$ of genus $p+1$ gives the following instance of \eqref{eqn:form7}.
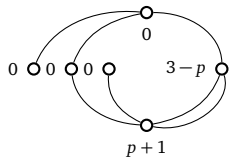
\[
\begin{tikzpicture}[yscale=0.75]
  \draw
  (-1.5,0) node[curve, label={180:0}] (c11) {}
  (-1,0) node[curve, label={180:0}] (c12) {}
  (-.5,0) node[curve, label={180:0}] (c13) {}
  
  (0, 1) node[curve, label={-90:0}] (t1) {}
  (0, -1) node[curve, label={-90:$p+1$}] (t) {}

  (1,0) node[curve, label={180:$3-p$}] (c2) {}
  
  (c2) edge[bend left=30] (t) edge[bend left=60] (t) edge[bend right=45] (t1)
  (t1) edge[bend right=40] (c11)
  (t1) edge[bend right=30] (c12)
  (t) edge [bend left=40] (c12)
  (t) edge [bend left=30] (c13);
\end{tikzpicture}
\]
The resulting $S \to P$ admits a section with self-intersection $p+1 \pmod 2$ and hence represents a divisor of $\orb Q^{\rm odd}$ for even $p$.
For $p = 0$, we get the divisor~\eqref{7:41} in \autoref{thm:7}.
For $p = 2$, we get a codimension 2 locus.

We similarly take all possible combinations of $C_1$, $C_2$, and $C_{\rm tail}$, compute the stable images (see \autoref{tab:71}), and do a dimension count to see which ones give divisors.
The combinations not shown in the \autoref{tab:71} correspond to boundary divisors of $\orb Q^{\rm odd}$ whose images in $\overline Q$ have codimension higher than one.
A prime (') denotes the dual graph obtained by a vertical flip (that is, by switching $A$ and $B$).

\begin{table}[ht]
  \centering
  \rowcolors{2}{gray!15}{white}
  \begin{tabular}{ccrrc}
    \hline
    \rowcolor{gray!25}
    $C_1$ & $C_2$& $g\left(C^1_{\rm tail}\right)$& $g\left(C^2_{\rm tail}\right)$ & Divisor in \autoref{thm:7}\\
    \hline

    \ref{c1:31} or \ref{c1:41}  & \ref{c2:odd2} $p = 0$ & $0$ & $1$ & \ref{7:41}\\
    \ref{c1:41} & \ref{c2:odd2} $p = 0$ & $2$ & $-1$ & \ref{7:31}\\
    \ref{c1:31} or \ref{c1:41} & \ref{c2:odd3} $p = 0$ & $2$ & $-1$ & \ref{7:21}\\
    \ref{c1:31} or \ref{c1:41} & \ref{c2:odd3} $p = 1$ & $3$ & $-1$ & \ref{7:22} \\
    \ref{c1:31} or \ref{c1:41} & \ref{c2:odd3} $p = 2$ & $4$ & $-1$ & \ref{7:23} \\
    \ref{c1:31} or \ref{c1:41} & \ref{c2:odd3} $p = 2$ & $0$ & $3$ & \ref{7:31}\\
    
    \ref{c1:31} or \ref{c1:41} & \ref{c2:even3} $p = 1$ & $2$ & $-1$ & \ref{7:33}\\
    \ref{c1:31} or \ref{c1:41} & \ref{c2:even3} $p = 2$ & $3$ & $-1$ & \ref{7:33} \\
    \ref{c1:31}'or \ref{c1:41}  & \ref{c2:even3} $p = 2$ & $-1$ & $3$ & \ref{7:31} \\
    \ref{c1:41} & \ref{c2:even3} $p = 4$ & $5$ & $-1$ & \ref{7:11} \\
    \ref{c1:41} &  \ref{c2:even5} & $5$ & $-1$ & \ref{7:11} \\
    \ref{c1:31} or \ref{c1:41}  & \ref{c2:even7} $p = 1$ & $0$ & $2$ & \ref{7:31} \\
  \end{tabular}
  \caption{Divisors of type (7) with trivial $G \to P_{\rm tail}$}
  \label{tab:71}
\end{table}

Let us now consider the case where $G \to P_{\rm tail}$ is nontrivial.
In this case $\coarse G$ is a rational curve and $\coarse{C_{\rm tail}}$ is its double cover.
The dual graph of the coarse space of $C = C_1 \cup C_{\rm tail} \cup C_2$ has the following form 
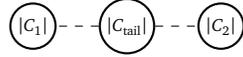
\[
\begin{tikzpicture}[xscale=1.25]
  \draw
  (-1,0) node[curve] (1) {$\coarse{C_1}$}
  (1, 0) node[curve] (2) {$\coarse{C_2}$}
  (0,0) node[curve] (t) {$\coarse{C_{\rm tail}}$}
  (t) edge[dashed] (1) edge[dashed] (2)
  ;
\end{tikzpicture}
\]
Again, the dashed lines represent either a node or two nodes with the same admissibility criterion as before.
The curve $C$ comes embedded in a surface $S$ fibered over $P$ with the intriguing feature that the piece $S_{\rm tail} \to P_{\rm tail}$ is a $\P^1 \cup \P^1$ bundle with non-trivial monodromy of the two components.
In this case, we do not know how to determine the parity of $(S, C)$.
But as far as the image of $\coarse C$ in $\overline{\orb M}_6$ is concerned, the question is moot by the following observation.
\begin{proposition}\label{thm:hard_parity_case_moot}
  Suppose $\phi \from \orb P \to \orb Z$ is a generic point in a boundary component of $\orb Q$ of type (7) such that the intermediate cover $G \to P_{\rm tail}$ is nontrivial.
  Then there exists a  $\phi' \from \orb P \to \orb Z$ in a boundary component of $\orb Q$ of type (7) of opposite parity which maps to the same point in $\overline{\orb M}_6$ as $\phi$. 
\end{proposition}
\begin{proof}
  Let $\psi \from \orb P \to \widetilde{\orb M}_{0,1+3}$ be the map induced by $\phi$ and $\orb D \to \orb P$ the associated triple cover.
  By \eqref{eqn:DEP}, we have $\orb D_{\rm tail} = \orb P_{\rm tail} \sqcup \orb E_\psi$.
  The data of $\phi$ gives a norm-trivial two-torsion line bundle $\orb L$ on $\orb D$.
  Let $x \in \orb P_{\rm tail} \subset \orb D_{\rm tail}$ be the point over the node $\orb P_{\rm tail} \cap \orb P_1$.
  Since $G \to P_{\rm tail}$ is non-trivial, by \autoref{thm:CGP_action} we get that $\Aut_x \orb D_{\rm tail}$ acts by $-1$ on $\orb L_x$.
  Let $\theta$ be the limiting theta characteristic on $\coarse{\orb D}$.
  By \autoref{thm:limit_theta}, the action of $\Aut_x \orb D_{\rm tail}$ on $\theta_x$ is also by $-1$.
  Note that the parity of $\phi$ is the parity of $h^0(\orb D, \theta \otimes \orb L)$.
  Let $\hat {\orb D} \to \orb D$ be the normalization at $x$ and let $x_1$, $x_2$ be the two points of $\hat{\orb D}$ over $x$.
  Let $\epsilon_x$ be the two-torsion line bundle on $\orb D$ obtained by taking the trivial line bundle on $\hat{\orb D}$ and gluing the fibers over $x_1$ and $x_2$ by $-1$.
  Let $\phi' \from \orb P \to \orb Z$ correspond to the same $\psi$ but the norm-trivial two-torsion line bundle $\orb L \otimes \epsilon_x$.
  By \eqref{eqn:theta-switch}, $\phi'$ has the opposite parity as $\phi$.
  The difference in the curve $C$ for $\phi$ and $\phi'$ is only in the manner of attaching $C_1$ to the rest of the curve.
  But on the level of coarse spaces, any choice leads to the same stable curve.
\end{proof}

We take all possible combinations of $C_1$, $C_2$, and $C_{\rm tail}$ and compute the stable images (see \autoref{tab:72}).
The combinations not shown in \autoref{tab:72} give loci of codimension higher than one.
The proof of \autoref{thm:7} is thus complete.
\begin{table}[hb]
  \centering
  \rowcolors{2}{gray!15}{white}
  \begin{tabular}{ccrc}
    \hline
    \rowcolor{gray!25}
    $C_1$ & $C_2$& $g\left(C_{\rm tail}\right)$& Divisor in \autoref{thm:7}\\
    \hline
    \ref{c1:32} or \ref{c1:42} & \ref{c2:odd4} ($p = 0$) & 2 & \ref{7:21}\\
    \ref{c1:32} or \ref{c1:42} & \ref{c2:odd4} ($p = 1$) & 3 & \ref{7:22}\\
    \ref{c1:32} or \ref{c1:42} & \ref{c2:odd4} ($p = 2$) & 4 & \ref{7:23}\\
    \ref{c1:32} or \ref{c1:42} & \ref{c2:even8} ($p = 1$) & 2 & \ref{7:32}\\     \ref{c1:32} or \ref{c1:42} & \ref{c2:even8} ($p = 2$) & 3 & \ref{7:33}\\
  \end{tabular}
  \caption{Divisors of type (7) with non-trivial $G \to P_{\rm tail}$}
  \label{tab:72}
\end{table}

\subsection{Divisors of type (8)}\label{sec:8}
\begin{proposition}\label{thm:8}
  There are $2$ irreducible components of $\overline Q \cap \Delta$ which are images of divisors of type (8) in $\orb Q_6^{\rm odd}$.
  Their generic points correspond to the following stable curves:

  \begin{itemize}
  \item With the dual graph
    \tikz{
      \draw 
      (0,0) node (X) [curve, label={180:$X$}] {} 
      (X) edge [loop right] (X);
    }
    \begin{enumerate}
    \item  $X$ hyperelliptic of genus $5$.
      \label{8:1}
    \end{enumerate}

  \item With the dual graph
    \tikz[baseline={(0,-0.1)}]{
      \draw 
      (0,0) node (X) [curve, label={180:$X$}] {}
      (1,0) node (Y) [curve, label={0:$Y$}] {}
      (X) edge [bend right=45] (Y)
      (X) edge (Y)
      (X) edge [bend left=45] (Y);
    }
    \begin{enumerate}[resume]
    \item $X$ hyperelliptic of genus $3$, $Y$ of genus $1$.
      \label{8:2}
    \end{enumerate}
  \end{itemize}
\end{proposition}

\begin{proof}
Recall that type (8) corresponds to $\phi \from \orb P \to \orb Z$ with the following dual graph.
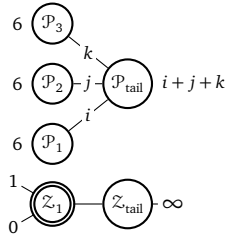
\[
\begin{tikzpicture}[baseline=(C2)]
  \draw
  (0,0) node[curve, double] (P1) {$\orb Z_1$}
  +(150:.6) node[int]  {$1$} edge (P1)
  +(210:.6) node[int]  {$0$}  edge (P1)
  (1,0) node[curve] (P2) {$\orb Z_{\rm tail}$}
  +(0:.6) node[int]  {$\infty$}  edge (P2);
  \begin{scope}[shift={(0,.8)}]
    \draw
    (0,0) node[curve, label={180:6}] (C1) {$\orb P_1$}
    (0,.8) node[curve, label={180:6}] (C2) {$\orb P_2$}
    (0,1.6) node[curve, label={180:6}] (C3) {$\orb P_3$}

    (1,.8) node[curve, label={0:$i+j+k$}] (E) {$\orb P_{\rm tail}$};
  \end{scope};
  \draw
  (C2) edge node[int] {$j$} (E)
  (C3) edge node[int] {$k$} (E)
  (C1) edge node[int] {$i$} (E)
  (P1) edge (P2);
\end{tikzpicture}
\]
We have already developed all the tools to analyze this case in \autoref{sec:7}.
The cover $C_{\rm tail} \to P_{\rm tail}$ factors as $C_{\rm tail} \to G \to P_{\rm tail}$.
Since $P_{\rm tail}$ has at most 3 orbifold points, the cover $G \to P_{\rm tail}$ is either trivial or non-trivial.
In the trivial case, $C_{\rm tail}$ is the disjoint union $C^1_{\rm tail} \sqcup C^2_{\rm tail}$ of two double covers of $P_{\rm tail}$.
In the non-trivial case, $\coarse{G}$ is a rational curve and $\coarse{C_{\rm tail}}$ is its double cover.
The curves $C_i$ given by $\orb P_i \to \orb Z_1$ for $i = 1,2,3$ are enumerated as \ref{c1:11}--\ref{c1:42} in \autoref{sec:7}.

We take all possible combinations of $C_1$, $C_2$, $C_3$, and $C_{\rm tail}$ and compute the stable images.
The case of trivial $G \to P_{\rm tail}$ gives two divisors.
Up to renumbering the subscripts, these arise from $C_1 = C_2 = \ref{c1:41}$ and $C_3 = \ref{c1:31}$ or $\ref{c1:41}$.
The first has $\coarse{C^1_{\rm tail}}$ of genus $-1$ and $\coarse{C^2_{\rm tail}}$ of genus 5 and it gives divisor~\eqref{8:1}.
The second has $\coarse{C^1_{\rm tail}}$ of genus 1 and $\coarse{C^2_{\rm tail}}$ of genus 3 and it gives divisor~\eqref{8:2}.
All other combinations give loci of codimension higher than one.

The case of non-trivial $G \to P_{\rm tail}$ gives one divisor.
By renumbering the subscripts if necessary, say that the non-trivial monodromy of $G \to P_{\rm tail}$ is at the node $x_1 = P_{\rm tail} \cap P_1$ and $x_2 = P_{\rm tail} \cap P_2$.
Then taking  $C_1$ and $C_2$ from $\{\text{\ref{c1:32}}, \text{\ref{c1:42}}\}$ and $C_3 = \text{\ref{c1:41}}$ gives divisor~\eqref{8:1}.
All other combinations give loci of codimension higher than one.
Note that the question of parity is moot in this case by the same argument as in \autoref{thm:hard_parity_case_moot}.
The proof of \autoref{thm:8} is now complete.
\end{proof}

\appendix
\section{Linear series on orbifold scrolls}\label{sec:orbiscrolls}
Let $\orb P$ be the orbifold curve $\P^1(\sqrt[r] 0)$, which has one orbifold point with stabilizer $\Z_r$ over $0$.
The goal of this section is to describe $\P^1$ bundles over ${\orb P}$, their coarse spaces, and linear series on them.
We recall the following standard facts about ${\orb P}$ (see \cite[Section~2]{mar.tha:12}).
\begin{proposition}Let ${\orb P} = \P^1(\sqrt[r] 0)$.
  \begin{enumerate}
  \item Every $\P^1$-bundle over ${\orb P}$ is the projectivization of a rank two vector bundle.
  \item Every vector bundle on ${\orb P}$ is the direct sum of line bundles.
  \item The line bundles on ${\orb P}$ are of the form $\O_{\orb P}(a)$ for $a \in \frac{1}{r} \Z$, where $\O_{\orb P}(1/r)$ refers to the dual of the ideal sheaf of the unique (reduced) orbifold point on ${\orb P}$.
  \end{enumerate}
\end{proposition}
Let $c \from {\orb P} \to \P^1$ be the coarse space map.
Note that $c^*\O_{\P^1}(a) = \O_{\orb P}(a)$ for $a \in \Z$ and $c_* \O_{\orb P}(a) = \O_{\P^1}(\lfloor a \rfloor)$ for $a \in \frac1r \Z$.
For $a > 0$ in $\frac{1}{r} \Z$, set $\F_a = \proj(\O_{\orb P} \oplus \O_{\orb P}(-a))$.
The tautological line bundle $\O_{\F_a}(1)$ on $\F_a$ has a unique section.
We denote its zero locus by $\sigma$ and call it the \emph{directrix}.
It is the unique section of $\F_a \to {\orb P}$ with negative self intersection $\sigma^2 = -a$.
It corresponds to the projection $\O_{\orb P} \oplus \O_{\orb P}(-a) \to \O_{\orb P}(-a)$.
There are sections $\tau$ disjoint from $\sigma$ corresponding to projections $\O_{\orb P} \oplus \O_{\orb P}(-a) \to \O_{\orb P}$.
These $\tau$ lie in the divisor class $\sigma + a F$, where $F$ is the pullback of $\O_{\orb P}(1)$.
Observe that if $a$ is not an integer, then $\tau(0)$ is independent of the choice of $\tau$.
We call $\tau$ a \emph{co-directrix}.

\begin{proposition}\label{prop:orbisingularities}
  Retain the notation introduced above.
  If $a \in \Z$, then $\coarse{\F_a}$ is smooth and $\coarse{\F_a} \to \P^1$ is the $\P^1$-bundle $\proj (\O_{\P^1} \oplus \O_{\P^1}(-a))$.
  If $a \not \in \Z$, then $\coarse{\F_a}$ is smooth except at the two points $\sigma(0)$ and $\tau(0)$.
  At $\tau(0)$, it has the singularity $\frac{1}{r}(1, ra)$.
  At $\sigma(0)$, it has the singularity $\frac{1}{r}(1, r-ra)$.
  Furthermore, the scheme theoretic fiber of $\coarse{\F_a} \to \P^1$ over $0$ has multiplicity $r/\gcd(r,ra)$.
\end{proposition}
\begin{proof}
  Fix a generator $\zeta \in \mu_r$.
  In local coordinates around $0$, we can write ${\orb P}$ as 
  \[ [\spec \C[x] / \mu_r],\]
  where $\zeta$ acts by $x \mapsto \zeta x$.
  In these coordinates, we can trivialize $\O_{\orb P} \oplus \O_{\orb P}(-a)$ as a $\mu_r$ equivariant vector bundle with basis $\langle X, Y \rangle$ on which $\zeta$ acts by $X \mapsto X$ and $Y \mapsto \zeta^{ra} Y$.
  We think of $X$ and $Y$ as homogeneous coordinates on the projectivization.
  Then $\sigma$ corresponds to $X = 0$ and $\tau$ to $Y = 0$.
  Locally around $\sigma(0)$ we can write $\F_a$ as
  \[ [ \spec \C[x, X/Y] / \mu_r], \text{ where } \zeta \cdot (x, X/Y) = (\zeta x, \zeta^{r-ra} X/Y).\]
  Similarly, around $\tau(0)$ we can write $\F_a$ as
  \[ [ \spec \C[x, Y/X] / \mu_r], \text{ where } \zeta \cdot (x, Y/X) = (\zeta x, \zeta^{ra} Y/X).\]
  The claims about the singularities follow from these presentations.

  In either chart, invert the second coordinate, and let $m \in \Z$ be such that $r$ divides $mra + \gcd(r,ra)$.
  Then the invariant ring is generated by $u = x^{\gcd(r,ra)}(X/Y)^{-m}$.
  On the other hand, the invariant ring in $\C[x]$ is generated by $v = x^r$.
  Up to an invertible function, the preimage of $v$ is $u^{r/\gcd(r,ra)}$.
  The claim about the multiplicity follows.
\end{proof}

We now turn to linear systems on $\F_a$.
Let $\pi \from \F_a \to {\orb P}$ be the projection.
\begin{proposition}\label{prop:a bound}
  Let $\orb C \subset \F_a$ be a member of $|n \sigma + m F|$.
  Then $\deg \omega_{\orb C/ \orb P} = (n-1)(2m-an)$.
  If $\orb C$ does not pass through $\sigma(0)$, then $m-na$ is a non-negative integer.
  If $\orb C$ is \'etale over $0$, then at least one of $m-na$ or $m-(n-1)a$ is a non-negative integer.
  If $\orb C$ is smooth, then $m-na \geq 0$ or $m - na = -a$.
  In the former case, $\orb C$ is connected.
  In the latter case, $\orb C$ is the disjoint union of $\sigma$ and a curve in $|(n-1) \tau|$.
\end{proposition}
\begin{proof}
  We have $\omega_{\F_a/\orb P} = -2\sigma - aF$.
  By adjunction, $\omega_{\orb C/\orb P} = (n-2) \sigma + (m+a)F$.
  Hence
  \[\deg \omega_{\orb C/\orb P} = ((n-2) \sigma + (m-a) F)(n \sigma + m F) = (n-1)(2m-an).\]

  For the next two statements, expand a global section $s$ of $\pi_* \O(n\sigma+mF)$ locally around $0$ as a homogeneous polynomial of degree $n$ in local coordinates $X \oplus Y$ for  $\O \oplus \O(-a)$.
  Say
  \[ s = p_0 X^n + p_1 X^{n-1}Y + \dots + p_{n-1} XY^{n-1} + p_n Y^n,\]
  where $p_i$ is the restriction of a global section of $\O(m-ia)$.
  For $\orb C$ to not pass through $\sigma(0)$, $p_n$ must not vanish at $0$.
  For the zero locus of $s$ to be \'etale over $0$, at least one of $p_n$ or $p_{n-1}$ must not vanish at $0$.
  But $\O(m-ia)$ has a section not vanishing at $0$ if and only if $m-ia$ is a nonnegative integer.

  For the next statements, note that $\orb C \cdot \sigma = m-na$.
  If $\orb C$ is smooth and $m-na < 0$, then $\orb C$ must contain $\sigma$ and have $\sigma \cdot (\orb C \setminus \sigma) = 0$.
  This forces $\orb C$ to be the disjoint union of $\sigma$ and a curve in $|(n-1) \tau|$.
  If $m - na \geq 0$, then we see that $h^0(\orb C, \O_{\orb C}) = 1$, which implies that $\orb C$ is connected.
\end{proof}
\begin{corollary}\label{prop:orbiscroll_4}
  Let $\orb C \subset \F_a$ be a curve in the linear system $4 \sigma + m F$ such that the degree of the ramification divisor of $\orb C \to \orb P$ is $b$.
  Then $m = b/6+2a$.
  If $\orb C$ does not pass through $\sigma(0)$, then $a \leq b/12$.
  If $\orb C$ is \'etale over $0$, then $a \leq b/6$.
  If $\orb C$ is smooth, then either $a \leq b/12$ or $a = b/6$.
\end{corollary}

\def\cprime{$'$}
\providecommand{\bysame}{\leavevmode\hbox to3em{\hrulefill}\thinspace}
\providecommand{\MR}{\relax\ifhmode\unskip\space\fi MR }
\providecommand{\MRhref}[2]{  \href{http://www.ams.org/mathscinet-getitem?mr=#1}{#2}
}
\providecommand{\href}[2]{#2}

\bibliographystyle{amsplain}

\end{document}